\documentclass[11pt,a4paper]{article}
 \usepackage{indentfirst, latexsym,bm}
 \usepackage{amsmath}
 \usepackage{pifont}
 \usepackage{amsfonts}
 \usepackage{mathrsfs}
 \usepackage{array}
 \usepackage{multirow} 
 \usepackage{graphicx}
 \usepackage{subfigure}
 \usepackage{picinpar}
 \usepackage{setspace}
 \usepackage[textsize=footnotesize]{todonotes}
 \RequirePackage[colorlinks,citecolor=blue,urlcolor=blue,linkcolor=blue]{hyperref}
 \usepackage[top=2cm,bottom=2cm, outer=2cm, inner=2cm]{geometry}
 
 \usepackage{caption}
 \usepackage{subcaption}
 \usepackage[labelfont={bf,small},textfont={small}]{caption}
 
 \RequirePackage[numbers]{natbib}
 \usepackage{enumitem} % math notations i.e. \mathscr
 
 \usepackage{booktabs}
 \usepackage{authblk}
 
 \usepackage{threeparttable}
 \usepackage{mathrsfs,amsfonts,amsmath}
 
 \usepackage{color}
 
 \makeatletter
 \def\namedlabel#1#2{\begingroup
 	#2%
 	\def\@currentlabel{#2}%
 	\phantomsection\label{#1}\endgroup
 }
 \makeatother

 \numberwithin{figure}{section}

 \newcommand\email[1]{\href{mailto:#1}{ \nolinkurl{#1}}}

 \newtheorem{theorem}{Theorem}[section]
 \newtheorem{definition}[theorem]{Definition}
 \newtheorem{lemma}[theorem]{Lemma}
 \newtheorem{corollary}[theorem]{Corollary}
 \newtheorem{proposition}[theorem]{Proposition}
 \newtheorem{remark}[theorem]{Remark}
 \newtheorem{condition}[theorem]{Condition}
 \newtheorem{example}[theorem]{Example}

 \def\blemma{\begin{lemma}}\def\elemma{\end{lemma}}
 \def\bproposition{\begin{proposition}}\def\eproposition{\end{proposition}}
 \def\ttheorem{\begin{theorem}}\def\etheorem{\end{theorem}}
 \def\bcorollary{\begin{corollary}}\def\ecorollary{\end{corollary}}
 \def\bremark{\begin{remark}}\def\eremark{\end{remark}}
 \def\bcondition{\begin{condition}}\def\econdition{\end{condition}}

 \def\benumerate{\begin{enumerate}}\def\eenumerate{\end{enumerate}}
 \def\bitemize{\begin{itemize}}\def\eitemize{\end{itemize}}

 \def\beqlb{\begin{eqnarray}}\def\eeqlb{\end{eqnarray}}
 \def\beqnn{\begin{eqnarray*}}\def\eeqnn{\end{eqnarray*}}
 \def\ar{\!\!\!&}

 \def\proof{\noindent{\it Proof.~~}}\def\qed{\hfill$\Box$\medskip}

\begin{document} 

  \title{\bf  \Large Local Asymptotics for Entrance Laws of  Reflected L\'evy Excursions}
 
 \author{Zhi-Hao Cui\footnote{School of Mathematical Sciences, Nankai University, China; email: cuizh.math@gmail.com} \quad  \  Hao Wu\footnote{School of Mathematical Sciences, Nankai University, China; email: wuhao.math@outlook.com} 
 	\quad\ and\quad 
 	Wei Xu\footnote{School of Mathematics and Statistics, Beijing Institute of Technology,  China; email: xuwei.math@gmail.com. Xu gratefully acknowledges financial support from the Fundamental Research Funds for the Central Universities, the National Natural Science Foundation of China (No. 11531001) and the National Key R\&D Program of China (No. 2023YFA1010103).} 
 }  
 \maketitle

 \begin{abstract}
 This paper establishes uniform local large-deviation asymptotics for the entrance law $\underline n_t(dx)$ of excursions of a L\'evy process with negative drift and locally regularly varying L\'evy measure $\nu$, reflected at its running infimum. 
 In the fixed-time regime, we prove that as $x\to\infty$,  
 $$
 \underline n_t\big((x,x+\delta]\big)
 \sim \underline n(\zeta\wedge t)\, \nu\big((x,x+\delta]\big),
 $$ 
 uniformly in  $\delta\in[\delta_0,\infty]$ for every $\delta_0>0$. 
 Our principal result concerns the large-time regime and shows that, as $t\to\infty$,  
 $$ 
 \underline n_t\big((x,x+\delta]\big)
  \sim \int_x^{x+\delta} V(z) \,\nu(\beta t+dz), 
$$
 uniformly in $x\geq0$ and $\delta\in[\delta_0,\infty]$, where $V$ denotes the renewal function. 
 Using a fluctuation-theoretic decomposition of the killed semigroup, we further derive corresponding uniform local large-deviation asymptotics for the L\'evy process killed upon its first passage below zero. %Our results provide a unified one-big-jump description of local excursion asymptotics across both large-space and large-time regimes. 
 
 	\bigskip
 	
 	
 	\smallskip
 	

 \end{abstract} 
 
% \tableofcontents

  \section{Introduction and main results} 
 \label{Sec.Introduction}
 \setcounter{equation}{0}

 Consider a one-dimensional  L\'evy process $X=\{ X_t:t\geq 0 \}$ with negative drift $ \beta:=-\mathbf{E}\big[X_1\big]\in (0,\infty)$, diffusion coefficient $\sigma\geq 0$ and L\'evy measure $\nu(dy)$ satisfying
 \beqlb\label{eqn.LevyMeasure}
 \int_{\mathbb{R}} \big( |y|\wedge |y|^2\big)\,\nu(dy)<\infty. 
 \eeqlb 
 For $x\in \mathbb{R}$, we write $\tau_x^-$ and $\tau_x^+$ for the first passage times of $X$ into $(-\infty,x)$ and $(x,\infty)$, respectively. 
 Let $\underline{X}$ and $ \overline{X}$ be the \textit{running infimum} and  \textit{supremum processes} by
 \beqnn
 \underline{X}_t:= \inf\big\{X_s: 0\leq s\leq t\big\}
 \quad \mbox{and}\quad 
 \overline{X}_t:= \sup\big\{X_s: 0\leq s\leq t\big\},\quad t\geq0. 
 \eeqnn  
 The reflected processes $X-\underline{X}$ and $\overline{X}-X$ are  non-negative strong Markov processes.
 Their excursions away from $0$ form two Poisson point processes on the space $\mathcal{E}$ of excursions, with respective intensities $\underline{n}(d\epsilon)$ and $\overline{n}(d\epsilon)$.  
 Here, each excursion $\epsilon \in \mathcal{E}$ is a non-negative c\`adl\`ag function on $\mathbb{R}_+$ with \textit{lifetime} 
 \beqnn
 \zeta=\inf\{t\geq0: \epsilon_s=0, \forall s\geq t \},\quad \inf \emptyset =+\infty. 
 \eeqnn  

 In our earlier work \cite{CuiWuXu2025}, we established the asymptotic behavior of the excursion-lifetime tail $\underline{n}(\zeta>t)$ under a suitable heavy-tail assumption. We also showed that long-lived excursions converge, under spatial rescaling, to a single-jump process and, under a suitable time rescaling, to a linearly decreasing path killed upon hitting zero. 
 These results, however, do not describe the local behavior of long-lived excursions such as their entrance laws
 \beqlb
  \underline{n}_t(dx):= \underline{n}\big( \epsilon_t\in dx,\,\zeta>t \big),\quad t>0,\, x>0.
 \eeqlb

 The aim of this work is to provide a unified treatment of the uniform local large-deviation asymptotics for $\underline{n}_t(dx)$: for $x>0$ and $\delta\in(0,\infty]$, 
 \beqlb\label{eqn.003}
 \underline{n}_t\big(x+\varDelta_\delta\big) =\underline{n}\big( \epsilon_t\in x+\varDelta_\delta,\,\zeta>t \big)
 \quad \mbox{with}\quad  x+ \varDelta_\delta:= (x,x+\delta] . 
 \eeqlb
 These asymptotics play a central role in the study of both the excursion measure $\underline{n}(d\epsilon)$ and the process $X$. For instance, the canonical process under $\underline{n}(d\epsilon)$ has the transition semigroup of $X$ killed upon its first passage below zero, i.e.,  for any $0<t<r<\infty$,
 \beqnn
 \underline{n}\big( \epsilon_t\in dx,\, \epsilon_r\in dy,\, \zeta>r \big)
 = \underline{n}_t(dx)\, \mathbf{P}_x\big( X_{r-t}\in dy,\, \tau_0^->r-t \big).
 \eeqnn
 Furthermore,  
 Theorem~5 of \cite{Chaumont2013} gives, for each $y>0$, the following identity of measures on $u>0$: \smallskip
 \beqlb\label{eqn2.9}
 \lefteqn{\mathbf P_y\big( X_t\in du,\, \tau_0^->t \big)}\cr
 \ar=\ar \overline{\tt d}\cdot \underline n_t(du-y)\cdot
 \mathbf 1_{\{u\geq y\}} +\underline{\tt d}\cdot \overline{n}_t(y-du)\cdot
 \mathbf{1}_{\{u\leq y\}}  + \int_0^t ds \int_{(y-u)^+}^y  \underline n_{t-s}(du+z-y) \, \overline n_s(dz) .
 \eeqlb 
 
 To the best of our knowledge, the only previous study of local large-deviation asymptotics for $\underline{n}_t(dx)$ is that of Doney and Rivero \cite{DoneyRivero2013,DoneyRivero2016Erratum}. They considered the case in which $X$ belongs to the domain of attraction of a stable law without centering, with norming function $c$. In particular, for any $0<\delta_0<\delta_1$, they showed that, as $t\to\infty$, uniformly over $\delta\in[\delta_0,\delta_1]$ and $x> 0$ with $x/c(t)\to0$,
 \beqnn 
  \underline{n}_t\big(x+\varDelta_\delta\big)
  \sim
  \frac{f(0)}{ t\cdot c(t)}\int_x^{x+\delta}V(z)\,dz,
 \eeqnn
 where $f$ is the density of the limiting stable law and  $V$ is the \textsl{renewal function} of the ascending ladder height process of $X$.
 For a nonlattice random walk $S=\{S_n:n\geq 0 \}$ in the domain of attraction of a stable law without centering. %let $T_0^-:=\inf\{n\geq1:S_n\leq0\}$.
 Doney \cite{Doney2012} provided a related conditional local limit theorem that states, for each fixed $\delta>0$, as $n\to\infty$,  
 \beqnn
 c_n\cdot \mathbf P\big(S_n\in x+\varDelta_\delta\,\big|\,\tau^-_0>n\big)-\delta\,g\big(x/c_n\big) \to 0, 
 \eeqnn uniformly over $x>0$, where $g$ is the density of the stable meander at time one. 
	When $\mathbf{E}[S_1]=-a<0$, Iglehart \cite{Iglehart1974} established, under an exponential moment condition, convergence in distribution of $S_n$ conditional on $\tau_0^->n$. Keener \cite{Keener1992} subsequently identified the limit under related assumptions. 
	Under suitable moment and regular variation assumptions on the density $b$ of $S_1$, Denisov et al.~\cite{DenisovVatutinWachtel2014} proved that, for each fixed $\delta>0$ and every sequence $x_n\to\infty$, as $n\to\infty$, uniformly over $x\geq x_n$, 
	\beqnn
	\frac{ \mathbf{P}\big( S_n\in x+\varDelta_\delta,\, \tau_0^->n \big)}
	{b(an+x)}
	\to \delta\cdot  \mathbf{E}\big[\tau_0^-\big].
	\eeqnn

  \subsection{Main results}
  
  We now formulate our results on the uniform asymptotic behaviors of  $\underline{n}_t\big(x+\varDelta_\delta\big) $  for fixed $t$ or as $t\to\infty$.  
  Throughout, we write $f\sim g$ whenever $f(x)/g(x)\to1$ as $x\to\infty$.
  For $\kappa\in\mathbb{R}$, a function $f$ is said to be regularly varying at infinity with index $\kappa$, written $f\in\mathrm{RV}^\infty_\kappa$, if  $f(\lambda x)/f(x)\to \lambda^\kappa$ as $x\to\infty$ for every $\lambda>0$.

  If the state $0$ is irregular for $(0,\infty)$, for $X$, then $\underline{n}(\epsilon_0>0)>0$,  and under the conditional measure $\underline n(\,\cdot\,|\,\epsilon_0>0)$, the canonical process has the same law as $X$ started from an initial position with distribution $\underline{n}(\epsilon_0\in dx\,|\,\epsilon_0>0)$. 
  Consequently, the asymptotic behavior of $\underline{n}_t(dx, \, \epsilon_0>0)$ can be derived directly from the corresponding asymptotics for the L\'evy process started from a random positive initial state; see \cite{CuiWuXu2026c}.
  In the present paper, we therefore restrict attention to the complementary case and impose the standing assumption that
  \begin{center}
  \bf the state $0$ is regular for $(0,\infty)$, for $X$.
  \end{center}

 Recall that, for $x>0$ and $\delta\in(0,\infty]$, we write $ x+ \varDelta_\delta= (x,x+\delta]$ with the convention $x+ \varDelta_\infty= (x,\infty)$. In particular, $\varDelta_\delta=(0,\delta]$. 
 For $x>0$, let $\overline{\nu}(x):= \nu\big((x,\infty)\big)$ be the tail-distribution of L\'evy measure $\nu(dy)$. 
 All our results are established under the following local regular-variation assumption on the L\'evy measure $\nu(dy)$. 
 
  \begin{condition}\label{Main.Condition}
  	For some $\alpha>1$, assume that $x\mapsto\nu\big(x+\varDelta_\delta\big)
  \in \mathrm{RV}^\infty_{-\alpha-1}$ for any 
  $\delta>0$. 
  \end{condition}
 
  \begin{remark} \label{Remark01}
  	By Proposition~3.3 in \cite{WuXu2026a}, Condition~\ref{Main.Condition} is equivalent to the following: for every $\delta>0$,
  	\beqlb\label{eqn.30}
  	\overline{\nu}\in\mathrm{RV}_{-\alpha}^{\infty}
  	\quad \mbox{and} \quad
  	\nu\big(x+\varDelta_\delta\big)
  	\sim
  	\alpha\cdot\frac{\overline{\nu}(x)}{x}\cdot\delta,
  	\quad \mbox{as } x\to\infty.
  	\eeqlb
 \end{remark}
 
 The regular variation of $\overline\nu$ describes the frequency of jumps exceeding a large level $x$, while the second relation in \eqref{eqn.30} determines how the L\'evy mass is distributed over bounded intervals near $x$. In particular, $\nu(x+\varDelta_\delta)/\delta \sim \alpha \cdot \overline\nu(x)/x$ for every fixed $\delta>0$,  so that the L\'evy measure is asymptotically locally smooth in its positive tail. This local information is essential in our setting, since the asymptotics of  $\underline{n}_t\big(x+\varDelta_\delta\big) $ depend on the probability of a large jump landing in a prescribed spatial window, rather than merely on the probability of exceeding $x$.
 
 Our first result provides  uniform local large-deviation asymptotics for
 $\underline{n}_t\big(x+\varDelta_\delta\big) $ at any fixed time $t$. 
 It states that an excursion observed at a very large level at a fixed time $t$ is  produced by a unique exceptionally large positive jump occurring while the excursion is alive. The local position of the excursion is determined by the Lévy measure of this jump, while the factor $\underline n(\zeta\wedge t)$ accounts for all possible times before $t$ at which the jump may occur.
  
   \begin{theorem}\label{MainThm00}
  	For any $t,\delta_0>0$, we have as $x\to\infty$, uniformly in $\delta\in [\delta_0,\infty]$, 
  	\beqlb\label{eqn.MainThm00}
  	\underline{n}_t\big(x+\varDelta_\delta\big) 
  	\sim \underline{n}(\zeta \wedge t) \cdot \nu\big( x+ \varDelta_\delta \big) . 
  	\eeqlb
  \end{theorem}
  
  Since $0$ is regular for $(0,\infty)$, for $X$, we have $\underline{\tt d}=0$. 
   For any $t,\delta,x,y>0$, integrating both sides of \eqref{eqn2.9} with respect to $u$ over $x+\varDelta_\delta$ and then applying Fubini's theorem yields that 
  \beqlb\label{eqn.30111}
  \mathbf P_y\big(
  X_t\in x+\varDelta_\delta,\,
  \tau_0^->t
  \big)
  \ar=\ar
  \overline{\tt d}\cdot
  \underline n\Big(
  \epsilon_t\in
  \big((x-y)^+,(x-y+\delta)^+\big],\,
  \zeta>t
  \Big)
  \cr\cr
  \ar\ar \quad
  +\int_0^t ds\int_0^y
  \underline n_{t-s}\Big(
  \big((x-y+z)^+,(x-y+z+\delta)^+\big]
  \Big)\,
  \overline n_s(dz). \quad
  \eeqlb 
  This decomposition allows us to apply the preceding theorem to derive the corresponding local asymptotic for the L\'evy process killed upon entering the negative half-line.
   
  \begin{corollary}\label{Main.Corollary01}
  	For any fixed $y>0$ and $t,\delta_0>0$, we have as $x\to\infty$, uniformly in
  	$\delta\in[\delta_0,\infty]$, 
  	\beqnn
  	\mathbf{P}_y\big( X_t\in x+\varDelta_\delta,\tau_0^->t \big)
  	\sim
  	\mathbf{E}_y\big[ t\wedge\tau_0^- \big]\cdot \nu\big(x+\varDelta_\delta\big).
  	\eeqnn
  \end{corollary}

  Our second main theorem provides a unified description of the uniform local large-deviation asymptotics of the entrance law in the large-time regime. 
  It shows that, at large times, the dominant contribution comes from a single large positive jump whose size is sufficient to compensate for the negative drift accumulated over a time interval of order $t$. 
  The result is uniform over the entire range of terminal spatial levels and over all intervals whose lengths are bounded away from zero. 
  Recall that $V$ denotes the renewal function of the ascending ladder height process of $X$.
 
 \begin{theorem}\label{MainThm01}
 	For any $\delta_0>0$, we have as $t\to\infty$, uniformly in $\delta\in [\delta_0,\infty]$ and $x\geq 0$,
 	\beqlb\label{eqn.MainThm01}
 	\underline{n}_t\big(x+\varDelta_\delta\big) 
 	\sim
  \alpha \int_{x}^{x+\delta} \frac{\overline{\nu}(\beta t+z)}{\beta t+z} \cdot V(z)\, dz
  \sim \int_x^{x+\delta} V(z) \,\nu(\beta t+dz). 
 	\eeqlb
 	\end{theorem}

 In particular, given any positive, non-decreasing function  $\theta$ on $\mathbb R_+$ such that $\theta_t/t\to 0$ as $t\to\infty$, the regular-variation of $\overline\nu$ yields that
 \beqnn
 \underline{n}_t\big(x+\varDelta_\delta\big) 
 \sim
 \alpha\cdot \frac{\overline{\nu}(\beta t+x)}{\beta t+x}
 \cdot \int_{x}^{x+\delta}V(z)dz,\quad \mbox{as $t\to\infty$},
 \eeqnn 	
 uniformly in $\delta \in [\theta _0,\theta _t]$ and $x\geq 0$. 
 This local asymptotic relation further suggests the  vague convergence of rescaled measures   $\frac{\beta t}{\overline{\nu}(\beta t)} \cdot \underline{n}_t\big(dz)$  to the $\sigma$-finite measure $\alpha \cdot V(z)\,dz$.  
 More precisely, we have the following consequence.

  \begin{corollary} \label{Corollary.101}
 For any $0\leq x<y<\infty$ and any continuous function $f$ on $\mathbb R_+$, we have  
 	\beqnn
 	\lim_{t\to\infty} \frac{\beta t}{\overline{\nu}(\beta t)}\int_x^y f(z)\, \underline{n}_t\big(dz) 
 	= \alpha \int_x^y f(z)\,V(z)\,dz . 
 	\eeqnn 
 \end{corollary}

 As our final main result, we establish the corresponding uniform local large-deviation asymptotics for the L\'evy process killed upon its first entrance into the negative half-line. This follows by combining Theorem~\ref{MainThm01} with the decomposition in \eqref{eqn.30111}.  
 
  \begin{corollary}\label{Main.Corollary02}
  	For any $y,\delta_0>0$, we have as $t\to\infty$, uniformly in $\delta\in [\delta_0,\infty]$ and $x\geq 0$,
  	\beqlb\label{eqn.Corollary02}
  	\mathbf{P}_y\big( X_t\in x+\varDelta_\delta,\, \tau_0^->t \big) 
  	\sim \alpha \cdot \widehat{V}(y)\cdot \int_x^{x+\delta}  \frac{\overline{\nu}(\beta t+z)}{\beta t+z} \cdot V(z)\, dz. 
  	\eeqlb
  	
  	\end{corollary}

   \smallskip
  {\it \textbf{Organization of this paper.}} In Section~\ref{Sec.Preliminaries}, we recall the basic fluctuation theory for L\'evy processes and also  present several auxiliary lemmas that are used frequently in the proofs of our main results. Section~\ref{Sec.Proof01} is devoted to proving Theorem~\ref{MainThm00} and Corollary~\ref{Main.Corollary01}. 
  In Section~\ref{Sec.Proof02}, we first provide  detailed proofs of Theorem~\ref{MainThm01} and Corollary~\ref{Corollary.101}, then prove  Corollary~\ref{Main.Corollary02} by combining the path decomposition \eqref{eqn2.9} with Theorem~\ref{MainThm01}.

  \section{Preliminaries} 
 \label{Sec.Preliminaries}
 \setcounter{equation}{0}

  %In this section, we first introduce some further notation and recall some elementary fluctuation theory of L\'evy processes. We refer to the monographs \cite{Bertoin1996,Doney2007,Kyprianou2014}  for more details. 
  %After that, we recall and establish several auxiliary lemmas that will play an important role in the proofs of our main results. 
  
  This section begins by introducing additional notation and recalling basic fluctuation theory for Lévy processes, for which we refer to \cite{Bertoin1996,Doney2007,Kyprianou2014}. We then present several auxiliary lemmas that will be essential for proving our main results.

  Let $D\big(\mathbb{R}_+;\mathbb{R}\big)$ be the space of all c\`adl\`ag functions on $\mathbb{R}_+$ endowed with the Skorokhod topology. 
 We assume all processes in this work are defined on a complete filtered probability space $(\Omega,\mathscr{F},\mathscr{F}_t,\mathbf{P})$ satisfying the usual hypotheses. 
 For $x\in\mathbb{R}$, let $\mathbf{P}_x$ and $\mathbf{E}_x$ denote the law and expectation of the  process starting from $x$. 
 For simplicity, we also write $\mathbf{P}= \mathbf{P}_0$ and $\mathbf{E}=\mathbf{E}_0$. 
  Let   $\overset{\rm a.s.}\to$, $\overset{\rm d}\to$  and  $\overset{\rm p}\to$ be the  almost sure convergence,  convergence in distribution and convergence in probability respectively.
 We also use $\overset{\rm a.s.}=$, $\overset{\rm d}=$ and $\overset{\rm p}=$ to denote almost sure equality, equality in distribution and equality in probability respectively.

 \subsection{Fluctuation theory}

 Let $\Delta X_t := X_t-X_{t-}$ denote the jump size of $X$ at time $t  > 0$. 
 The point measure of all jumps 
 \beqnn
 \sum_{t>0} \mathbf{1}_{\{ \Delta X_t\neq 0 \}}\cdot \delta_{(t,\Delta X_t)}(ds,dy),
 \eeqnn
 is a Poisson random measure on $(0,\infty)\times \mathbb{R}$ with intensity $ds\, \nu(dy)$.  
 By the compensation formula for Poisson random measures (see \cite[p.7]{Bertoin1996}), we have that for any $(\mathscr{F}_t)$-predictable process $\mathbf{F}=\{ \mathbf{F}_t:t\geq 0 \}$ taking values in the space of non-negative measurable functions on $\mathbb{R}$, 
  \beqlb\label{eqn.CompenstationX}
  \mathbf{E}\bigg[ \sum_{t>0} \mathbf{F}_t(\Delta X_t) \cdot \mathbf{1}_{\{ \Delta X_t\neq 0 \}}  \bigg]= \mathbf{E}\bigg[\int_0^\infty dt \int_{\mathbb{R}\setminus\{0\}}\mathbf             {F}_t(y)\, \nu(dy)\bigg].
  \eeqlb
  
 The reflected process $X-\underline{X}$ is a Markov process with Feller transition semigroup given by $\mathbf{P}_x(X_t\in dy, \underline{X}_t>0)$ for any $x,y>0$; see, e.g., Proposition~1 in \cite[p.156]{Bertoin1996}. 
 Its excursions away from $0$ form a Poisson point process on $\mathcal{E}$ with intensity $\underline{n}(d\epsilon)$.  
 Under $\underline{n}(d\epsilon)$, the excursion process is Markovian and shares the same semigroup of $X$ killed upon entering $(-\infty,0]$, that is, for any $t,s>0$, any measurable functionals $F$ on $D\big([0,t];\mathbb{R}\big)$ and  $G$ on $D\big([0,s];\mathbb{R}\big)$, 
 \beqlb\label{eqn.n-Markov}
 \lefteqn{\underline{n}\Big(F\big(\epsilon_r, r \in[0,t]\big) \cdot G\big(\epsilon_{r+t}, r\in[0,s]\big),\, \zeta>t+s\Big)}\ar\ar\cr
 \ar\ar\cr
 \ar=\ar \underline{n}\Big(F\big(\epsilon_r, r \leq t\big)\cdot \mathbf{E}_{\epsilon_t}\big[G\big(X_r, r\in[0,s]\big),\,  \underline{X}_s>0\big],\, \zeta>t\Big)  .
 \eeqlb
 An important consequence of \eqref{eqn.CompenstationX} and \eqref{eqn.n-Markov} is that the compensation formula also holds under $\underline{n}(d\epsilon)$, i.e., let $\mathscr{G}_t:=\sigma\big(\{\epsilon_s:s\in [0,t], \epsilon \in \mathcal{E}\}\big)$ for any $t\geq 0$, for any $(\mathscr{G}_t)$-predictable process $\mathbf{F}=\{ \mathbf{F}_t:t\geq 0 \}$ taking values in the space of non-negative measurable functions on $\mathbb{R}$, we have 
 \beqlb\label{eqn.Compensation}
 \underline{n}\bigg(\sum_{0<s\leq \zeta} \mathbf{F}_s(\Delta \epsilon_s) \cdot 1_{\{\Delta \epsilon_s \neq 0\}}\bigg)
 = \underline{n}\bigg(\int_0^\zeta ds \int_{\mathbb{R}\setminus\{0\}} \mathbf{F}_s(y)\,\nu(dy)\bigg).
 \eeqlb

 \begin{proposition}\label{Prop.203}
  For any $t> 0$ and $x,y> 0$, we always have 
  \beqnn
 	\underline{n}\bigg(  \sum_{0<s<r\leq t\wedge \zeta} \mathbf{1}_{\{\Delta \epsilon_s > x ,\, \Delta \epsilon_r > y \}} \bigg) 
 	\leq t \cdot \underline{n}(\zeta ) \cdot \overline{\nu}(x)  \cdot \overline{\nu}(y).   
  \eeqnn
 \end{proposition}
 \proof  By  the Markov property \eqref{eqn.n-Markov}, we first have 
 \beqnn
 \underline{n}\bigg(  \sum_{0<s<r\leq t\wedge \zeta} \mathbf{1}_{\{\Delta \epsilon_s > x ,\, \Delta \epsilon_r > y \}} \bigg) 
 = \underline{n}\bigg( \sum_{0<s\leq t\wedge \zeta} \mathbf{1}_{\{\Delta \epsilon_s > x  \}} \cdot \mathbf{E}_{\epsilon_{s-} +\Delta\epsilon_s} \Big[ \sum_{0<r\leq t-s} \mathbf{1}_{\{ \tau_0^->r,\,\Delta X_r > y  \}}\Big] \bigg).
 \eeqnn
 Using the compensation formulas \eqref{eqn.Compensation} and  \eqref{eqn.CompenstationX} one-by-one yields that  the right-hand side equals
 \beqnn
 \lefteqn{ \int_0^t ds \int_0^\infty \underline{n}_s(dz) \int_x^\infty  \mathbf{E}_{z+w} \bigg[ \sum_{0<r\leq t-s} \mathbf{1}_{\{ \tau_0^->r,\,\Delta X_r > y \}}\bigg]  \, \nu(dw)   } \cr
 \ar=\ar \int_0^t ds \int_0^\infty \underline{n}_s(dz) \int_x^\infty \nu(dw)  \int_0^{t-s}  \mathbf{P}_{z+w} \big(   \tau_0^->r \big)   \, dr  \cdot \overline{\nu}(y)
 \leq t\cdot \int_0^t   \underline{n}(\zeta>s) \, ds   \cdot \overline{\nu}(x)  \cdot \overline{\nu}(y)  , 
 \eeqnn
 which is smaller than $ t \cdot \underline{n}(\zeta ) \cdot \overline{\nu}(x)  \cdot \overline{\nu}(y)$. The proof ends.
 \qed  
 
 %The state $0$ is regular for $(0,\infty)$ [resp. for $(-\infty,0)$] if and only if it is regular for $\{ 0\}$, for $\overline{X}-X$ [resp. for $X-\underline{X}$].  

 %Note that our assumption that $X$ has non-monotone sample paths, the state $0$ is always regular for $(0,\infty)$ or/and $(-\infty,0)$.
 %It is well-known that the two reflected processes $\overline{X}-X$ and $X-\underline{X}$ are strong Markov processes. 
 %The point processes of their excursions away $0$ form two Poisson point processes with intensities $\overline{n}(d\epsilon)$ and $\underline{n}(d\epsilon)$ respectively. 
 %The state $0$ is regular for $(0,\infty)$ [resp. for $(-\infty,0)$] if and only if it is regular for $\{ 0\}$, for $\overline{X}-X$ [resp. for $X-\underline{X}$].  

 Write $\widehat{X}_t:=-X_t$ for the dual L\'evy process of $X$. Let $\widehat{L} = \{\widehat{L}_t: t\ge 0\}$ be the local time at zero of $X-\underline{X}$ in the sense of \cite[p.109]{Bertoin1996}.
 Without loss of generality, we normalize it so that 
 \beqnn
  \mathbf{E}\bigg[ \int_0^\infty e^{-t} d \widehat{L}_t\bigg] =1. 
 \eeqnn
 Its right-inverse $\widehat{L}^{-1}$ is referred to as the \textsl{descending ladder time process}, and $\widehat{H}= \{ \widehat{H}_t :t\geq 0 \}$ with $\widehat{H}_t := \widehat{X}_{\widehat{L}^{-1}_t}$ is called the \textsl{descending ladder height process}. 
 Lemma~2 in \cite[p.157]{Bertoin1996} states that $(\widehat{L}^{-1},\widehat{H})$ is a L\'evy process (possibly killed at an exponential rate), whose \textit{bivariate exponent}, by \cite[p.166]{Bertoin1996}, is given by
 \beqnn
 \widehat\kappa(\lambda,u)
 =-\log \mathbf{E}\Big[e^{-\lambda \widehat{L}^{-1}_1-u\widehat{H}_1}\Big]
 \ar=\ar
 \exp\Big\{\int_0^\infty \frac{dt}{t}\int_{(-\infty,0]}
 (e^{-t}-e^{-\lambda t+ux})\mathbf{P}(X_t\in dx)\Big\},
 \quad \lambda,u\geq 0.
 \eeqnn
  In particular, both $\widehat L^{-1}$ and $\widehat H$ are (possibly killed) subordinators with drifts $\underline{\tt d},\underline{ a}\geq 0$.
  Moreover, $\underline{\tt d}=0$ if and only if the state $0$ is regular for   $(0,\infty)$, for $X$. 
  By (6.16) in \cite[p.170]{Kyprianou2014}, the bivariate exponent $\widehat\kappa(\lambda, u)$ also admits the  representation
  \beqlb\label{eqn.kappa}
   \widehat \kappa(\lambda,u) \ar=\ar \underline{\tt d}\cdot \lambda + \underline{ a}\cdot u + \int_0^\infty \int_0^\infty (1- e^{-\lambda t-u x})\,\underline{n}(\zeta \in dt,-\epsilon_\zeta \in dx). 
  \eeqlb
  Since $0$ is assumed to be regular for $(0, \infty)$, for $X$, we have $\underline{\mathrm{d}}=0$.
  Applying integration by parts to \eqref{eqn.kappa} with $u=0$ gives that
  \beqlb\label{eqn.551}
  \widehat \kappa(\lambda,0) =  \int_0^\infty (1- e^{-\lambda t})\,\underline{n}(\zeta \in dt)
  = \lambda \int_0^\infty  e^{-\lambda t}\underline{n}(\zeta > t)\, dt,\quad \lambda \geq 0.
  \eeqlb
  By Theorem 6.9 (ii) in \cite{Kyprianou2014},  we have that $\kappa(0,0) >0$ since $X$ drifts to  $-\infty$ and hence 
  \beqlb\label{eqn.501}
  \underline{n}(\zeta )= \int_0^\infty  \underline{n}(\zeta > t)dt=\lim_{\lambda\to 0}\frac{\widehat \kappa(\lambda,0)}{\lambda} =\frac{1}{ \kappa(0,0)} < \infty.
  \eeqlb
  Here the third equality follows from the identity $\widehat \kappa(\lambda,0) \cdot \kappa(\lambda,0)=\lambda$ for all $\lambda \geq 0$.

 Denote by $\widehat{U}(dt,dy)$ the \textsl{potential measure} of $(\widehat{L}^{-1},\widehat{H})$. 
 It is a $\sigma$-finite measure on $\mathbb{R}^2_+$ defined by
 \beqlb\label{eqn.U}
 \widehat{U}(dt,dy):= \int_0^\infty \mathbf{P}\big( \widehat{L}^{-1}_s \in  dt,\widehat{H}_s\in dy \big)\,ds.
 \eeqlb
 In particular, its marginal cumulative function in space is also known as the \textsl{renewal function} associated with the ladder height process $\widehat{H}$, that is, 
 \beqlb\label{eqn.DefV}
  \widehat{V}(x) := \widehat{U}\big(\mathbb{R}_+,[0,x]\big)= \int_0^\infty \mathbf{P}\big( \widehat{H}_s \leq x \big)\,ds,\quad x\geq 0;
 \eeqlb
 see e.g.  \cite[Chapter VI]{Bertoin1996} and \cite[Chapter 6]{Kyprianou2014}. 
 By \eqref{eqn.kappa} and \eqref{eqn.U}, the function $\widehat{V}$ also satisfies that
 \beqlb\label{eqn.447}
 \widehat{V}(x) = \bar{\tt d} +  \int_0^\infty  \overline{n}(  \epsilon_s\leq x, \zeta>s)\, ds ,\quad x\geq 0;
 \eeqlb
 see also (16) in \cite{DoneyRivero2013}. 
  Recall that $\tau_0^-$ is the first passage time of $X$ into $(-\infty,0)$.    For any $x>0$, by Proposition 17 (ii) in \cite[p.172]{Bertoin1996} and Theorem 2.2  in \cite{DenisovShneer2013} we have 
 \beqlb  \label{Asytau} 
 \mathbf{E}_{x}[\tau_0^-]%=\frac{\widehat{V}(x)}{\kappa(0,0)} 
 = \underline{n}(\zeta ) \cdot \widehat{V}(x)
 \quad \mbox{and}\quad 
 \mathbf{P}_x\big(\tau_0^->t\big)  
 \sim  \underline{n}(\zeta ) \cdot \widehat{V}(x) \cdot \overline{\nu}(\beta t) \in \mathrm{RV}^\infty_{-\alpha} ,\quad \mbox{as }t\to\infty. 
 \eeqlb
 Analogously, 
 %let $\overline{X} = \{\overline{X}_t: t\ge 0\}$ be the \textit{supremum process} with $\overline{X}_t = \sup_{s\in[0,t]}X_s$ and $\overline{X}-X $ be the \textit{reflected process at the past supremum}. 
 quantities associated with $\overline{X}$ and $\overline{X}-X$ are denoted without hats, e.g., $\kappa(\lambda,u)$, $U(dt,dy)$ and $V(x)$,  and satisfy analogous properties. 
 In particular, we have for any $x\geq 0$, 
 \beqlb\label{eqn.2011}
 V(x) = \int_0^\infty \mathbf{P}\big( H_s \leq x \big)\,ds =   \int_0^\infty  \underline{n}(  \epsilon_s\leq x, \zeta>s)\, ds  
 \quad \mbox{and}\quad 
 V(\infty)=  \underline{n}(\zeta) <\infty. 
 \eeqlb

 \subsection{Auxiliary lemmas}
 \label{Sec.AuxiliaryLemma}

   This section collects three auxiliary lemmas that play an important role in the proofs of our main theorems. The first is drawn from Corollary~1.4 in \cite{WuXu2026a}, which provides a unified treatment of uniform local large-deviation asymptotics for centered L\'evy processes with regularly varying jumps. The second and third establish, respectively, a moment estimate and the local asymptotics of the entrance law $\overline n_t(dx)$; see Lemmas~\ref{Lem5.1} and~\ref{Lem5.4}.

 \begin{lemma}  \label{Lemma.203}
 For any $\lambda,\delta_0>0$ and any non-negative function $\vartheta$ on $\mathbb{R}_+$ with $\vartheta_x/x\to0$ as $x\to\infty$, we have as $t\to\infty$, 
  \beqlb\label{eqn.324}
  \sup_{0<s\leq  t} \sup_{x\geq \lambda t} \sup_{|y|\leq \vartheta_x} \sup_{\delta\in[\delta_0,\infty]} 
  \bigg| \frac{\mathbf{P}\big(X_s + \beta s \in x -y+ \varDelta_\delta \big)}{s\cdot  \nu\big(x + \varDelta_\delta \big)   } -1\bigg|    
   \to 0. 
  \eeqlb
 \end{lemma}  
 
   \begin{proposition}\label{Prop.206}
  	There exists a constant $C>0$ such that for all $t\geq 0$, 
  	\beqnn
  	\mathbf{E}\bigg[ \sup_{s\in[0,t]} \big|X_s\big| \bigg] \leq C\cdot \big( \sqrt{t} +t  \big).
  	\eeqnn
  	% Moreover, for any $t_0>0$, the family $\{|X_t/t| \}_{t\geq t_0}$ is uniformly integrable.
  \end{proposition}
  \proof By the L\'evy-It\^o decomposition, we have $X_t= -\beta t +Y_t+M_t$ with
  \beqnn
  Y_t:=  \sigma B_t+ \int_0^t \int_{|y|\leq 1} y \,\widetilde{N}(ds,dy)
  \quad \mbox{and}\quad 
  M_t:=   \int_0^t \int_{|y|>1} y \widetilde{N}(ds,dy),\quad t\geq 0.
  \eeqnn 
  Here $B$ is a standard Brownian motion and $\widetilde{N}(ds,dy)$ is a compensated Poisson random measure on $(0,\infty)\times \mathbb{R}$ with intensity $ds\,\nu(dy)$. 
  By the triangle inequality, 
  \beqnn
  \mathbf{E}\bigg[ \sup_{s\in[0,t]} \big|X_s\big| \bigg] \leq \beta t+ 	\mathbf{E}\bigg[ \sup_{s\in[0,t]} \big|Y_s\big| \bigg] + 	\mathbf{E}\bigg[ \sup_{s\in[0,t]} \big|M_s\big| \bigg] .
  \eeqnn  
  By using Jensen's inequality, then the Burkholder–Davis–Gundy inequality and finally \eqref{eqn.LevyMeasure}, there exists a constant $C>0$ such that for all $t\geq 0$, 
  \beqlb\label{eqn.3044}
  \mathbf{E}\Big[ \sup_{r\in[0,t]}  |Y_r | \Big] 
  \ar\leq\ar \bigg(\mathbf{E}\Big[ \sup_{r\in[0,t]}  |Y_r |^2 \Big] \bigg)^{1/2}
  \leq C\cdot \Big( \sigma^2 + \int_{|y|\leq 1} |y|^2 \,\nu(dy) \Big)^{1/2}\cdot \sqrt{t}  \leq C\cdot \sqrt{t}. 
  \eeqlb
  Additionally, by the definition of $\widetilde{N}(ds,dy)$ and then using \eqref{eqn.LevyMeasure} again, we also have uniformly in $t\geq 0$, 
  \beqnn
  \sup_{r\leq t}\big|M_r\big| \leq   \int_{|y|>1} |y|\,\nu(dy) \cdot  t+\int_0^t \int_{|y|>1} |y| \,N(ds,dy) 
  \quad \mbox{and then}\quad
  \mathbf{E}\Big[ \sup_{r\in[0,t]}  |M_r | \Big] \leq C\cdot t,
  \eeqnn
  The desired uniform upper bound follows by putting these two estimates together. 
  % 	For the second claim, by the triangle inequality we have 
  % 	\beqnn
  % 	\mathbf{E}\big[\big|X_t/t+\beta\big|\big] 
  % 	\ar\leq\ar  \mathbf{E}\big[\big|Y_t/t\big|\big]  +  \mathbf{E}\big[\big|M_t/t\big|\big] .
  % 	\eeqnn
  % 	By \eqref{eqn.3044}, we have $ \mathbf{E}\big[\big|Y_t/t\big|\big]  \leq C/\sqrt{t} \to 0$ as $t\to\infty$. 
  % 	Moreover, for any $K>0$, we also have 
  % 	\beqnn
  % 	\mathbf{E}\big[\big|M_t/t\big|\big] 
  % 	\ar\leq\ar  \mathbf{E}\Big[ \frac{1}{t}\Big| \int_0^t \int_{1<|y|\leq K}y \widetilde{N}(ds,dy) \Big| \Big] +  \mathbf{E}\Big[ \frac{1}{t}\Big| \int_0^t \int_{|y|> K}y \widetilde{N}(ds,dy) \Big| \Big]\cr
  % 	\ar\leq\ar  \bigg(\frac{1}{t^2}\cdot \mathbf{E}\Big[ \Big| \int_0^t \int_{1<|y|\leq K}y \widetilde{N}(ds,dy) \Big|^2 \Big]\bigg)^{1/2} +  \mathbf{E}\Big[ \frac{1}{t}\Big| \int_0^t \int_{|y|> K}y \widetilde{N}(ds,dy) \Big| \Big]\cr
  % 	\ar\leq\ar \Big(  \frac{1}{t}  \int_{1<|y|\leq K}y^2\nu(dy)  \Big)^{1/2} + 2\int_{|y|>K} |y| \nu(dy), 
  % 	\eeqnn   
  % 	which goes to $0$ as $t\to\infty$ and then $K\to\infty$. Combining these estimates together yields that $ \mathbf{E}\big[\big|X_t/t+\beta\big|\big]  \to 0$ as $t\to\infty$ and then $\{ |X_t/t| \}_{t\geq t_0}$ is uniformly integrable. 
  \qed 
  \begin{lemma}\label{Lem5.1}
  	We have $\overline{n} (\epsilon_t)=\overline{n} (\epsilon_t, \zeta>t )<\infty$ for any $t\geq 0$. 
  \end{lemma}
  \proof For any $q> 0$, by using (7) in \cite[p.120]{Bertoin1995} and then Proposition~\ref{Prop.206} we have 
  \beqnn
  \int_{0}^{\infty}e^{-q t}\,\overline{n} (\epsilon_t,\zeta>t)\, dt
  \ar=\ar \kappa(q,0)\cdot\mathbf{E}\Big[\int_{0}^{\infty}e^{-qt}  (\overline{X}_t-X_t)dt\Big] 
  = \kappa(q,0)\cdot \int_{0}^{\infty}e^{-qt}  \big(\mathbf{E}\big[\overline{X}_t\big]+\beta t\big)dt,
  \eeqnn  
  which is finite because of Proposition~\ref{Prop.206} and then  the desired result follows. 
  \qed
 
 \begin{proposition}\label{Lemma.205}
 	For any $x\in \mathbb{R}$ and $\delta >0$, there exists a constant $C>0$ such that for all $t\geq 0$,
 	\beqlb\label{eqn.2051}
 	\mathbf{P}\big( X_t\in x+\varDelta_\delta\big)\leq C\cdot  \big(\overline{\nu}(\beta t) \wedge 1\big).
 	\eeqlb
 \end{proposition}
 \proof Note that $\mathbf{P}\big( X_t\in x+\varDelta_\delta\big) \leq \mathbf{P}\big( X_t\in [x]-1+\varDelta_{[\delta]+3}\big)$ for all $x\in \mathbb{R}$, it suffices to prove \eqref{eqn.2051} with $x,\delta\in \mathbb{Z}$. 
 By the monotonicity of $\overline{\nu}$, it suffices to prove that the desired uniform upper bound holds for all large $t$.  
 Firstly, we have 
 \beqlb\label{eqn.213}
 \frac{\mathbf{P}\big( X_t\in x+\varDelta_\delta\big)}{\overline{\nu}(\beta t)}
 \ar=\ar  \frac{t }{\overline{\nu}(\beta t)} \sum_{k=x}^{x+\delta-1} \frac{\mathbf{P}(X_t\in  k+\varDelta_1)}{t\cdot\nu(\beta t+k+\varDelta_1)}\cdot   \nu\big(\beta t+k+\varDelta_1\big) \cr
 \ar\leq\ar  \frac{t  }{\overline{\nu}(\beta t)} \cdot   \nu\big(\beta t+x+\varDelta_{\delta}\big) \cdot\sup_{k=x,\cdots,x+\delta-1} \frac{\mathbf{P}(X_t\in  k+\varDelta_1)}{t\cdot\nu(\beta t+k+\varDelta_1)} .
 \eeqlb
 By Lemma~\ref{Lemma.203} with $\lambda=\beta/2$,   there exist two constants $C,t_0>0$  such that for all $t\geq t_0$,
 \beqnn
  \sup_{k=x,\cdots,x+\delta-1} \frac{\mathbf{P}(X_t\in  k+\varDelta_1)}{t\cdot\nu(\beta t+k+\varDelta_1)}
 \leq    \sup_{z\geq \beta t/2}\frac{\mathbf{P}(X_t+\beta t\in z+\varDelta_1)}{t\cdot \nu\big(z+\varDelta_1\big)}  \leq C. 
 \eeqnn
 Moreover, by \eqref{eqn.30} and enlarging $t_0$ if needed, we obtain that for all $t\geq t_0$,
 \beqnn
 \frac{t  }{\overline{\nu}(\beta t)} \cdot   \nu\big(\beta t+x+\varDelta_{\delta}\big)
 \leq   \frac{2\alpha}{\beta} \cdot (1+\delta) . 
 \eeqnn 
 Then \eqref{eqn.2051} follows by taking these two estimates back into the right-hand side of \eqref{eqn.213}. 
 \qed

 \begin{proposition}\label{Lemma.207}
 For any $x,y>0$, we have  as $t\to \infty$,
 \beqnn
 \frac{\beta t}{\overline{\nu}(\beta t)}\cdot  \mathbf{P}_{-x}\big(X_t\in[-y,0),\tau_{0}^{+}>t\big)
 \to \alpha\cdot  V(x)\int_{0}^{y}\widehat{V}(z)dz .
 \eeqnn
 \end{proposition}
 \proof 
 By applying Theorem 3.15 in \cite{Xu2021a} and the extended continuity theorem for Laplace transforms (see Theorem XIII.1.2 in \cite[p.413]{Feller1971}), along with  the boundedness of $V(x)$ and $\widehat{V}(x)$ on each finite interval of the non-negative semi-axis, we obtain the desired result.
 \qed
 
 \begin{proposition}\label{Lemma.208}
 	For any $x>0$, there exist two constants $C,t_0>0$ such that for all $t\geq t_0$, 
 	\beqlb\label{eqn.214}
 	\sup_{y\geq 0} \mathbf{P}_{-y}\big(X_t\in(-x,0],\, 	\tau_0^+>t \big)  \leq C\cdot \frac{\overline\nu(\beta t)}{t}. 
 	\eeqlb
 \end{proposition}
 \proof Consider a three-parameter function $F_t(u,v) := \mathbf P\big( \overline{X}_t\leq u,\, \overline{X}_t-X_t\leq v \big)$ with $t\geq 0$ and $u,v\geq 0$. 
 By (3.32) and (3.38) in \cite{Xu2021a},  we have 
 \beqlb\label{eqn.2031}
 \int_0^\infty F_s(u,v)\,ds = V(u)\cdot \widehat{V}(v)
 \eeqlb
 and
 \beqlb\label{eqn.203}
 F_t(u,v) =  \frac{1}{t}	\int_0^t ds \int_0^u F_s(u-z,v) \, \mathbf P\big(X_{t-s}\in dz
 \big) 
 + \frac{1}{t}	\int_0^t ds \int_0^v  F_s(u,v-z) \, \mathbf P\big( -X_{t-s}\in dz
 \big). 
 \eeqlb 
 The spatial homogeneity of $X$ yields that
 \beqnn
 \mathbf P_{-y}\big( X_t\in(-x,0],\, \tau_0^+>t \big) = \mathbf P\big( X_t\in(y-x,y],\, \overline{X}_t\leq y \big).
 \eeqnn
 The fact $\big\{ X_t\in(y-x,y],\, \overline{X}_t\leq y \big\} \subset
 \big\{ \overline{X}_t \in (y-x-1,y],\, \overline{X}_t-X_t\leq x \big\}$ induces that 
 \beqnn
 \mathbf P_{-y}\big( X_t\in(-x,0],\, \tau_0^+>t \big) 
 \leq  F_t(y,x)-F_t(y-x-1,x) ,
 \eeqnn
 which, by \eqref{eqn.203}, can be decomposed into the following parts
 where 
 \beqlb
 I_1(t,y) \ar:=\ar   \frac{1}{t}\int_0^{t/2} ds 
 \int_0^y \big( F_s(y-z,x)-F_s(y-x-1-z,x) \big) \, 
 \mathbf P\big( X_{t-s}\in dz \big), \label{eqn.204} \\
 I_2(t,y) \ar:=\ar   \frac{1}{t}\int_{t/2}^t ds 
 \int_0^y \big( F_s(y-z,x)-F_s(y-x-1-z,x) \big) \, 
 \mathbf P\big( X_{t-s}\in dz \big), \label{eqn.205}\\
 I_3(t,y) \ar:=\ar  \frac{1}{t}\int_0^t ds \int_0^x \big( F_s(y,x-z)-F_s(y-x-1,x-z) \big) \,
 \mathbf P\big( -X_{t-s}\in dz \big) .  \label{eqn.206}
 \eeqlb
 Thus \eqref{eqn.214}   follows if we can prove that there exist constants $C,t_0>0$ such that for all $t\geq t_0$, 
 \beqlb\label{eqn.2201}
 \sup_{y\geq 0} I_i(t,y) \leq C\cdot \frac{\overline\nu(\beta t)}{t},
 \quad i=1,2,3.
 \eeqlb
 
 \textbf{Case $i=1$.}
 The increase of $F_s(u,v)$ in $u$ yields that
 \beqlb
 0\leq F_s(y-z,x)-F_s(y-x-1-z,x) \leq F_s\big(y-[z],x)-F_s(y-x-[z]-2,x\big) , 
 \eeqlb
 for all $z\in(0,y]$, which follows that
 \beqlb\label{eqn.207}
 \lefteqn{\int_0^y \big( F_s(y-z,x)-F_s(y-x-1-z,x) \big) \, 
 	\mathbf P\big( X_{t-s}\in dz \big)}\ar\ar\cr
 \ar\leq\ar \sum_{k=0}^{[y]}\int_k^{k+1} \big( F_s(y-z,x)-F_s(y-x-1-z,x) \big) \, 
 \mathbf P\big( X_{t-s}\in dz \big)\cr
 \ar\leq\ar  \sum_{k=0}^{[y]} \big( F_s(y-k,x)-F_s(y-x-k-2,x) \big) \cdot  
 \mathbf P\big( X_{t-s}\in k+\varDelta_1 \big)\cr
 \ar\leq\ar \sup_{u\geq 0} 	\mathbf P\big( X_{t-s}\in u+\varDelta_1  \big) \cdot \sum_{k=0}^{[y]} \big( F_s(y-k,x)-F_s(y-x-k-2,x) \big).
 \eeqlb
 Plugging this into the right-hand sides of \eqref{eqn.204}, 
 \beqlb\label{eqn.208}
 \lefteqn{\frac{I_{1}(t,y)}{\overline{\nu}(\beta t)} 
 	\leq \frac{1}{t}
 	\int_0^{t/2} \sum_{k=0}^{[y]} \big(F_s(y-k,x) - F_s(y-x-k-2,x) \big)\cdot  \sup_{u\geq0} \frac{\mathbf P\big( X_{t-s}\in  u+\varDelta_1 \big)}{\overline{\nu}(\beta t)} \,ds } 
 \ar\ar\cr
 \ar\leq\ar \frac{1}{t}\sup_{r\in[\frac{t}{2},t]}\sup_{u\geq0} \frac{\mathbf P\big( X_r\in  u+\varDelta_1 \big)}{\overline{\nu}(\beta t)}\cdot  \int_0^{t/2} \sum_{k=0}^{[y]} \big(F_s(y-k,x) - F_s(y-x-k-2,x) \big)ds.   \qquad
 \eeqlb
 By using Lemma~\ref{Lemma.203} with $\lambda =\beta/2$, then  \eqref{eqn.30} and the monotonicity of $\overline{\nu}$, we obtain that uniformly for all sufficiently  large $t$,
 \beqnn
 \sup_{r\in[\frac{t}{2},t]}\sup_{u\geq0} \frac{\mathbf P\big( X_r\in  u+\varDelta_1 \big) }{r\cdot \nu\big(\beta r+u +\varDelta_1\big)} 
 \leq\sup_{r\in[\frac{t}{2},t]}\sup_{u\geq \beta t/2} \frac{\mathbf P\big( X_r+\beta r\in  u+\varDelta_1 \big) }{r\cdot \nu\big(u +\varDelta_1\big)} \leq C
 \eeqnn
 and 
 \beqnn
 \sup_{r\in[\frac{t}{2},t]}\sup_{u\geq0} \frac{r\cdot \nu\big(\beta r+u +\varDelta_1\big)}{\overline{\nu}(\beta t) } 
 \leq C\cdot    \frac{ \overline{\nu}(\beta t/2)}{  \overline{\nu}(\beta t) } \leq C ,
 \eeqnn
 which follows that 
 \beqlb\label{eqn.222}
  \sup_{r\in[\frac{t}{2},t]}\sup_{u\geq0} \frac{\mathbf P\big( X_r\in  u+\varDelta_1 \big) }{\overline{\nu}(\beta t) } \leq C.
 \eeqlb
 Moreover, applying Fubini's theorem to the last integral in \eqref{eqn.208} and then using \eqref{eqn.2031},  
 \beqnn
 \lefteqn{\int_0^{t/2} \sum_{k=0}^{[y]} \big(F_s(y-k,x) - F_s(y-x-k-2,x) \big)ds}\ar\ar\cr
 \ar\leq\ar \sum_{k=0}^{[y]}   \int_0^\infty F_s(y-k,x)\,ds - \sum_{k=0}^{[y]}  \int_0^\infty F_s(y-x-k-2,x)\,ds\cr
 \ar=\ar \sum_{k=0}^{[y]}    V(y-k) \cdot \widehat{V}(x) 
 -  \sum_{k=0}^{[y]} V(y-x-k-2)\cdot \widehat{V}(x)  . 
 \eeqnn 
 which, by  the increase of $V$ and the fact that $V(\infty)<\infty$, can be bounded by 
 \beqnn
 \lefteqn{\widehat{V}(x)\sum_{k=0}^{[y]}    V\big([y]+1-k\big)  - \widehat{V}(x)\sum_{k=0}^{[y]}   V\big([y]-[x]-k-3\big)  }\ar\ar\cr
 \ar=\ar \widehat{V}(x)\bigg(\sum_{k=1}^{[y]+1}    V(k)  -  \sum_{k=-[x]-3}^{[y]-[x]-3}  V(k) \bigg)
 \leq  \widehat{V}(x)\cdot \sum_{k=[y]-[x]-2}^{[y]+1} V(k)  \leq   \widehat{V}(x) \cdot (x+4) \cdot V(\infty) ,
 \eeqnn
 uniformly in $y\geq 0$.  
 Taking this and \eqref{eqn.222} back into \eqref{eqn.208},  there exist some constants $C,t_0>0$ such that for all $t\geq t_0$,
 \beqnn
 \sup_{y\geq 0}\frac{I_{1}(t,y)}{\overline{\nu}(\beta t)}  \leq \frac{C}{t}
 \quad \mbox{and then}\quad 
 \sup_{y\geq 0} I_{1}(t,y) \leq C\cdot \frac{\overline{\nu}(\beta t)}{t}. 
 \eeqnn 
 
 \textbf{Case $i=2$.} We now consider $I_2(t,y)$.  
 By the definition of  $F_s(u,v)$, we have 
 \beqnn
 \sum_{k=0}^{[y]} \big(F_s(y-k,x) - F_s(y-x-k-2,x) \big)
 \ar=\ar \sum_{k=0}^{[y]} \mathbf{P}\big( y-x-k-2<\overline{X}_s\leq y-k,\, \overline{X}_s-X_s\leq x \big) \cr
 \ar=\ar \mathbf{E}\bigg[ \mathbf{1}_{\{\overline{X}_s-X_s\leq x \}}\cdot \sum_{k=0}^{[y]}  \mathbf{1}_{\{ y-x-k-2<\overline{X}_s\leq y-k \}} \bigg], 
 \eeqnn
 which  can be bounded by $\big([x]+4\big) \cdot \mathbf{P} \big(  \overline{X}_s-X_s\leq x\big) \leq (x+4 ) \cdot \mathbf{P}_x \big(\tau_0^->s\big)$ for all $y\geq 0$, since 
 \beqnn
 \sum_{k=0}^{[y]}  \mathbf{1}_{\{ y-x-k-2<z\leq y-k \}} \leq x+4,\quad 
 \forall\, y,z\geq 0.
 \eeqnn 
 Taking this back into the right-hand side of the last inequality in \eqref{eqn.207} and then plugging \eqref{eqn.207} into $I_2(t,y)$, we get
 \beqnn
 \sup_{y\geq 0} I_2(t,y) 
 \ar\leq\ar \frac{x+4}{t}\int_{t/2}^t  \sup_{u\geq 0} 	\mathbf P\big( X_{t-s}\in u+\varDelta_1 \big) \cdot \mathbf{P}_x \big(\tau_0^->s\big)\, ds\cr
 \ar\leq\ar \frac{x+4}{t}\cdot  \mathbf{P}_x \big(\tau_0^->t/2\big) \cdot \int_0^\infty\sup_{u\geq 0} 	\mathbf P\big( X_{s}\in u+\varDelta_1 \big) ds .
 \eeqnn
 By \eqref{eqn.222}, the last integral is finite. 
 Moreover,  the second result of \eqref{Asytau} yields that $  \mathbf{P}_x \big(\tau_0^->t/2\big) \leq  C \cdot \overline{\nu}(\beta t)$ for all $t\geq 0$. 
 Consequently, the upper bound \eqref{eqn.2201} with $i=2$ holds for large $t$.
 
 \textbf{Case $i=3$.}  
 For $0\leq z\leq x$, by the definition of $F_s(u,v)$ one can easily identify that  
 \beqnn
 F_s(y,x-z)-F_s(y-x-1,x-z)  \leq F_s(y,x)-F_s(y-x-1,x),
 \eeqnn
 which follows that $I_3(t,y)\leq I_{31}(t,y) + I_{32}(t,y)$ with
 \beqnn
 I_{31}(t,y)\ar:=\ar\frac{1}{t}\int_0^{t/2} \big( F_s(y,x)-F_s(y-x-1,x) \big)\cdot \mathbf P\big( X_{t-s}\in[-x,0] \big)\, ds\cr
 I_{32}(t,y)\ar:=\ar \frac{1}{t}\int_{t/2}^t \big( F_s(y,x)-F_s(y-x-1,x) \big)\cdot \mathbf P\big( X_{t-s}\in[-x,0] \big)\, ds.
 \eeqnn
 For the first term, by \eqref{eqn.2031} we have 
 \beqnn
 \frac{I_{31}(t,y)}{\overline{\nu}(\beta t)} 
 \ar\leq\ar  \frac{1}{t}\sup_{r\in[t/2,t]}
 \frac{\mathbf P\big( X_r\in[-x,0] \big)}{\overline{\nu}(\beta t)}\cdot \int_0^{t/2}\big( F_s(y,x)-F_s(y-x-1,x) \big)\, ds\cr
 \ar\leq\ar  \frac{1}{t}\sup_{r\in[t/2,t]}
 \frac{\mathbf P\big( X_r\in[-x,0] \big)}{\overline{\nu}(\beta t)}\cdot \int_0^\infty\big( F_s(y,x)-F_s(y-x-1,x) \big)\, ds \cr
 \ar=\ar  \frac{1}{t}\sup_{r\in[t/2,t]}
 \frac{\mathbf P\big( X_r\in[-x,0] \big)}{\overline{\nu}(\beta t)}\cdot  \big( \widehat{V}(x)V(y)-\widehat{V}(x)V(y-x-1) \big) . 
 \eeqnn
 Similarly as in \eqref{eqn.222}, we also have uniformly in large $t$,
 \beqnn
  \sup_{r\in[t/2,t]}
 \frac{\mathbf P\big( X_r\in[-x,0] \big)}{\overline{\nu}(\beta t)} \leq C.
 \eeqnn
 This along with the fact that $V(\infty)<\infty$ induces that uniformly in large  $t$
 \beqlb\label{eqn,221}
 \sup_{y\geq 0}  \frac{I_{31}(t,y)}{\overline{\nu}(\beta t)}  \leq \frac{C}{t}  . 
 \eeqlb
 For $I_{32}(t,y)$,  we have first for all $t>0$,
 \beqnn
 \frac{I_{32}(t,y)}{\overline{\nu}(\beta t) } 
 %\ar\leq\ar  \frac{1}{t} \sup_{r\in[t/2,t]}\big( F_r(y,x)-F_r(y-x-1,x) \big)\cdot \int_{t/2}^t  \mathbf P\big( X_{t-s}\in[-x,0] \big)\, ds\cr
 \ar\leq\ar \frac{1}{t} \sup_{r\in[t/2,t]}\frac{ F_r(y,x)-F_r(y-x-1,x)}{\overline{\nu}(\beta t) }\cdot \int_0^\infty  \mathbf P\big( X_s\in[-x,0] \big)\, ds. 
 \eeqnn
  Proposition~\ref{Lemma.205} induces that the last integral is finite.  
 Moreover, the definition of $F_r(u,v)$ shows that 
 \beqnn
 F_r(y,x)-F_r(y-x-1,x) = \mathbf{P}\big(y-x-1<\overline{X}_r\leq y,\,
 \overline{X}_r-X_r\leq x \big) .
 \eeqnn
 Note that $\{  y-x-1<\overline{X}_r\leq y\}\cap \{ \overline{X}_r-X_r\leq x \} \subset\{y-2x-1< X_r\leq y\}$, we have
 \beqnn
 \sup_{r\in[t/2,t]} \frac{F_r(y,x)-F_r(y-x-1,x)}{\overline{\nu}(\beta t) }
 \leq \sup_{r\in[t/2,t]}  \frac{\mathbf{P}\big( X_r\in(y-2x-1,y] \big)}{\overline{\nu}(\beta t) }. 
 \eeqnn
 Similarly as in \eqref{eqn.222}, there exist constants $C,t_0>0$ such that for all $t\geq t_0$,
 \beqnn
  \sup_{y\geq 0} \sup_{r\in[t/2,t]}  \frac{\mathbf{P}\big( X_r\in(y-2x-1,y] \big)}{\overline{\nu}(\beta t) }
 \ar\leq\ar    C
 \quad\mbox{and then}\quad 
 \sup_{y\geq 0} I_{32}(t,y) \leq C \cdot \frac{\overline{\nu}(\beta t)}{t}.  
 \eeqnn 
 This along with \eqref{eqn,221} immediately yields the upper bound \eqref{eqn.2201} with $i=3$. 
 \qed

 \begin{lemma}\label{Lem5.4}
 For any $x\geq 0$ and $\delta>0$, we have  as $t\to \infty$,
 \beqnn
 \frac{\beta t}{\overline{\nu}(\beta t)}\cdot 	\overline{n}_t\big(x+\varDelta_\delta\big) \to  \alpha\int_x^{x+\delta} \widehat{V}(z)\, dz . 
 \eeqnn
 \end{lemma}
 \proof We just need to prove the case of $x=0$. Other cases follow directly.  
 For any $t>1$,  by the Markov  property we have 
 \beqnn
  \frac{\beta t}{\overline{\nu}(\beta t)}\cdot\overline{n}_t\big(\varDelta_\delta\big)
 = \int_0^\infty  \frac{\beta t}{\overline{\nu}(\beta t)}\cdot\mathbf{P}_{-w}\big(X_{t-1} \in [-\delta,0) , \, \tau_0^+>t-1\big)\,
 \overline{n}_1(dw).
 \eeqnn 
 Passing both sides to their limits as $t\to\infty$, then using the dominated convergence theorem with the help of  Proposition~\ref{Lemma.208} and finally Proposition~\ref{Lemma.207}, we obtain 
 \beqnn
  \lim_{t\to\infty} \frac{\beta t}{\overline{\nu}(\beta t)}\cdot\overline{n}_t\big(\varDelta_\delta\big)
  %\ar=\ar \int_0^\infty   \lim_{t\to\infty}  \frac{\beta t}{\overline{\nu}(\beta t)}\cdot\mathbf{P}_{-w}\big(X_{t-r} \in [-\delta,0) , \, \tau_0^+>t-r\big)\, \overline{n}_r(dw)\cr
  \ar=\ar \alpha \int_0^\delta \widehat{V}(z)\,dz \cdot \int_0^\infty  V(w) \,
  \overline{n}_1(dw) =\alpha \int_0^\delta \widehat{V}(z)\,dz  . 
 \eeqnn
 The last equality follows directly from Theorem 3 in  \cite{Chaumont1996}  or (3.6) in \cite{CuiWuXu2025}. 
 \qed

  \section{Proofs of Theorem~\ref{MainThm00} and Corollary~\ref{Main.Corollary01}}
 \label{Sec.Proof01}
 \setcounter{equation}{0} 
 
 This section is devoted to the proofs of Theorem~\ref{MainThm00} and Corollary~\ref{Main.Corollary01}, which rest on the one-big-jump principle. Roughly speaking, this principle asserts that the dominant mechanism by which an excursion reaches a high level is the occurrence of a single large positive jump.
 To formulate this mechanism precisely, for $x\geq0$ and $\epsilon\in \mathcal{E}$,  define 
 \beqlb\label{eqn.70}
 \mathcal{B}^x(\epsilon):= \big\{ 0<s\leq \zeta: \Delta \epsilon_s>x \big\} 
 \quad \mbox{and}\quad 
 \mathcal{J}^x(\epsilon) := \inf \mathcal{B}^x(\epsilon) ,
 \eeqlb 
 with the convention that $\inf\emptyset=\infty$.
 The set $\mathcal{B}^x(\epsilon)$ consists of all jumps of $\epsilon$ whose size exceeds $x$, while $\mathcal{J}^x(\epsilon)$ denotes the arrival time of the first such jump. 
 When there is no risk of confusion, we suppress the dependence on $\omega$, and simply write $\mathcal{B}^x$ and $\mathcal{J}^x$. 
 Splitting according to whether $\mathcal{J}^{\lambda x}$ occurs by time $t$, we obtain
  \beqlb\label{eqn.5011}
   \underline{n}\big( \epsilon_t\in x+\varDelta_\delta,\,  \zeta>t \big) = \underline{n}\big( \epsilon_t\in x+\varDelta_\delta,\, \mathcal{J}^{\lambda x}> t,\,\zeta>t \big)  + \underline{n}\big( \epsilon_t\in x+\varDelta_\delta,\, \mathcal{J}^{\lambda x}\leq t,\,\zeta>t \big) 
  \eeqlb
 The next two propositions establish the asymptotic behavior of the two terms on the right-hand side respectively.
  
   \begin{proposition}\label{Prop.501}
  	For any $t,\delta>0$ and $\lambda\in (0,1)$, we have as $x\to\infty$,
  	\beqnn
  	 \frac{\underline{n}\big( \epsilon_t\in x+\varDelta_\delta,\, \mathcal{J}^{\lambda x}> t,\,\zeta>t \big) }{ \nu \big(  x+\varDelta_\delta \big)} \to 0.
  	\eeqnn 
  \end{proposition}
  \proof For any fixed $\eta \in (0,1-\lambda)$, on the event
  $\{\epsilon_t\in x+\varDelta_\delta,\,\mathcal{J}^{\lambda x}> t,\,\zeta>t\}$,
  we observe that $\tau^+_{\eta x}\leq t$ and $ \eta  x \leq   \epsilon_{\tau^+_{\eta x}} \leq (\eta+\lambda) x$.  This follows that 
  \beqnn
   \underline{n}\big( \epsilon_t\in x+\varDelta_\delta,\, \mathcal{J}^{\lambda x}> t,\,\zeta>t \big)
  \ar=\ar \int_0^t  \int_{\eta x}^{ (\eta+\lambda) x} \underline{n}\big( \tau^+_{\eta x} \in ds,\, \epsilon_s \in dz,\, \epsilon_t\in x+\varDelta_\delta,\, \mathcal{J}^{\lambda x}> t,\,\zeta>t \big)\cr
  \ar\leq\ar  \int_0^t   \int_{\eta x}^{ (\eta+\lambda) x}  \underline{n}\big( \tau^+_{\eta x} \in ds,\, \epsilon_s \in dz,\, \epsilon_t\in x+\varDelta_\delta,\,\zeta>t \big). 
  \eeqnn
 By the  Markov property  and the spatial homogeneity of $X$, the last expression equals
  \beqlb\label{eqn.4021}
 \lefteqn{ \int_0^t  \int_{\eta x}^{ (\eta+\lambda) x}
  \mathbf{P}_z\big(
  X_{t-s}\in x+\varDelta_\delta,
  \tau_0^->t-s
  \big) \underline{n}\big(\tau^+_{\eta x}\in ds,\,\epsilon_s\in dz,\,\zeta>s \big)}\cr
  \ar=\ar  \int_0^t \int_{\eta x}^{ (\eta+\lambda) x}
  \mathbf{P}\big(
  X_{t-s}\in x-z+\varDelta_\delta,
  \tau_{-z}^->t-s \big) \underline{n}\big(\tau^+_{\eta x}\in ds,\,\epsilon_s\in dz,\,\zeta>s \big),
  \eeqlb
  which can be bounded by $\sup_{0<s\leq t} \sup_{y\geq(1-\eta-\lambda)x}  \mathbf{P}\big( X_s  \in   y+\varDelta_\delta  \big)\cdot \underline{n}\big(\tau^+_{\eta x}\leq t \big)$.
  Consequently,  
  \beqnn
   \frac{\underline{n}\big( \epsilon_t\in x+\varDelta_\delta,\, \mathcal{J}^{\lambda x}> t,\,\zeta>t \big) }{ \nu \big(  x+\varDelta_\delta \big)}
 % \ar\leq\ar
%  \int_0^t  \int_{(\eta-\lambda) x}^{\eta x} 
%  \mathbf{P}\big( X_{t-s}\in x-z+\varDelta_\delta \big)\,
%  \underline{n}\big( \tau^+_{\eta x}\in ds, \epsilon_s\in dz,\zeta>s
%  \big)
 % \cr\cr
  \ar\leq\ar
  \sup_{0<s\leq t} \sup_{y\geq(1-\eta-\lambda)x}  \frac{ \mathbf{P}\big( X_s\in y+\varDelta_\delta  \big) }{ \nu \big(  x+\varDelta_\delta \big)}  \cdot \underline{n}\big(\tau^+_{\eta x}\leq t \big).
  \eeqnn 
  Invoking Lemma~\ref{Lemma.203} with  $\vartheta_x=0$, and subsequently applying \eqref{eqn.30} along with the monotonicity of $\overline{\nu}$, we obtain a constant $C>0$ such that, uniformly for sufficiently large $x$,
  \beqnn
  \lefteqn{\sup_{0<s\leq t} \sup_{y\geq(1-\eta-\lambda)x}  \frac{ \mathbf{P}\big( X_s\in y+\varDelta_\delta  \big) }{ \nu \big(  x+\varDelta_\delta \big)}  }\cr
  \ar=\ar \sup_{0<s\leq t} \sup_{y\geq(1-\eta-\lambda)x}  \frac{\mathbf{P}\big( X_s + \beta s\in \beta s+ y+\varDelta_\delta  \big) }{\nu \big(  \beta s + y+\varDelta_\delta \big)} \cdot \frac{\nu \big(  \beta s + y+\varDelta_\delta \big)}{\nu \big(    x+\varDelta_\delta \big)}  
  \leq C . 
  \eeqnn 
 Since $ \underline{n}\big(\tau^+_{\eta x}\leq t \big)\leq \underline{n}\big(\overline\epsilon > \eta x \big) \to 0$ as $x\to\infty$,  the desired limit follows directly. 
  \qed

  \begin{proposition}\label{Prop.502}
 	For any $t,\delta>0$ and $\lambda\in (0,1)$, we have as $x\to\infty$,
 	\beqlb\label{eqn.301}
 	\bigg|\frac{\underline{n}\big( \epsilon_t\in x+\varDelta_\delta,\, \mathcal{J}^{\lambda x}\leq t,\,\zeta>t \big) }{ \nu \big(  x+\varDelta_\delta \big)}-\underline{n}(t\wedge\zeta) \bigg| \to 0. 
 	\eeqlb
  \end{proposition}
 	\proof Let $ \mathcal{B}_t^{\lambda x}:= \mathcal{B}^{\lambda x}\cap [0,t]$  and $\# \mathcal{B}_t^{\lambda x}$ denote its cardinality.
 	When $\mathcal{J}^{\lambda x}\leq t$, we have $\mathcal{J}^{\lambda x} \in  \mathcal{B}_t^{\lambda x}$  and then 
  \beqnn
   0\leq  \underline{n}\big( \#\mathcal{B}_t^{\lambda x},\, \epsilon_t\in x+\varDelta_\delta,\, \zeta>t \big)  - \underline{n}\big( \epsilon_t\in x+\varDelta_\delta,\, \mathcal{J}^{\lambda x}\leq t,\,\zeta>t \big) 
  \leq  \underline{n}\Big( \#\mathcal{B}_t^{\lambda x} \big( \#\mathcal{B}_t^{\lambda x}-1 \big) \Big).
  \eeqnn
 Using the square of a sum formula and Proposition~\ref{Prop.203} shows that 
  \beqnn
   \underline{n}\Big( \#\mathcal{B}_t^{\lambda x} \big( \#\mathcal{B}_t^{\lambda x}-1 \big) \Big)
   \ar=\ar 2\cdot \underline{n}\bigg(  \sum_{0<s<r\leq t\wedge \zeta} \mathbf{1}_{\{\Delta \epsilon_s >\lambda x ,\, \Delta \epsilon_r > \lambda x  \}} \bigg) \leq 2\cdot t\cdot \underline{n}(\zeta ) \cdot \big| \overline{\nu}(\lambda x)  \big|^2 , 
  \eeqnn
   which is $o\big(  \nu(  x+\varDelta_\delta)\big)$ as $x\to\infty$. Consequently, the desired limit \eqref{eqn.301} follows immediately if 
 \beqlb\label{eqn.311}
  \limsup_{x\to\infty}\bigg|\frac{ \underline{n}\big( \#\mathcal{B}_t^{\lambda x},\,\epsilon_t\in x+\varDelta_\delta,\, \zeta>t \big) }{ \nu \big(  x+\varDelta_\delta \big)}-\underline{n}(t\wedge\zeta) \bigg| = 0. 
  \eeqlb
 Firstly, we have 
 \beqnn
 \underline{n}\Big( \# \mathcal{B}_t^{\lambda x},\, \epsilon_t\in x+\varDelta_\delta,\, \zeta>t \Big)
 \ar=\ar \underline{n}\bigg(  \sum_{0<s\leq t } \mathbf{1}_{\{ \Delta\epsilon_s> \lambda x,\, \epsilon_t\in x+\varDelta_\delta,\, \zeta>t \}} \bigg).
 \eeqnn
 Using the Markov property and the compensation formula \eqref{eqn.Compensation} to the right-hand side, we obtain
 \beqnn 
   \underline{n}\Big( \# \mathcal{B}_t^{\lambda x},\, \epsilon_t\in x+\varDelta_\delta,\, \zeta>t \Big) 
 	\ar=\ar 
 	\int_0^t ds \int_0^\infty \underline{n}_s(dz)
 	\int_{\lambda x}^\infty \mathbf{P}_{z+w}\big( X_{t-s}\in x+\varDelta_\delta,\,
 	\tau_0^->t-s \big)\nu(dw) .
 \eeqnn
 	By the spatial-homogeneity of $X$, the definition of $\tau_{-z-w}^-$ and Fubini's theorem,
 	\beqnn
 	\lefteqn{\int_{\lambda x}^\infty \mathbf{P}_{z+w}\big( X_{t-s}\in x+\varDelta_\delta,\, \tau_0^->t-s \big) \nu(dw) }\cr
 	\ar=\ar \int_{\lambda x}^\infty \mathbf{P}\big( X_{t-s}\in x-z-w+\varDelta_\delta,\,
 	\tau_{-z-w}^->t-s \big)
 	\nu(dw)\cr
 	\ar=\ar \int_{\lambda x}^\infty \mathbf{P}\big( X_{t-s}\in x-z-w+\varDelta_\delta,\, \underline{X}_{t-s}\geq -z-w  \big)
 	\nu(dw) = \mathbf{E}\Big[ \nu\big( D_\delta(t-s,x,z) \big) \Big],
 	\eeqnn
  where $D_\delta(t-s,x,z):= (\lambda x,\infty)\cap [-z-\underline{X}_{t-s},\infty ) \cap \big( x-z-X_{t-s} +\varDelta_\delta \big)$. 
 Consequently, 
  \beqlb\label{eqn.312}
  \bigg|\frac{ \underline{n}\big( \# \mathcal{B}_t^{\lambda x},\,\epsilon_t\in x+\varDelta_\delta,\, \zeta>t \big) }{ \nu \big(  x+\varDelta_\delta \big)}-\underline{n}(t\wedge\zeta) \bigg| 
  \ar=\ar \bigg|  \int_0^t ds \int_0^\infty  \frac{\mathbf{E}\big[ \nu\big( D_\delta(t-s,x,z) \big) \big]}{\nu \big(  x+\varDelta_\delta \big) } \,\underline{n}_s(dz) -\underline{n}(t\wedge\zeta) \bigg|\cr
  \ar=\ar \bigg|  \int_0^t ds \int_0^\infty \Big(  \frac{\mathbf{E}\big[ \nu\big( D_\delta(t-s,x,z) \big) \big]}{\nu \big(  x+\varDelta_\delta \big) } -1\Big) \,\underline{n}_s(dz)   \bigg|\cr
  \ar\leq\ar \int_0^t ds \int_0^\infty \bigg|  \mathbf{E}\Big[ \frac{ \nu\big( D_\delta(t-s,x,z) \big) }{\nu \big(  x+\varDelta_\delta \big) } \Big]-1\bigg| \,\underline{n}_s(dz)  . 
  \eeqlb 
  Observe that the interval $D_\delta(t-s,x,z)$ has left endpoint  larger than $\lambda x$ and   length  smaller than $\delta$. 
  By \eqref{eqn.30} and the monotonicity of $\overline{\nu}$, there exist some constants $C,x_0>0$ such that
  \beqlb\label{eqn.4020}
   \sup_{x\geq x_0}\sup_{s\in[0,t]}\sup_{z\geq 0} \frac{ \nu\big( D_\delta(t-s,x,z) \big) }{\nu \big(  x+\varDelta_\delta \big) } \leq C \quad a.s.
   \quad \mbox{and}\quad 
   \sup_{x\geq x_0}\sup_{s\in[0,t]}\sup_{z\geq 0}\mathbf{E}\bigg[ \frac{ \nu\big( D_\delta(t-s,x,z) \big) }{\nu \big(  x+\varDelta_\delta \big) } \bigg] \leq C. 
  \eeqlb 
  This permits us to take the limits of both sides of \eqref{eqn.312} as $x\to\infty$ and then invoke the dominated convergence theorem to deduce \eqref{eqn.311}, once the following holds for any fixed $z\ge0$ and $s\in[0,t]$:
  \beqnn
    \limsup_{x\to\infty} \Big|  \frac{  \nu\big( D_\delta(t-s,x,z) \big) }{\nu \big(  x+\varDelta_\delta \big) } -1\Big|  \overset{\rm a.s.}=0. 
  \eeqnn
  Actually, note that $x-z-X_{t-s}\geq \lambda x$ and $ x-z-X_{t-s}\geq-z-\underline X_{t-s}$ for sufficiently large $x$, we have  $D_\delta(t-s,x,z) = x-z-X_{t-s}+ \varDelta_\delta$ and then the preceding desired limit follows directly from  Condition~\ref{Main.Condition}. 
 This proves \eqref{eqn.311} and therefore \eqref{eqn.301}. 
 \qed 
 
 \medskip
 
 \textbf{\textit{Proof of Theorem~\ref{MainThm00}.}}   
 For any fixed $\delta>0$, the asymptotic equivalence \eqref{eqn.MainThm00} is obtained directly by substituting the two limits established in Propositions~\ref{Prop.501} and~\ref{Prop.502} into the right-hand side of \eqref{eqn.5011}. We now extend this to hold uniformly for $\delta\in[\delta_0,\infty]$.
 For any large integer $n\geq 1/\delta_0$, we have 
 \beqnn
   \frac{1}{ \underline{n}(\zeta \wedge t)}\sum_{k=0}^{[n\delta]-1}\underline{n}_t\Big(x+\frac{k}{n}+ \varDelta_{1/n} \Big) 
   \leq \frac{\underline{n}_t\big(x+\varDelta_\delta \big) }{ \underline{n}(\zeta \wedge t)} \leq  \frac{1}{ \underline{n}(\zeta \wedge t)}\sum_{k=0}^{[n\delta]}\underline{n}_t\Big(x+\frac{k}{n}+ \varDelta_{1/n} \Big). 
 \eeqnn 
 Applying the asymptotic equivalence \eqref{eqn.MainThm00} with $\delta=1/n$, we obtain a positive function $\rho_n$ on $\mathbb{R}_+$ that decreases to $0$ at infinity and satisfies, for all $x>0$ and $k\ge0$,
 \beqnn
  \big(1-\rho_n(x) \big) \cdot   \nu\Big( x+\frac{k}{n}+ \varDelta_{1/n} \Big) 
  \leq \frac{\underline{n}_t\big(x+k/n+ \varDelta_{1/n} \big) }{ \underline{n}(\zeta \wedge t)}
  \leq \big(1+\rho_n(x) \big) \cdot   \nu\Big( x+\frac{k}{n}+ \varDelta_{1/n} \Big).
 \eeqnn
 Combining these two estimates together yields
 \beqnn
 \frac{\underline{n}_t\big(x+\varDelta_\delta \big) }{ \underline{n}(\zeta \wedge t)   \nu\big( x + \varDelta_{\delta} \big)} \leq  \frac{\big(1+\rho_n(x) \big)}{  \nu\big( x + \varDelta_{\delta} \big)}  \sum_{k=0}^{[n\delta]}  \nu\Big( x+\frac{k}{n}+ \varDelta_{1/n} \Big)
 \leq \big(1+\rho_n(x) \big) \cdot   \frac{\nu\big( x + \varDelta_{\delta+1/n} \big)}{\nu\big( x + \varDelta_{\delta} \big)}.
 \eeqnn
 In view of the monotonicity of $\overline{\nu}(x+\varDelta_\delta)$ in $\delta$, together with \eqref{eqn.30}, this yields 
 \beqnn
  \limsup_{x\to\infty} \sup_{\delta\in [\delta_0,\infty]}\frac{\underline{n}_t\big(x+\varDelta_\delta \big) }{ \underline{n}(\zeta \wedge t)   \nu\big( x + \varDelta_{\delta} \big)} 
  \ar\leq \ar   \limsup_{x\to\infty} \sup_{\delta\in [\delta_0,\infty]}\frac{\nu\big( x + \varDelta_{\delta+1/n} \big)}{\nu\big( x + \varDelta_{\delta} \big)}\cr
  \ar\leq\ar 1+\limsup_{x\to\infty} \sup_{\delta\in [\delta_0,\infty]}\frac{\nu\big( x +\delta+ \varDelta_{1/n} \big)}{\nu\big( x + \varDelta_{\delta_0} \big)}  \leq  1+\frac{1/n}{\delta_0} .
 \eeqnn
 On the other hand, an analogous argument gives 
 \beqnn
 \liminf_{x\to\infty} \inf_{\delta\in [\delta_0,\infty]}\frac{\underline{n}_t\big(x+\varDelta_\delta \big) }{ \underline{n}(\zeta \wedge t)   \nu\big( x + \varDelta_{\delta} \big)} \geq    \liminf_{x\to\infty}  \inf_{\delta\in [\delta_0,\infty]} \frac{\nu\big( x + \varDelta_{\delta-1/n} \big)}{\nu\big( x + \varDelta_{\delta} \big)} \geq  1-\frac{1/n}{\delta_0} .
 \eeqnn
 Combining these two bounds together and then letting $n\to\infty$ yields the desired uniform asymptotic equivalence \eqref{eqn.MainThm00}.
 \qed

 \medskip

 \textbf{\textit{Proof of Corollary~\ref{Main.Corollary01}.}} 
 	In view of \eqref{eqn.30111}, we first have for $x>y$,
 	\beqnn
 	\mathbf{P}_y\big( X_t\in x+\varDelta_\delta,\tau_0^->t \big) 
 	\ar=\ar \overline{\tt d}\cdot \underline{n}_t\big(x-y+\varDelta_\delta \big) 
 	+ \int_0^t ds\int_0^y \underline{n}_{t-s}\big( x+z-y+\varDelta_\delta \big)\, \overline{n}_s(dz). 
 	\eeqnn
 	Integrating \eqref{eqn2.9} and then applying Fubini's theorem also gives
 	\beqnn
 	\mathbf{E}_y\big[t\wedge\tau_0^-\big]
 	= \int_0^t \mathbf{P}_y(\tau_0^->s)\, ds 
 	\ar=\ar
 	\overline{\tt d} \int_0^t  \underline n(\zeta>s)ds
 	+ \int_0^t ds\int_0^s dr\int_0^y  \underline n(\zeta>s-r)\, \overline{n}_r(dz)
 	\cr\cr
 	\ar=\ar
 	\overline{\tt d}\cdot 
 	\underline n(\zeta\wedge t)
 	+  \int_0^t dr\int_0^y  \overline{n}_r(dz) \int_r^t \underline{n}(\zeta>s-r)\,ds
 	\cr\cr
 	\ar=\ar \overline{\tt d}\cdot \underline{n}(\zeta\wedge t)
 	+  \int_0^t ds  \int_0^y \underline{n}\big( \zeta\wedge(t-s)\big) \, \overline{n}_s(dz).
 	\eeqnn
 	Combining these two equalities together and then using the triangle inequality yields 
 	\beqlb\label{eqn.3005}
 	\lefteqn{\bigg| \frac{\mathbf{P}_y\big( X_t\in x+\varDelta_\delta,\tau_0^->t \big)}{ \nu\big(x+\varDelta_\delta\big)}- \mathbf{E}_y\big[ t\wedge\tau_0^- \big]\bigg|}\cr
 	\ar\leq\ar  \overline{\tt d}\cdot \bigg| \frac{\underline{n}_t\big(x-y+\varDelta_\delta \big) }{\nu\big(x+\varDelta_\delta\big)} -\underline{n}(\zeta\wedge t)\bigg|
 	+ \int_0^t ds\int_0^y \bigg|\frac{\underline{n}_{t-s}\big( x+z-y+\varDelta_\delta \big)}{\nu\big(x+\varDelta_\delta\big)}- \underline{n}\big( \zeta\wedge(t-s)\big)\bigg|\, \overline{n}_s(dz). \qquad
 	\eeqlb
 	
 	We first prove the desired convergence for any fixed $\delta>0$. For the first term on the right-hand side, applying the triangle inequality gives
 	\beqnn
 	\bigg|
 	\frac{\underline{n}_t\big(x-y+\varDelta_\delta \big)}
 	{\nu\big(x+\varDelta_\delta\big)}
 	-\underline{n}(\zeta\wedge t)
 	\bigg|
 	\ar\leq\ar
 	\bigg|
 	\frac{\underline{n}_t\big(x-y+\varDelta_\delta \big)}
 	{\nu\big(x-y+\varDelta_\delta\big)}
 	-\underline{n}(\zeta\wedge t)
 	\bigg|
 	\cdot
 	\frac{\nu\big(x-y+\varDelta_\delta\big)}
 	{\nu\big(x+\varDelta_\delta\big)}
 	\cr\cr
 	\ar\ar
 	+
 	\bigg|
 	\frac{\nu\big(x-y+\varDelta_\delta\big)}
 	{\nu\big(x+\varDelta_\delta\big)}
 	-1
 	\bigg|
 	\cdot \underline{n}(\zeta\wedge t),
 	\eeqnn
 	which vanishes as $x\to\infty$ by Theorem~\ref{MainThm00} and Condition~\ref{Main.Condition}. 
 	Moreover, a similar argument gives that, for any $s\in[0,t]$ and $z\in[0,y]$,
 	\beqlb\label{eqn.3008}
 	\lim_{x\to\infty}
 	\bigg|
 	\frac{\underline{n}_{t-s}\big( x-y+z+\varDelta_\delta \big)}
 	{\nu\big(x+\varDelta_\delta\big)}
 	-\underline{n}\big( \zeta\wedge(t-s)\big)
 	\bigg|
 	=0 .
 	\eeqlb
 	In view of \eqref{eqn.312}-\eqref{eqn.4020} and then using Condition~\ref{Main.Condition}, we have that for all sufficiently large $x$,
 	\beqnn
 	\lefteqn{
 		\sup_{s\in[0,t]} \sup_{z\in [0,y]}
 		\underline{n}\big(
 		\epsilon_{t-s}\in x-y+z+\varDelta_{\delta},\,
 		\mathcal{J}^{\lambda (x-y+z)}\leq t-s,\,
 		\zeta> t-s
 		\big)
 	}\cr\cr
 	\ar=\ar
 	\sup_{s\in[0,t]} \sup_{z\in [0,y]}
 	\underline{n}\big(
 	\epsilon_s\in x-z+\varDelta_{\delta},\,
 	\mathcal{J}^{\lambda (x-z)}\leq s,\,
 	\zeta> s
 	\big)
 	\cr
 	\ar\leq\ar
 	\sup_{s\in[0,t]} \sup_{z\in [0,y]}
 	\underline{n}\big(
 	\#\mathcal{B}_t^{\lambda (x-z)},\,
 	\epsilon_s\in x-z+\varDelta_{\delta},\,
 	\zeta>s
 	\big)
 	\leq C\cdot \nu\big(x+\varDelta_{\delta}\big).
 	\eeqnn
 	On the other hand, similarly as in the proof of Proposition~\ref{Prop.501}, we also can prove that for all sufficiently large $x$,
 	\beqnn
 	\sup_{s\in[0,t]} \sup_{z\in [0,y]}
 	\underline{n}\big(
 	\epsilon_{t-s}\in x-y+z+\varDelta_{\delta},\,
 	\mathcal{J}^{\lambda (x-y+z)}> t-s,\,
 	\zeta> t-s
 	\big)
 	\leq
 	C\cdot \nu\big(x+\varDelta_{\delta}\big)
 	\cdot \underline{n}(\overline\epsilon \geq \eta x/2).
 	\eeqnn
 	Note $\underline{n}(\overline\epsilon \geq \eta x/2)\to0$ as $x\to\infty$. Combining these two upper bounds gives that for all sufficiently large $x$,
 	\beqnn
 	\sup_{s\in[0,t]} \sup_{z\in [0,y]}
 	\underline{n}\big(
 	\epsilon_{t-s}\in x-y+z+\varDelta_{\delta},\,
 	\zeta> t-s
 	\big)
 	\leq C\cdot \nu\big(x+\varDelta_{\delta}\big).
 	\eeqnn
 	Armed with this and \eqref{eqn.3008}, we use the dominated convergence theorem to obtain
 	\beqnn
 	\lefteqn{\lim_{x\to\infty}\int_0^t ds\int_0^y \bigg|\frac{\underline{n}_{t-s}\big( x+z-y+\varDelta_\delta \big)}{\nu\big(x+\varDelta_\delta\big)}- \underline{n}\big( \zeta\wedge(t-s)\big)\bigg|\, \overline{n}_s(dz) }\cr 
 	\ar\leq\ar  \int_0^t ds\int_0^y \lim_{x\to\infty}\bigg|\frac{\underline{n}_{t-s}\big( x+z-y+\varDelta_\delta \big)}{\nu\big(x+\varDelta_\delta\big)}- \underline{n}\big( \zeta\wedge(t-s)\big)\bigg|\, \overline{n}_s(dz) =0.
 	\eeqnn
    Substituting this and \eqref{eqn.3008} into the right-hand side of \eqref{eqn.3005} proves the desired convergence for any fixed $\delta>0$. 
    We now extend this to hold uniformly for $\delta\in[\delta_0,\infty]$. 
    For any large integer $n\geq1/\delta_0$,  
    \beqnn
    %\sum_{k=0}^{[n\delta]-1} \mathbf{P}_y\big( X_t\in x+k/n+\varDelta_{1/n},\tau_0^->t \big)  \ar\leq\ar 
    \mathbf{P}_y\big(
    X_t\in x+\varDelta_\delta,\tau_0^->t
    \big) 
    \ar\leq\ar
    \sum_{k=0}^{[n\delta]}
    \mathbf{P}_y\big(
    X_t\in x+k/n+\varDelta_{1/n},\tau_0^->t
    \big).
    \eeqnn
    By the preceding convergence with $\delta=1/n$, there exists  a positive function $\rho_n$ on $\mathbb{R}_+$ decreasing to $0$ at infinity such that, for all $x>y$ and $k\geq0$,
    \beqnn
%    \big(1-\rho_n(x)\big)\,
%    \mathbf{E}_y\big[t\wedge\tau_0^-\big]\,
%    \nu\big(x+k/n+\varDelta_{1/n}\big)
%    \ar\leq\ar
    \mathbf{P}_y\big(
    X_t\in x+k/n+\varDelta_{1/n},\tau_0^->t
    \big) 
    \ar\leq\ar
    \big(1+\rho_n(x)\big)\,
    \mathbf{E}_y\big[t\wedge\tau_0^-\big]\,
    \nu\big(x+k/n+\varDelta_{1/n}\big).
    \eeqnn
    Putting these two estimates together yields
    \beqlb\label{eqn.3022}
    \frac{
    	\mathbf{P}_y\big(
    	X_t\in x+\varDelta_\delta,\tau_0^->t
    	\big)}
    {\mathbf{E}_y\big[t\wedge\tau_0^-\big]\,
    	\nu\big(x+\varDelta_\delta\big)}
    \ar\leq\ar
    \big(1+\rho_n(x)\big)
    \frac{
    	\nu\big(x+\varDelta_{\delta+1/n}\big)}
    {\nu\big(x+\varDelta_\delta\big)}. 
    \eeqlb
  Since $\delta\to \nu\big(x+\varDelta_\delta\big)$ is non-decreasing, by \eqref{eqn.30} we have 
    \beqnn
    	\limsup_{x\to\infty}
    	\sup_{\delta\in[\delta_0,\infty)}
    	\frac{
    		\mathbf{P}_y\big(
    		X_t\in x+\varDelta_\delta,\tau_0^->t
    		\big)}
    	{\mathbf{E}_y\big[t\wedge\tau_0^-\big]\,
    		\nu\big(x+\varDelta_\delta\big)} 
    \ar\leq\ar
    1+
    \limsup_{x\to\infty}
    \sup_{\delta\in[\delta_0,\infty)}
    \frac{
    	\nu\big(x+\delta+\varDelta_{1/n}\big)}
    {\nu\big(x+\varDelta_{\delta_0}\big)}
    \cr\cr
    \ar\leq\ar
    1+
    \limsup_{x\to\infty}
    \sup_{u\geq x}
    \frac{\nu\big(u+\varDelta_{1/n}\big)}
    {\nu\big(x+\varDelta_{1/n}\big)}
    \cdot
    \frac{\nu\big(x+\varDelta_{1/n}\big)}
    {\nu\big(x+\varDelta_{\delta_0}\big)}
    \leq 
    1+\frac{1}{n\delta_0}.
    \eeqnn
    On the other hand, an analogous argument gives
    \beqnn
    	\liminf_{x\to\infty}
    	\inf_{\delta\in[\delta_0,\infty)}
    	\frac{
    		\mathbf{P}_y\big(
    		X_t\in x+\varDelta_\delta,\tau_0^->t
    		\big)}
    	{\mathbf{E}_y\big[t\wedge\tau_0^-\big]\,
    		\nu\big(x+\varDelta_\delta\big)}
    \geq
    \liminf_{x\to\infty}
    \inf_{\delta\in[\delta_0,\infty)}
    \frac{
    	\nu\big(x+\varDelta_{\delta-1/n}\big)}
    {\nu\big(x+\varDelta_\delta\big)}
    \geq
    1-\frac{1}{n\delta_0}.
    \eeqnn
    Combining these two estimates together and then letting $n\to\infty$ yields the desired uniform asymptotic equivalence for $\delta\in[\delta_0,\infty)$. 
    For $\delta=\infty$, the right-hand side of \eqref{eqn.3022} reduces to $1+\rho_n(x)$ which goes to $1$ as $x\to\infty$.  
    The corresponding lower bound follows by the same argument.
    This completes the whole proof.
 	\qed

  \section{Proofs of Theorem~\ref{MainThm01} and Corollary~\ref{Corollary.101}, \ref{Main.Corollary02}} 
\label{Sec.Proof02}
\setcounter{equation}{0}

 In this section, we first provide a detailed proof of Theorem~\ref{MainThm01}, which characterizes the precise local behavior of the entrance law of long-lived excursions. We then prove Corollaries~\ref{Corollary.101} and \ref{Main.Corollary02} by combining the path decomposition \eqref{eqn2.9} with Theorem~\ref{MainThm01}.
 Our proofs are strongly inspired by the argument used in the proof of Theorem~1.7 in \cite{CuiWuXu2025}, whose key observation is that, for an excursion to survive for a long time, a large jump must typically occur at an early stage of the excursion. This mechanism is formalized in the following two lemmas, whose proofs are deferred to Section~\ref{Sec.401}.

\begin{lemma}\label{Lemma.401}
	For any constants $\delta>0$ and $\eta \in \big(0,\frac{\alpha\wedge 2-1}{8(\alpha+1)}\big)$, we have  
	\beqnn 
	\lim_{T\to\infty} \limsup_{t\to\infty}\sup_{x\geq 0} \frac{\underline{n}\big(\epsilon_t\in x+\varDelta_\delta, \mathcal{J}^{\eta(\beta t+x)}>T,\zeta >t\big)}{\overline{\nu}(\beta t+x)/ (\beta t +x) } = 0. 
	\eeqnn
\end{lemma}  

\begin{lemma}\label{Lemma.402}
	For any constants $\delta,T>0$ and $\eta\in(0,\frac{\alpha\wedge 2-1}{8(\alpha+1)})$, we have  
	\beqlb\label{eqn4.26} 
	\limsup_{t\to\infty}\sup_{x\geq0} \bigg| \frac{ \underline{n}\big( \epsilon_t\in x+\varDelta_\delta,\mathcal{J}^{\eta(\beta t+x)}\leq T,\zeta>t \big)}{\alpha\cdot \overline\nu(\beta t+x)/(\beta t+x)} -\frac{\underline n(\zeta\wedge T)}{ \underline n(\zeta)  } \int_x^{x+\delta}V(z)\, dz  \bigg| = 0. 
	\eeqlb 
\end{lemma}

\medskip

 \textbf{\textit{Proof of Theorem~\ref{MainThm01}.}}  
 We only provide the detailed proof of the first uniform asymptotic equivalence in \eqref{eqn.MainThm01}. The second follows, uniformly in $x\geq 0$ and $\delta\in[\delta_0,\infty]$, from the standard discretization argument, together with \eqref{eqn.30} and the monotonicity and boundedness of $V$. We omit the details.
 
 We first prove that the first asymptotic equivalence \eqref{eqn.MainThm01} holds uniformly in $x\geq0$ for each fixed $\delta>0$. 
 By \eqref{eqn.30}, we have as $t\to\infty$,
 \beqnn
  \sup_{x\geq 0}\sup_{z\in x+\varDelta_\delta}\bigg|  \frac{\overline{\nu}(\beta t +z)/(\beta t +z)}{\overline{\nu}(\beta t +x)/(\beta t +x)} -1\bigg| \to 0.
 \eeqnn
 Since $V(x)>0$ when $x>0$ and $V(\infty)<\infty$; see \eqref{eqn.2011},  then
 \beqlb\label{eqn.201}
 0<\inf_{x\geq 0} \int_x^{x+\delta} V(z)\,dz \leq \sup_{x\geq 0} \int_x^{x+\delta} V(z)\,dz\leq V(\infty)\cdot \delta<\infty. 
 \eeqlb
 These two results indicate that we just need to prove that as $t\to\infty$, 
\beqnn
\sup_{x\geq 0}\bigg| \frac{\underline{n}_t\big( x+\varDelta_\delta\big)}{\alpha\cdot  \overline{\nu}(\beta t +x)/(\beta t +x) }-   \int_x^{x+\delta} V(z)\,dz\bigg|\to 0. 
\eeqnn
For any  $T>0$ and $\eta \in \big(0, \frac{\alpha\wedge 2+1}{8(\alpha+1)} \big)$, by the triangle inequality this supremum can be bounded by 
\beqnn
\sup_{x\geq 0} \frac{ \underline{n}\big(\epsilon_t\in x+\varDelta_\delta,\,\mathcal{J}^{\eta(\beta t+x)}>T,\, \zeta >t\big)}{\alpha\cdot  \overline{\nu}(\beta t +x)/(\beta t +x) } 
+ \sup_{x\geq 0}\bigg| \frac{\underline{n}\big(\epsilon_t\in x+\varDelta_\delta,\,\mathcal{J}^{\eta(\beta t+x)}\leq T,\, \zeta >t\big)}{\alpha\cdot  \overline{\nu}(\beta t +x)/(\beta t +x) }-\int_x^{x+\delta} V(z)\,dz\bigg| ,
\eeqnn
which goes to $0$ as $t\to\infty$ and then $T\to \infty$ by Lemma~\ref{Lemma.401} and \ref{Lemma.402}. 
Hence \eqref{eqn.MainThm01} holds uniformly in $x\geq0$ for each fixed $\delta>0$. 
We now extend it to hold uniformly for  $\delta\in [\delta_0,\infty]$. 
For any large integer $n\geq 1/\delta_0$, we have 
\beqlb\label{eqn.4001}
\sum_{k=0}^{[n\delta]-1} \underline{n}_t\big( x+k/n+\varDelta_{1/n}\big)\leq \underline{n}_t\big( x+\varDelta_\delta\big) 
\ar\leq\ar \sum_{k=0}^{[n\delta]} \underline{n}_t\big( x+k/n+\varDelta_{1/n}\big).
\eeqlb
In view of our previous result, there exists a positive function $\rho_n$ on $\mathbb{R}_+$ that decreases to $0$ at infinity and satisfies,   for all $x\geq 0$ and $k\geq 0$, 
\beqnn
\underline{n}_t\big(x+k/n+\varDelta_{1/n}\big) 
\ar\geq\ar \big( 1-\rho_n(t) \big)  \cdot  \alpha\cdot\int_{x+k/n}^{x+k/n+1/n} \frac{\overline{\nu}(\beta t+z)}{\beta t+z} \cdot V(z)\, dz 
\eeqnn
and 
\beqnn
\underline{n}_t\big(x+k/n+\varDelta_{1/n}\big) 
\ar\leq\ar \big( 1+\rho_n(t) \big)  \cdot  \alpha\cdot\int_{x+k/n}^{x+k/n+1/n} \frac{\overline{\nu}(\beta t+z)}{\beta t+z} \cdot V(z)\, dz.
\eeqnn
Substituting these two estimates into the lower and upper bounds in \eqref{eqn.4001} respectively yields that 
\beqnn
\underline{n}_t\big( x+\varDelta_\delta\big) 
\ar\geq\ar  \alpha\cdot \big( 1-\rho_n(t) \big)  \cdot \sum_{k=0}^{[n\delta]-1} \int_{x+k/n}^{x+k/n+1/n} \frac{\overline{\nu}(\beta t+z)}{\beta t+z} \cdot V(z)\, dz\cr
\ar\geq\ar \alpha\cdot \big( 1-\rho_n(t) \big)  \cdot \bigg( \int_{x}^{x+\delta} \frac{\overline{\nu}(\beta t+z)}{\beta t+z} \cdot V(z)\, dz - \int_{x+\delta-1/n}^{x+\delta} \frac{\overline{\nu}(\beta t+z)}{\beta t+z} \cdot V(z)\, dz \bigg)
\eeqnn
and
\beqnn
\underline{n}_t\big( x+\varDelta_\delta\big) 
\ar\leq\ar  \alpha\cdot \big( 1+\rho_n(t) \big)  \cdot \sum_{k=0}^{[n\delta]} \int_{x+k/n}^{x+k/n+1/n} \frac{\overline{\nu}(\beta t+z)}{\beta t+z} \cdot V(z)\, dz\cr
\ar\leq\ar \alpha\cdot \big( 1+\rho_n(t) \big)  \cdot  \bigg(   \int_{x}^{x+\delta} \frac{\overline{\nu}(\beta t+z)}{\beta t+z} \cdot V(z)\, dz
+  \int_{x+\delta}^{x+\delta+1/n} \frac{\overline{\nu}(\beta t+z)}{\beta t+z} \cdot V(z)\, dz\big).
\eeqnn
Since $\rho_n(t)\to0$ as $t\to\infty$, these two estimates yield that the desired asymptotic equivalence \eqref{eqn.MainThm01} with $\delta=\infty$ holds uniformly for $x\geq0$. 
Moreover, since $\delta\mapsto  \underline{n}_t\big( x+\varDelta_\delta\big)$ is non-decreasing, our first result shows that for all sufficiently large $t$ and all $x\geq 0$,
\beqnn
\inf_{\delta\in [\delta_0,\infty)} \underline{n}_t\big( x+\varDelta_\delta\big)
 =  \underline{n}_t\big( x+\varDelta_{\delta_0}\big) \geq C\cdot  \frac{\overline{\nu}(\beta t+x)}{\beta t+x}. 
\eeqnn 
Hence the preceding lower and upper bounds also indicate that \eqref{eqn.MainThm01} holds uniformly in $x\geq0$ and  $\delta\in [\delta_0,\infty)$ if we can identify
\beqlb\label{eqn.4002}
\limsup_{n\to\infty} \limsup_{t\to\infty} \sup_{x\geq 0}\sup_{\delta\in [\delta_0,\infty)}
\frac{\beta t+x}{\overline{\nu}(\beta t+x)} \int_{x+\delta-1/n}^{x+\delta+1/n} \frac{\overline{\nu}(\beta t+z)}{\beta t+z} \cdot V(z)\, dz =0.
\eeqlb
Actually, since $V(z)\leq V(\infty)<\infty$ and $\overline\nu \in \mathrm{RV}^\infty_{-\alpha}$ is non-increasing, we have as $t\to\infty$, 
\beqnn
\sup_{x\geq 0}\sup_{\delta\in [\delta_0,\infty)}
\frac{\beta t+x}{\overline{\nu}(\beta t+x)} \int_{x+\delta-\frac{1}{n}}^{x+\delta+\frac{1}{n}} \frac{\overline{\nu}(\beta t+z)}{\beta t+z} \cdot V(z)\, dz  
\leq \frac{2V(\infty)}{n} \cdot \sup_{x\geq \beta t} \frac{x}{\overline{\nu}(x)}\frac{\overline{\nu}(x-1)}{x-1}
\to \frac{2V(\infty)}{n},
\eeqnn
which vanishes as $n\to\infty$  and then the limit \eqref{eqn.4002} holds. This completes the proof.
\qed

\textbf{\textit{Proof of Corollary~\ref{Corollary.101}.}}  
For each integer $n\geq1$, define the step-function approximation of $f$:
\beqnn
f_n(z):= \sum_{k=0}^\infty f(k/n) \cdot \mathbf{1}_{\{ z\in(\frac{k}{n},\frac{k+1}{n}] \}},\quad z\geq0.
\eeqnn
Since $f$ is continuous, it is uniformly continuous on every compact interval. Hence, as $n\to\infty$,
\beqlb\label{eqn.302}
\sup_{z\in(x,y]}|f(z)-f_n(z)| \to 0. 
\eeqlb
 By the triangle inequality,  
 \beqlb\label{eqn.4003}
	\lefteqn{\bigg| \frac{\beta t}{\overline{\nu}(\beta t)} \int_x^y f(z)\,\underline{n}_t(dz) - \alpha\int_x^y f(z)V(z)\,dz \bigg|}\cr
 \ar\leq\ar
	\frac{\beta t}{\overline{\nu}(\beta t)}
	\int_x^y |f(z)-f_n(z)|\,\underline{n}_t(dz) +
	\alpha\int_x^y |f(z)-f_n(z)|V(z)\,dz\cr
	\ar\ar +
	\bigg|
	\frac{\beta t}{\overline{\nu}(\beta t)}
	\int_x^y f_n(z)\,\underline{n}_t(dz)
	-
	\alpha\int_x^y f_n(z)V(z)\,dz
	\bigg| .
	\eeqlb
	For any $y>x\geq0$ and $\delta>0$, by
	$\overline{\nu}\in\mathrm{RV}_{-\alpha}^{\infty}$ and
	Theorem~\ref{MainThm01}, we have as $t\to\infty$,
	\beqlb\label{eqn.4004}
	\sup_{z\in[x,y]}
	\bigg|
	\frac{\beta t}{\overline{\nu}(\beta t)}
	\cdot
	\frac{\overline{\nu}(\beta t+z)}{\beta t+z}
	-1
	\bigg|
	\to0
	\quad\mbox{and}\quad
	\sup_{z\in[x,y]}
	\bigg|
	\frac{\beta t}{\overline{\nu}(\beta t)}
	\cdot
	\underline{n}_t\big(z+\varDelta_\delta\big)
	-
	\alpha\int_z^{z+\delta}V(w)\,dw
	\bigg|
	\to0.
	\eeqlb
	Thus the first two terms on the right-hand side of
	\eqref{eqn.4003} can be bounded uniformly for all sufficiently large $t$ by
	\beqnn
	\sup_{z\in(x,y]}|f(z)-f_n(z)|
	\cdot
	\bigg(\frac{\beta t}{\overline{\nu}(\beta t)}
	\underline{n}_t\big(x+\varDelta_{y-x}\big)  +
	\alpha\int_x^y V(z)\,dz\bigg) \leq C\cdot \sup_{z\in(x,y]}\big|f(z)-f_n(z)\big|,
	\eeqnn
	 which tends to $0$ as $n\to\infty$ by \eqref{eqn.302}. It remains to prove that the third term also vanishes as
	$t\to\infty$. For fixed $n$, by the definition of $f_n$ and the triangle inequality,
	\beqnn
	\lefteqn{ \bigg| \frac{\beta t}{\overline{\nu}(\beta t)} \int_x^y f_n(z)\,\underline{n}_t(dz) - 	\alpha\int_x^y f_n(z)V(z)\,dz \bigg| }
	\cr
	\ar\leq\ar
	\Big|f\big([nx]/n\big)\Big|\cdot 
	\bigg|
	\frac{\beta t}{\overline{\nu}(\beta t)}
	\underline{n}_t\bigg(
	x+\varDelta_{\frac{[nx]+1}{n}-x}
	\bigg)
	-
	\alpha\int_x^{([nx]+1)/n}V(z)\,dz
	\bigg|
	\cr\cr
	\ar\ar+
	\sum_{k=[nx]+1}^{[ny]-1}
	\Big|f(k/n)\Big| \cdot
	\bigg|
	\frac{\beta t}{\overline{\nu}(\beta t)}
	\underline{n}_t\big(k/n+\varDelta_{1/n}\big)
	-
	\alpha\int_{k/n}^{(k+1)/n}V(z)\,dz
	\bigg|
	\cr\cr
	\ar\ar+
	\Big|f\big([ny]/n\big)\Big|\cdot 
	\bigg|
	\frac{\beta t}{\overline{\nu}(\beta t)}
	\underline{n}_t\bigg(
	[ny]/n+\varDelta_{y-[ny]/n}
	\bigg)
	-
	\alpha\int_{[ny]/n}^{y}V(z)\,dz
	\bigg|,
	\eeqnn 
	which goes to $0$, as $t\to\infty$, by applying the second limit in \eqref{eqn.4004} to each nonempty
	interval.
Combining the preceding estimates, first letting $t\to\infty$ and then $n\to\infty$, proves the desired convergence. This completes the proof. 
\qed

 We now turn to the proof of Corollary~\ref{Main.Corollary02}. 
  In view of \eqref{eqn.30111},  for any large $t,R>0$ with $t>2R$, we may decompose the targeted local probability as 
 \beqlb\label{eqn.605}
 \frac{\beta t+x}{\overline{\nu}(\beta t+x)}\cdot \mathbf{P}_y\big( X_t\in x+\varDelta_\delta,\, \tau_0^->t \big)  = \overline{\tt d}\cdot  I_1^{y,\delta}(t,x) + I_2^{y,\delta}(t,R,x) + I_3^{y,\delta}(t,R,x) +I_4^{y,\delta}(t,R,x), \quad
 \eeqlb
 with the four terms on the right-hand side given by
 \beqnn
 I_1^{y,\delta}(t,x)\ar:=\ar	\frac{\beta t +x}{\overline\nu(\beta t+x)}\cdot \underline{n}_t\big(  \big((x-y)^+,(x-y+\delta)^+] \big),\cr
 I_2^{y,\delta}(t,R,x)\ar:=\ar \frac{\beta t +x}{\overline\nu(\beta t+x)}  \int_0^R ds\int_0^y \underline{n}_{t-s}\big( \big((x-y+z)^+, (x-y+z+\delta)^+\big] \big) \, \overline{n}_s(dz) ,\cr
 I_3^{y,\delta}(t,R,x)\ar:=\ar \frac{\beta t +x}{\overline\nu(\beta t+x)} \int_{t-R}^t ds\int_0^y \underline{n}_{t-s}\big( \big((x-y+z)^+, (x-y+z+\delta)^+\big] \big) \, \overline{n}_s(dz) , \cr
 I_4^{y,\delta}(t,R,x)\ar:=\ar\frac{\beta t +x}{\overline\nu(\beta t+x)}\int_R^{t-R} ds\int_0^y \underline{n}_{t-s}\big( \big((x-y+z)^+, (x-y+z+\delta)^+\big] \big) \, \overline{n}_s(dz) .
 \eeqnn
 The asymptotics of these four terms are established in the next four lemmas, respectively. Their proofs are deferred to Section~\ref{Sec.402}.

 \begin{lemma}\label{Lemma.601}
	For any $\delta,y>0$, we have as $t\to\infty$, uniformly in $x\geq 0$,
	\beqnn
	I_1^{y,\delta}(t,x) 
    \to \widetilde{I}_1^{y,\delta}(x):= \alpha \cdot   \int_{(x-y)^+}^{(x-y+\delta)^+} V(u)\, du .
	\eeqnn
 \end{lemma}

 \begin{lemma}\label{Lemma.602}
	For any $\delta,y>0$ and $R>0$, we have as $t\to\infty$, uniformly in $x\geq 0$,
	\beqnn
	I_2^{y,\delta}(t,R,x) 
	\ar\to\ar  \widetilde{I}_2^{y,\delta}(R,x):= \alpha \cdot  \int_0^R ds\int_0^y    \overline{n}_s(dz)
	\int_{(x-y+z)^+}^{(x-y+z+\delta)^+}V(u)\,du  . 
	\eeqnn
 \end{lemma}

\begin{lemma}\label{Lemma.603}
	For any $\delta,y>0$ and $R>0$, we have as $t\to\infty$, uniformly in $x\geq 0$,
	\beqnn
	I_3^{y,\delta}(t,R,x) 
	\ar\to\ar \widetilde{I}_3^{y,\delta}(R,x):=  \alpha \cdot   \int_0^y\widehat{V}(u) \,du 
	\int_0^R\underline{n}_s\big( ((x-y+u)^+,(x-y+u+\delta)^+]\big)\,ds . 
	\eeqnn
\end{lemma}

\begin{lemma}\label{Lemma.604}
	For any $\delta,y>0$, we have  $I_4^{y,\delta}(t,R,x) \to0 $ uniformly in $x\geq 0$ as $t\to\infty$ and then $R\to\infty$.  
%	\beqnn
%	\limsup_{R\to\infty}\limsup_{t\to\infty}\sup_{x\geq 0} I_4^{y,\delta}(t,R,x)   =0 . 
%	\eeqnn
\end{lemma}

\textbf{\textit{Proof of Corollary~\ref{Main.Corollary02}.}} We establish here only that the asymptotic equivalence \eqref{eqn.Corollary02} is uniform in $x\ge0$ for each fixed $\delta>0$. The extension to uniformity in $\delta\in[\delta_0,\infty]$ follows from an argument analogous to the one employed in the proof of Theorem~\ref{MainThm01}. 

As in the proof of Theorem~\ref{MainThm01}, it suffices to prove that, as $t\to\infty$,
\beqlb\label{eqn.602}
\sup_{x\geq 0}\bigg|\frac{\beta t+x}{\overline{\nu}(\beta t+x)}\cdot \mathbf{P}_y\big( X_t\in x+\varDelta_\delta,\, \tau_0^->t \big)  - \alpha \cdot  \widehat{V}(y)\cdot \int_x^{x+\delta}  V(z)\, dz\bigg|\to 0.
\eeqlb
For any $R>0$, recall the three quantities   $\widetilde{I}_1^{y,\delta}(x)$,  $\widetilde{I}_2^{y,\delta}(R,x)$ and $\widetilde{I}_3^{y,\delta}(R,x)$ defined in Lemma~\ref{Lemma.601}-\ref{Lemma.603}, and also define 
\beqnn 
\widetilde{I}_4^{y,\delta}(R,x)  \ar:=\ar \alpha \cdot  \int_R^\infty ds\int_0^y    \overline{n}_s(dz) \int_{(x-y+z)^+}^{(x-y+z+\delta)^+}V(u)\,du,\cr
\widetilde{I}_5^{y,\delta}(R,x)  \ar:=\ar \alpha \cdot   \int_0^y\widehat{V}(u) \,du  \int_R^\infty\underline{n}_s\big( ((x-y+u)^+,(x-y+u+\delta)^+]\big)\,ds. 
\eeqnn
 Consider the following function 
\beqnn
F(z) := \alpha	\int_{(x-y+z)^+}^{(x-y+z+\delta)^+}V(u)\, du, \quad  0\leq z\leq y.
\eeqnn
One can identify that $F$ is absolutely continuous on $(0,y)$   with   Radon--Nikodym derivative 
\beqnn
 F'(z) \ar=\ar \alpha\Big(V\big((x-y+z+\delta)^+\big) -V\big( (x-y+z)^+\big)\Big)\cr 
 \ar=\ar \alpha\int_0^\infty \underline{n}_s\big( ((x-y+z)^+,(x-y+z+\delta)^+] \big)\,ds.
\eeqnn
This induces that 
\beqnn
\widetilde{I}_3^{y,\delta}(R,x) + \widetilde{I}_5^{y,\delta}(R,x)
\ar=\ar  \int_0^y \widehat V(z)F'(z)\,dz  = \int_0^y \widehat V(z)\,dF(z).
\eeqnn
In addition,  in view of \eqref{eqn.447},  we have
\beqnn
\overline{\tt d}\cdot  \widetilde{I}_1^{y,\delta}(x) +  \widetilde{I}_2^{y,\delta}(R,x) +  \widetilde{I}_4^{y,\delta}(R,x)	
\ar=\ar \overline{\tt d}\cdot F(0)+ \int_0^\infty ds 	\int_0^y F(z)\,\overline{n}_s(dz)
=   \int_{[0,y]}F(z)\widehat V(dz).
\eeqnn 
Combining these two equalities  together  yields 
\beqnn
\lefteqn{\overline{\tt d}\cdot  \widetilde{I}_1^{y,\delta}(x) + \widetilde{I}_2^{y,\delta}(R,x) + \widetilde{I}_3^{y,\delta}(R,x) +  \widetilde{I}_4^{y,\delta}(R,x)  +  \widetilde{I}_5^{y,\delta}(R,x) }\cr\cr
\ar=\ar  \int_{[0,y]}F(z)\widehat V(dz)+\int_0^y\widehat V(z)\,dF(z)
= \widehat V(y)\cdot F(y) =\alpha \cdot \widehat{V}(y)\cdot  \int_x^{x+\delta}V(u)du.
\eeqnn
Comparing this with \eqref{eqn.605} and appealing to Lemmas~\ref{Lemma.601}--\ref{Lemma.604}, we conclude that the limit \eqref{eqn.602} follows immediately provided the following two limits are proved, 
\beqlb\label{eqn.603}
\limsup_{R\to\infty}\sup_{x\geq 0} \widetilde{I}_4^{y,\delta}(R,x)   =0 
\quad \mbox{and}\quad 
\limsup_{R\to\infty}\sup_{x\geq 0} \widetilde{I}_5^{y,\delta}(R,x)  = 0 . 
\eeqlb
Indeed, since $V(z)\le V(\infty)<\infty$ for all $z\ge0$, we have
\beqnn
\sup_{x\geq 0} \widetilde{I}_4^{y,\delta}(R,x)  \ar\leq\ar \alpha \delta V(\infty) \int_R^\infty     \overline{n}_s\big([0,y]\big) \, ds,
\eeqnn
which tends to $0$ as $R\to\infty$, since $\int_0^\infty \overline{n}_s([0,y])\,ds<\infty$; see \eqref{eqn.447}. 
Moreover, we also have 
\beqnn
\sup_{x\geq 0} \widetilde{I}_5^{y,\delta}(R,x)  \leq \alpha  \int_0^y\widehat{V}(u) \,du \cdot \int_R^\infty     \underline{n}(\zeta >s) \, ds
\eeqnn
which also tends to $0$ as $R\to\infty$, since $\int_0^\infty \underline{n}(\zeta>s)\,ds = \underline{n}(\zeta)<\infty$; see \eqref{eqn.501}.
Thus the two limits in \eqref{eqn.603} hold, and consequently the limit \eqref{eqn.602} follows. This completes the proof. 
\qed

\subsection{Proofs of Lemma~\ref{Lemma.401} and \ref{Lemma.402}}
\label{Sec.401}

\begin{proposition}\label{Prop.401}
	For any $\lambda>0$ and $\eta \in (0,\frac{\alpha\wedge 2-1}{4(\alpha+1)})$, we have as $t\to\infty$,
	\beqlb 
	\sup_{x\geq-(\beta-\lambda)t}
	\frac{ \mathbf P(X_t\geq x,\mathcal{J}^{\eta(\beta t+x)}>t) }{ \overline{\nu}(\beta t+x)/(\beta t +x)}
	\to 0.
	\eeqlb
	
\end{proposition}
\proof Let $\widetilde{X}_t=X_t+\beta t$ for $t\geq 0$. 
The L\'evy--It\^{o}
decomposition theorem allows us to represent $\widetilde{X}$ as  
\beqnn
\widetilde X_t= M_t+Q_t^{(x)}  + \int_0^t \int_{\eta(\beta t+x)}^{\infty}
y\,N(ds,dy) - t\cdot \int_{\eta(\beta t+x)}^\infty y\,\nu(dy),
\eeqnn
where the two processes $M$ and $Q^{(x)}  $ are given by
\beqnn
M_t :=  \sigma B_t + \int_0^t
\int_{-\infty}^1 y\,\widetilde N(ds,dy)
\quad\mbox{and}\quad
Q_t^{(x)}  := \int_0^t \int_1^{\eta(\beta t+x)} y\,\widetilde{N}(ds,dy),\quad t\geq 0 . 
\eeqnn 
Here $B$ is a standard Brownian motion, $N(ds,dy)$ is a Poisson random measure on $(0,\infty)\times \mathbb{R}$ with intensity $ds\,\nu(dy)$ and $\widetilde{N}(ds,dy):=N(ds,dy)-ds\,\nu(dy)$. 
For any $t> 0$, conditioned on the event $\{\mathcal{J}^{\eta(\beta t+x)}>t\}$, there is no positive jump larger than $\eta(\beta t+x)$ before time $t$. 
This implies that $\widetilde{X}_t\le M_t+Q_t^{(x)}$, and therefore 
\beqnn
\mathbf{P}\big(X_t\geq x, \mathcal{J}^{\eta(\beta t+x)}>t \big)=
\mathbf{P}\big(\widetilde{X}_t\geq \beta t +x,\mathcal{J}^{\eta(\beta t+x)}>t\big) 
\leq 
\mathbf P\Big( M_t\geq \frac{\beta t+x}{2} \Big)
+ \mathbf P\Big(Q_t^{(x)}  \geq \frac{\beta t+x}{2}\Big).
\eeqnn
Consequently, it suffices to prove that as $t\to\infty$,
\beqlb\label{eqn.3001}
\sup_{x\geq-(\beta-\lambda)t} \frac{\mathbf P\big( M_t\geq (\beta t+x)/2 \big)}{\overline{\nu}(\beta t+x)/(\beta t +x)} \to 0
\quad \mbox{and} \quad
\sup_{x\geq-(\beta-\lambda)t} \frac{\mathbf P\big( Q_t^{(x)}  \geq (\beta t+x)/2 \big)}{\overline{\nu}(\beta t+x)/(\beta t +x)}\to 0. 
\eeqlb

\textbf{(i)} Since $M$ is a  L\'evy process with zero mean and all jumps bounded above by $1$, Theorem~25.17 in \cite[ p.~165]{Sato1999} implies that it has finite exponential moments of all non-negative orders, i.e., 
\beqnn
\varPsi_M(u):= \log\mathbf E\big[\exp\{uM_1\}\big] <\infty,\quad \forall\, u \geq 0
\eeqnn
and $\varPsi_M(0+)=\varPsi_M'(0+)=0$. Thus there exists a constant $u_\lambda>0$ such that $\varPsi_M(u_\lambda)\leq \lambda u_\lambda/4$. 
By Chebyshev's inequality, for all  $x\geq -(\beta-\lambda)t$,
\beqnn
\mathbf P\big( M_t\geq (\beta t+x)/2 \big)
\ar\leq\ar \mathbf{E}\Big[ \exp\big\{	u_\lambda M_t \big\} \Big] \cdot  \exp\Big\{	  - \frac{u_\lambda}{2}\cdot (\beta t+x)\Big\} \cr\cr
\ar=\ar  
\exp\Big\{	t\cdot \varPsi_M(u_\lambda) - \frac{u_\lambda}{2}\cdot (\beta t+x)\Big\}\cr \cr
\ar\leq\ar
\exp\Big\{ \frac{\lambda\cdot u_\lambda }{4}\cdot t
-\frac{u_\lambda}{2}\cdot (\beta t+x) \Big\} 
\leq 
\exp\Big\{
-\frac{u_\lambda }{4}\cdot (\beta t+x)
\Big\}.
\eeqnn 
Since $\beta t+x\ge\lambda t$ and $\overline{\nu}\in\mathrm{RV}_{-\alpha}^\infty$, we have, as $t\to\infty$, 
\beqnn
\sup_{x\geq-(\beta-\lambda)t}\frac{\mathbf{P}\big( M_t\geq (\beta t+x)/2 \big)}{ \overline{\nu}(\beta t+x)/(\beta t +x) } 
\leq   \sup_{x\geq-(\beta-\lambda)t}\frac{	\exp\big\{- u_\lambda/4\cdot (\beta t+x) \big\}}{ \overline{\nu}(\beta t+x)/(\beta t +x) }  \to 0 . 
\eeqnn 
%and then the first limit \eqref{eqn.3001} holds.

\textbf{(ii)}  For any integer $n\geq 1$,  the time-orthogonality of $\widetilde{N}(ds,dy)$ allows us to decompose $Q_t^{(x)}  $ into the following $n$ i.i.d. parts:
\beqnn
Q_{t,n,j}^{(x)}  := \int_{(j-1)t/n}^{jt/n} \int_1^{\eta(\beta t+x)} y\, \widetilde{N}(ds,dy),\quad j=1,\cdots, n.  
\eeqnn
Choose a constant $\kappa \in \big(4\eta(\alpha+1)+1,2\wedge \alpha\big)$. 
By the Fuk--Nagaev inequality; see Corollary 1.6 in
\cite{Nagaev1979},  there exists a constant $C>0$ such that for all $n\geq 1$, all sufficiently large $t$ and $x\geq -(\beta-\lambda)t$,
\beqlb\label{eqn.421}
\mathbf P\Big(Q_t^{(x)}  \geq \frac{\beta t+x}{2}\Big)
\ar=\ar \mathbf{P}\bigg(  \sum_{j=1}^nQ_{t,n,j}^{(x)}  \geq \frac{\beta t+x}{2} \bigg)\cr
\ar\leq\ar \sum_{j=1}^n \mathbf{P}\big(  Q_{t,n,j}^{(x)}  \geq \eta(\beta t+x) \big)
+ C\cdot \bigg( \frac{ 1}{ (\beta t+x)^\kappa} \cdot \sum_{j=1}^n \mathbf{E}\Big[ \big|Q_{t,n,j}^{(x)} \vee 0\big|^\kappa\Big] \bigg)^{1/(4\eta)} \cr
\ar=\ar n\cdot \mathbf{P}\big(  Q_{t,n,1}^{(x)} \geq \eta(\beta t+x) \big)
+  \frac{C}{ (\beta t+x)^{\frac{\kappa}{4\eta}}} \cdot \bigg(n\cdot \mathbf{E}\Big[ \big|Q_{t,n,1}^{(x)}\big|^\kappa\Big]\bigg)^{1/(4\eta)} . 
\eeqlb
Using the Burkholder-Davis-Gundy inequality, then the inequality $(a+b)^{\kappa/2}\leq a^{\kappa/2}+b^{\kappa/2}$ for all $a,b\geq 0$ and finally \eqref{eqn.LevyMeasure}, there exists a constant $C>0$ such that, for all $n\geq 1$, all sufficiently large $t$ and $x\geq -(\beta-\lambda)t$,
\beqnn
n\cdot \mathbf{E}\Big[ \big|Q_{t,n,1}^{(x)}\big|^\kappa\Big] 
\ar\leq\ar C\cdot  n\cdot \mathbf E\bigg[ \Big(  \int_0^{t/n} \int_1^{\eta(\beta t+x)} y^2\,N(ds,dy)\Big)^{\kappa/2} \bigg]\cr
\ar\leq\ar C\cdot n\cdot \mathbf E\bigg[  \int_0^{t/n} \int_1^{\eta(\beta t+x)} y^\kappa\,N(ds,dy) \bigg] \cr
\ar=\ar C\cdot t\cdot \int_1^{\eta(\beta t+x)}y^\kappa\,\nu(dy) \cr\cr
\ar\leq\ar C\cdot  t . 
\eeqnn
Moreover, the definition of $ Q_{t,n,1}^{(x)}$ and the fact $\nu\big((1,\eta(\beta t+x)] \big)<\infty$ show that 
\beqnn
Q_{t,n,1}^{(x)} = \int_0^{t/n} \int_1^{\eta(\beta t+x)} y\,N(ds,dy) - \frac{t}{n}\cdot  \int_1^{\eta(\beta t+x)} y\,\nu(dy). 
\eeqnn
This implies that, for all $n\geq 1$, all sufficiently large $t$ and $x\geq -(\beta-\lambda)t$,
\beqnn
\big\{ Q_{t,n,1}^{(x)} \geq \eta(\beta t+x)\big\} \subset \big\{ N\big((0,t/n],(1,\eta(\beta t+x)] \big) \geq 2 \big\} \subset \big\{ N\big((0,t/n],(1, \infty) \big) \geq 2 \big\} .
\eeqnn 
Since $N((0,t/n],(1,\infty))$ is Poisson distributed with rate $(t/n)\overline{\nu}(1)$, we have
\beqnn
\sup_{x\geq -(\beta-\lambda)t}  n\cdot \mathbf{P}\big(  Q_{t,n,1}^{(x)}\geq \eta(\beta t+x) \big)
\ar\leq\ar    n\cdot \mathbf{P}\big(  N \big((0,t/n],(1,\infty) \big) \geq 2\big)
\leq n\cdot \big(t/n \cdot \overline{\nu}(1) \big)^2 .
\eeqnn 
Substituting these two estimates into \eqref{eqn.421} and then using the arbitrariness of $n$, we obtain that, uniformly in  sufficiently large $t$ and $x\geq -(\beta-\lambda)t$,
\beqnn
\mathbf P\Big(Q_t^{(x)}  \geq \frac{\beta t+x}{2}\Big)
\leq C\cdot (\beta t+x)^{-\frac{\kappa-1}{4\eta}}.
\eeqnn
This induces that 
\beqnn
\sup_{x\geq-(\beta-\lambda)t} \frac{\mathbf P\big( Q_t^{(x)}\geq (\beta t+x)/2 \big)}{\overline{\nu}(\beta t+x)/(\beta t +x)} \leq C\cdot  \sup_{x\geq-(\beta-\lambda)t}  \frac{(\beta t+x)^{-\frac{\kappa-1}{4\eta}+1}}{\overline{\nu}(\beta t+x)},
\eeqnn
which goes to $0$ as $t\to\infty$, since $\overline{\nu}(z)\in\mathrm{RV}_{-\alpha}^\infty$ and $\alpha<\frac{\kappa-1}{4\eta}-1$. The second limit in \eqref{eqn.3001}  is proved.
\qed 

\begin{proposition}\label{Prop.403}
	For any $\delta>0$, there exist two constants $C,x_0>0$ such that for any  $x\geq x_0$, any subset $B_x\subset (x,\infty)$ and any random variable $\xi$,
	\beqlb\label{eqn.4008}
	\int_{B_x} \mathbf{P}\big( \xi \in -w+\varDelta_\delta \big) \,\nu(dw)
	\leq C\cdot \frac{\overline{\nu}(x)}{x}\cdot \mathbf{P}\bigg( \xi\in \bigcup_{w\in B_x } \big(-w+\varDelta_\delta\big)  \bigg)\cdot \delta.
	\eeqlb
\end{proposition}
\proof For each integer $k\geq 0$, let $B_x(k):= B_x \cap\big( x+k\delta+\varDelta_\delta\big)$. Note that $-w+\varDelta_\delta \subseteq	 \cup_{x\in B_x(k)} \big(-x+\varDelta_\delta \big)$ for any $w\in B_x(k)$, we have
\beqnn
\int_{B_x} \mathbf{P}\big( \xi \in -w+\varDelta_\delta \big) \,\nu(dw)
\ar=\ar \sum_{k=0}^\infty \int_{B_x(k)} \mathbf{P}\big( \xi \in -w+\varDelta_\delta \big) \,\nu(dw)\cr
\ar\leq\ar \sum_{k=0}^\infty  \mathbf{P}\bigg( \xi \in \bigcup_{w\in B_x(k)}\big(-w+\varDelta_\delta \big) \bigg)  \,\nu\big(x+k\delta+\varDelta_\delta\big).
\eeqnn
By \eqref{eqn.30} and the monotonicity of $\overline\nu$, there exist two constants $C,x_0>0$ such that for all $x\geq x_0$, 
\beqlb\label{eqn.4007}
\nu\big(x+k\delta+\varDelta_\delta\big)
\ar\leq\ar C\cdot \frac{\overline\nu\big(x+k\delta \big)}{x+k\delta} \cdot \delta 
\leq C\cdot \frac{\overline\nu(x)}{x}  \cdot \delta .
\eeqlb
Moreover, note that $B_x(k)\cap B_x(k+i)=\emptyset$ for any $x\geq 0$, $k\geq 0$ and $i\geq 3$, we have $\sum_{k=0}^\infty \mathbf{1}_{\{y \in \cup_{w\in B_x(k)} (-w+\varDelta_\delta )  \}} \leq 3 \cdot \mathbf{1}_{\{ y \in \cup_{w\in B_x}  (-w+\varDelta_\delta )\}}$ for all $y\in \mathbb{R}$ and then
%$\sum_{k=0}^\infty \mathbf{1}_{\{ y\in A_k(t,x,z) \}}\leq 3$ for all $y,x,z \in \mathbb{R}$ and $t\geq 0$ and then 
\beqnn
\sum_{k=0}^\infty  \mathbf{P}\bigg( \xi \in \bigcup_{w\in B_x(k)}\big(-w+\varDelta_\delta \big) \bigg) 
=\mathbf{E}\bigg[ \sum_{k=0}^\infty \mathbf{1}_{\{\xi \in \cup_{w\in B_x(k)} (-w+\varDelta_\delta )  \}}\bigg] 
\leq 3\cdot \mathbf{P}\Big( \xi\in\bigcup_{w\in B_x } \big(-w+\varDelta_\delta\big)  \Big).
\eeqnn
Combining this together with \eqref{eqn.4007} yields the desired upper bound \eqref{eqn.4008} immediately.
\qed

\medskip

\textbf{\textit{Proof of Lemma~\ref{Lemma.401}.}}  By the triangle inequality, it suffices to prove 
\beqlb\label{eqn.4005}
\limsup_{t\to\infty}\sup_{x\geq 0} \frac{\underline{n}\big(\epsilon_t\in x+\varDelta_\delta, \mathcal{J}^{\eta(\beta t+x)}>t,\zeta >t\big)}{\overline{\nu}(\beta t+x)/ (\beta t +x) } =0
\eeqlb 
and
\beqlb\label{eqn.4006}
\limsup_{T\to\infty}\limsup_{t\to\infty}
\sup_{x\geq 0} \frac{\underline{n}\big(\epsilon_t\in x+\varDelta_\delta, T<\mathcal{J}^{\eta(\beta t+x)}\leq t,\zeta >t\big)}{ \overline{\nu}(\beta t+x)/ (\beta t +x)  } =0. 
\eeqlb

\textbf{Part~I.} We first prove the limit \eqref{eqn.4005}. 
Note that $\{\epsilon_t\in x+\varDelta_\delta\}\subset\{\epsilon_t>x\}$. Hence it suffices to prove that, as $t\to\infty$, 
\beqnn
\sup_{x\geq 0} \frac{\underline{n}\big(\epsilon_t>x, \mathcal{J}^{\eta(\beta t+x)}>t,\zeta >t\big)}{ \overline{\nu}(\beta t+x)/ (\beta t +x)  }\to 0.
\eeqnn
Let $\tau_1^+:=\inf\{s>0:\epsilon_s>1\}$, with the convention $\inf\emptyset=\infty$. 
For any fixed $T>0$ and any $t>T$, the preceding supremum is bounded by 
\beqlb\label{eqn.3004}
\sup_{x\geq 0} \frac{\underline{n}\big(\epsilon_t>x, \mathcal{J}^{\eta(\beta t+x)}>t,\tau_1^+ \leq T,\zeta >t\big)}{ \overline{\nu}(\beta t+x)/ (\beta t +x) }
+ \sup_{x\geq 0} \frac{\underline{n}\big(\epsilon_t>x, \mathcal{J}^{\eta(\beta t+x)}>t,\tau_1^+ >T,\zeta >t\big)}{ \overline{\nu}(\beta t+x)/ (\beta t +x) }.
\eeqlb
Hence it suffices to prove that both of these suprema vanish as $t\to\infty$. 

\textbf{(i)}  On the event $\{\tau_1^+\le T,\; \mathcal{J}^{\eta(\beta t+x)}>t\}$, we have $1/2\le \epsilon_{\tau_1^+}\le 1+\eta(\beta t+x)$.
By the Markov property, the first supremum in \eqref{eqn.3004} equals 
\beqnn
\lefteqn{\sup_{x\geq 0}  \frac{1}{  \overline{\nu}(\beta t+x)/ (\beta t +x) } \int_0^T \int_{1/2}^{1+\eta(\beta t+x)} \underline{n}\big(\epsilon_t>x, \mathcal{J}^{\eta(\beta t+x)}>t,\zeta >t,\tau_1^+ \in ds, \epsilon_s\in dz\big)}\cr
\ar\leq \ar \sup_{x\geq 0}  \int_0^T \int_{1/2}^{1+\eta(\beta t+x)}
\frac{\mathbf P_z\big( X_{t-s}>x, \mathcal{J}^{\eta(\beta t+x)}>t-s,\tau_0^->t-s\big) }{  \overline{\nu}(\beta t+x)/ (\beta t +x) }
\underline{n}\big(\tau_1^+\in ds,\epsilon_s\in dz,\zeta>s\big) .
\eeqnn
This can be bounded by 
\beqlb\label{eqn.3002}
\underline{n}\big( \tau_1^+\leq T,\zeta>\tau_1^+ \big) \cdot  \sup_{x\geq 0}
\sup_{s\in[0,T]} \sup_{1/2\leq z\leq1+\eta(\beta t+x)}  \frac{\mathbf{P}\big( X_{t-s}>x-z, \mathcal{J}^{\eta(\beta t+x)}>t-s\big)  }{  \overline{\nu}(\beta t+x)/ (\beta t +x) }. 
\eeqlb
Firstly, we have  $ \underline{n}\big( \tau_1^+\leq T,\zeta>\tau_1^+ \big) \leq  \underline{n}\big( \overline\epsilon\geq 1 \big)<\infty $. 
Moreover, for sufficiently large $t$, one easily verifies that, for all $s\in[0,T]$ and $1/2\le z\le 1+\eta(\beta t+x)$,
\beqlb\label{eqn.3003}
\beta t+x \leq \frac{ \beta (t-s)+x-z }{1-2\eta}  
\quad \mbox{and}\quad
x-z \geq -\frac{1+\eta}{2}\cdot \beta (t-s) .
\eeqlb
This, together with $\overline{\nu}\in\mathrm{RV}_{-\alpha}^\infty$, yields that, uniformly for sufficiently large $t$, all $x\ge0$, $s\in[0,T]$, and $1/2\le z\le 1+\eta(\beta t+x)$,
\beqnn
\frac{\overline{\nu}(\beta t+x)}{ \beta t +x} \geq C\cdot \frac{\overline{\nu}(\beta (t-s)+x-z)}{ \beta (t-s) +x-z} 
\eeqnn
and 
\beqnn
\mathbf{P}\big( X_{t-s}>x-z, \mathcal{J}^{\eta(\beta t+x)}>t-s\big) 
\ar\leq\ar \mathbf{P}\big( X_{t-s}>x-z, \mathcal{J}^{\frac{\eta }{1-2\eta} \cdot(\beta (t-s)+x-z) }>t-s\big) . 
\eeqnn
Substituting these into \eqref{eqn.3002} and then using the second inequality in \eqref{eqn.3003}, we obtain
\beqnn
\lefteqn{ \sup_{x\geq 0} \frac{\underline{n}\big(\epsilon_t>x, \mathcal{J}^{\eta(\beta t+x)}>t,\tau_1^+ \leq T,\zeta >t\big)}{ \overline{\nu}(\beta t+x)/ (\beta t +x) }}\cr
\ar\leq\ar C \cdot  \sup_{x\geq 0}
\sup_{s\in[0,T]} \sup_{1/2\leq z\leq1+\eta(\beta t+x)}  \frac{\mathbf{P}\big( X_{t-s}>x-z, \mathcal{J}^{\frac{\eta }{1-2\eta} \cdot(\beta (t-s)+x-z) }>t-s\big) }{ \overline{\nu}(\beta (t-s)+x-z)/ (\beta (t-s)+x-z)} \cr
\ar\leq\ar C\cdot \sup_{s\in[0,T]} \sup_{ y\geq  -\frac{1+\eta}{2}\cdot \beta (t-s)}  \frac{\mathbf{P}\big( X_{t-s}>y, \mathcal{J}^{\frac{\eta }{1-2\eta} \cdot(\beta (t-s)+y) }>t-s\big)}{\overline{\nu}(\beta (t-s)+y)/ (\beta (t-s)+y)  },
\eeqnn
which tends to $0$ as $t\to\infty$ by Proposition~\ref{Prop.401}, since $\frac{\eta}{1-2\eta}\le\frac{\alpha\wedge2-1}{4(\alpha+1)}$.

\textbf{(ii)}  On the event $\{\tau_1^+>T\}$ we have $\epsilon_T\leq1$. 
By the Markov property,  the second supremum in \eqref{eqn.3004} is bounded by 
\beqnn 
\sup_{x\geq 0}   \int_{[0,1]}  \frac{\mathbf{P}\big( X_{t-T}>x-z,\mathcal{J}^{\eta(\beta t+x)}>t-T,\tau_0^->t-T \big)}{  \overline{\nu}(\beta t+x)/ (\beta t +x) } \, \underline{n}\big( \epsilon_T\in dz ,\zeta>T \big)  ,
\eeqnn
which can be bounded by 
\beqnn
\underline{n}\big( \epsilon_T\leq 1,\zeta>T \big) \cdot  \sup_{x\geq 0}  \frac{\mathbf{P}\big( X_{t-T}>x-1,\mathcal{J}^{\eta(\beta t+x)}>t-T \big)}{  \overline{\nu}(\beta t+x)/ (\beta t +x) }  .
\eeqnn 
An argument analogous to (but simpler than) that in \textbf{(i)} shows that the preceding supremum tends to $0$ as $t\to\infty$. 
Consequently, the second supremum in \eqref{eqn.3004} also vanishes as $t\to\infty$. 
Combining the limits obtained in these two steps yields the limit \eqref{eqn.4005}.

\textbf{Part~II.} We now prove the limit \eqref{eqn.4006}. 
On the event $\{T<\mathcal{J}^{\eta(\beta t+x)}\le t,\ \zeta>t\}$, at least one jump of size larger than $\eta(\beta t+x)$ occurs during $[T,t]$. 
Hence 
\beqnn
\underline{n}\big(\epsilon_t\in x+\varDelta_\delta, T<\mathcal{J}^{\eta(\beta t+x)}\leq t,\zeta >t\big) 
\ar\leq\ar \underline{n}\bigg( \sum_{T<s\leq t} \mathbf{1}_{\{ \Delta\epsilon_s> \eta(\beta t+x),\, \epsilon_t\in x+\varDelta_\delta,\, \zeta>t \}}  \bigg).
\eeqnn
Using the Markov property and the compensation formula \eqref{eqn.Compensation} to the right-hand side yields that  
\beqlb\label{eqn.3006}
\lefteqn{\underline{n}\bigg( \sum_{T<s\leq t} \mathbf{1}_{\{ \Delta\epsilon_s> \eta(\beta t+x),\, \epsilon_t\in x+\varDelta_\delta,\, \zeta>t \}}  \bigg)}\cr
\ar=\ar \int_T^tds
\int_0^\infty \underline{n}_s(dz) \int_{\eta(\beta t+x)}^\infty   \mathbf{P}_{z+w}\big( X_{t-s}\in x+\varDelta_\delta,\,\tau_0^->t-s \big) \,\nu(dw)  \cr
\ar\leq\ar \int_T^tds
\int_0^\infty \underline{n}_s(dz) \int_{\eta(\beta t+x)}^\infty   \mathbf{P} \big( X_{t-s}\in x-z-w+\varDelta_\delta  \big) \,\nu(dw)  \cr
\ar\leq \ar  \int_T^\infty \underline{n}(\zeta >s)\,ds \cdot \sup_{r\in[0,t]} \sup_{z\geq 0}  \int_{\eta(\beta t+x)}^\infty   \mathbf{P} \big( X_r-x+z\in -w+\varDelta_\delta \big) \,\nu(dw)  .
\eeqlb
By Proposition~\ref{Prop.403},  there exist two constants $C,t_0>0$ such that for all $t\geq t_0$, 
\beqnn
\sup_{r\in[0,t]} \sup_{z\geq 0}  \int_{\eta(\beta t+x)}^\infty   \mathbf{P} \big( X_r-x+z\in -w+\varDelta_\delta \big) \,\nu(dw)  
\leq  C\cdot \frac{\overline\nu\big(\eta(\beta t+x) \big)}{\eta(\beta t+x) } \cdot \delta .
\eeqnn
Taking this back into \eqref{eqn.3006} and then using the fact $\overline{\nu}\in \mathrm{RV}^\infty_{-\alpha}$, we obtain that for all $t\geq t_0$,
\beqnn
\sup_{x\geq 0} \frac{\underline{n}\big(\epsilon_t\in x+\varDelta_\delta, T<\mathcal{J}^{\eta(\beta t+x)}\leq t,\zeta >t\big)}{ \overline{\nu}(\beta t+x)/ (\beta t +x)  } 
\leq C\cdot  \int_T^\infty \underline{n}(\zeta >s)\,ds,
\eeqnn
 which goes to $0$ as $T\to\infty$ since $\underline{n}(\zeta)<\infty$; see \eqref{eqn.501}. 
 Hence the limit \eqref{eqn.4006} holds and the whole proof ends.
 \qed

 \medskip

 \textbf{\textit{Proof of Lemma~\ref{Lemma.402}.}}   Since all jumps occurring before $\mathcal{J}^{\eta(\beta t+x)}$ are  not larger than $\eta(\beta t+x) $, we have  for any $t>T$,
 \beqnn
 \underline{n}\big( \epsilon_t\in x+\varDelta_\delta,\mathcal{J}^{\eta(\beta t+x)}\leq T,\zeta>t \big) 
 \ar=\ar  \underline{n}\bigg( \sum_{0<s\leq T} \mathbf{1}_{\{ \Delta\epsilon_s> \eta(\beta t+x),\,  \sup_{0<r<s}\Delta\epsilon_r \leq\eta(\beta t+x),\, \epsilon_t\in x+\varDelta_\delta,\,\zeta >t \}}   \bigg)  ,
 \eeqnn
 which can be further decomposed into the following two terms
 \beqnn
 \quad \underline{n}\bigg( \sum_{0<s\leq T} \mathbf{1}_{\{ \Delta\epsilon_s > \eta(\beta t+x),\,   \epsilon_t\in x+\varDelta_\delta,\,\zeta >t \}}   \bigg) 
 -\underline{n}\bigg( \sum_{0<s\leq T} \mathbf{1}_{\{ \Delta\epsilon_s > \eta(\beta t+x),\,  \sup_{0<r<s}\Delta\epsilon_r >\eta(\beta t+x),\, \epsilon_t\in x+\varDelta_\delta,\,\zeta >t \}}   \bigg) .
 \eeqnn
	 We first consider the second term. Note that on the event $\{\Delta\epsilon_s> \eta(\beta t+x),\,  \sup_{0<r<s}\Delta\epsilon_r >\eta(\beta t+x) \}$,  there always exist at least two jumps larger than $\eta(\beta t+x)$ occurring during $(0,T\wedge \zeta]$. 
 Hence it can be bounded by   
 \beqnn
 \underline{n}\bigg( \sum_{0<s<r\leq T\wedge \zeta} \mathbf{1}_{\{ \Delta\epsilon_s > \eta(\beta t+x),\,   \Delta\epsilon_r >\eta(\beta t+x)   \}}   \bigg) 
 \leq  T\cdot \underline{n}(\zeta)\cdot \Big|\overline{\nu}\big( \eta(\beta t+x) \big)\Big|^2 ,
 \eeqnn
 which is  $o\big( \overline\nu(\beta t+x)/(\beta t+x) \big)$.
 Here this inequality follows from Proposition~\ref{Prop.203}. 
 Thus the contribution of the second term can be ignored. 
 Consequently, it suffices to prove that as $t\to\infty$, 
 \beqlb\label{eqn.3007}
  \sup_{x\geq0} \bigg| \frac{   \underline{n}\big( \sum_{0<s\leq T} \mathbf{1}_{\{ \Delta\epsilon_s > \eta(\beta t+x),\,   \epsilon_t\in x+\varDelta_\delta,\,\zeta >t \}}   \big)}{\alpha \cdot \overline\nu(\beta t+x)/(\beta t+x)} -\frac{\underline n(\zeta\wedge T)}{ \underline n(\zeta)  } \int_x^{x+\delta}V(z)\, dz  \bigg| \to 0. 
 \eeqlb
By using the Markov property and the compensation formula \eqref{eqn.Compensation} again, we have 
\beqnn
\lefteqn{ \underline{n}\bigg( \sum_{0<s\leq T} \mathbf{1}_{\{ \Delta\epsilon_s > \eta(\beta t+x),\,   \epsilon_t\in x+\varDelta_\delta,\,\zeta >t \}}   \bigg) }\cr
\ar=\ar  \int_0^T ds
\int_0^\infty   \underline{n}_s(dz) 
\int_{\eta(\beta t+x)}^\infty \mathbf P_{z+w}\big( X_{t-s}\in x+\varDelta_\delta,\tau_0^->t-s \big)  \nu(dw).
\eeqnn 
For any fixed $\kappa\in(0,1/3)$, we decompose the right-hand side into the following two parts:
\beqnn
I_1(t,x,\kappa)\ar:=\ar  \int_0^T ds \int_0^\infty  \underline{n}_s(dz) \int_{B(t,x,\kappa)}  \mathbf P_{z+w}\big( X_{t-s}\in x+\varDelta_\delta,\tau_0^->t-s \big) \, \nu(dw), \cr
I_2(t,x,\kappa)\ar:=\ar    \int_0^T ds \int_0^\infty   \underline{n}_s(dz) \int_{(1-\kappa)(\beta t+x)}^{(1+\kappa)(\beta t+x)} \mathbf P_{z+w}\big( X_{t-s}\in x+\varDelta_\delta,\tau_0^->t-s \big) \, \nu(dw). 
\eeqnn
where $B(t,x,\kappa):= \big(\eta(\beta t+x), (1-\kappa)(\beta t+x)\big]\cup \big((1+\kappa)(\beta t+x),\infty\big)$. 
It is obvious that \eqref{eqn.3007} and then \eqref{eqn4.26}  hold once we prove the next two limits:
\beqlb\label{eqn.3009}
\limsup_{t\to\infty}\sup_{x\geq 0}\frac{I_1(t,x,\kappa)}{\overline\nu(\beta t+x)/(\beta t+x)} =0 
\eeqlb
and 
\beqlb\label{eqn.3010}
\limsup_{\kappa\to0+}\limsup_{t\to\infty}\sup_{x\geq 0}\bigg|\frac{I_2(t,x,\kappa)}{\alpha\cdot\overline\nu(\beta t+x)/(\beta t+x)} -\frac{\underline n(\zeta\wedge T)}{ \underline n(\zeta)  } \int_x^{x+\delta}V(z)\, dz \bigg|=0 . 
\eeqlb

\textbf{Case $i=1$.} Firstly, by the spatial homogeneity of $X$, we have 
\beqnn
I_1(t,x,\kappa)\ar\leq \ar  \int_0^T ds \int_0^\infty  \underline{n}_s(dz) \int_{B(t,x,\kappa)}  \mathbf{P}\big( X_{t-s}\in x-z-w+\varDelta_\delta  \big) \, \nu(dw).
\eeqnn
For any $K>0$, we decompose the right-hand side into the following two terms:
\beqnn
I_1^{\leq K}(t,x,\kappa)\ar:=\ar  \int_0^T ds \int_0^K \underline{n}_s(dz) \int_{B(t,x,\kappa)}  \mathbf{P}\big( X_{t-s}\in x-z-w+\varDelta_\delta  \big) \, \nu(dw),\cr
I_1^{>K}(t,x,\kappa)\ar:=\ar  \int_0^T ds \int_K^\infty  \underline{n}_s(dz) \int_{B(t,x,\kappa)}  \mathbf{P}\big( X_{t-s}\in x-z-w+\varDelta_\delta  \big) \, \nu(dw). 
\eeqnn
For the second term, notice that $B(t,x,\kappa) \subset\big(  \eta(\beta t+x), \infty\big)$, we have 
\beqnn
I_1^{>K}(t,x,\kappa) 
\ar\leq\ar  \sup_{r\in[0,T]} \sup_{z\geq 0}\int_{\eta(\beta t+x)}^{\infty}   \mathbf{P}\big( X_{t-r}\in x-z-w+\varDelta_\delta  \big) \, \nu(dw) \cdot \int_0^T\underline{n}(\epsilon_s>K,\,\zeta>s)\,ds 
\eeqnn
By Proposition~\ref{Prop.403}, there exist two constants $C,t_0>0$ such that for all $t\geq t_0$, 
\beqnn
I_1^{>K}(t,x,\kappa)  \ar\leq\ar C\cdot  \frac{\overline\nu\big(\beta t+x\big)}{\beta t+x} \cdot \int_0^T\underline{n}(\epsilon_s>K,\,\zeta>s)\,ds . 
\eeqnn  
This induces that
\beqlb\label{eqn.4009}
\limsup_{t\to\infty}\sup_{x\geq 0}\frac{ I_1^{>K}(t,x,\kappa)}{\overline\nu(\beta t+x)/(\beta t+x)} \leq C\cdot   \int_0^T\underline{n}(\epsilon_s>K,\,\zeta>s)\,ds ,
\eeqlb
which tends to $0$ as $K\to\infty$ since  $\underline{n}(\zeta)<\infty$; see \eqref{eqn.501}. 
We now consider $I_1^{\leq K}(t,x,\kappa)$. Firstly,  
\beqnn
I_1^{\leq K}(t,x,\kappa)\ar\leq \ar  \underline{n}(\zeta)\cdot \sup_{s\in[0,T]} \sup_{z\in[0,K]}  \int_{B(t,x,\kappa)}  \mathbf{P}\big( X_{t-s}\in x-z-w+\varDelta_\delta  \big) \, \nu(dw).   
\eeqnn
Using Proposition~\ref{Prop.403} along with the fact that $B(t,x,\kappa) \subset \big(  \eta(\beta t+x), \infty\big)$ again yields that uniformly in sufficiently large $t$,
\beqnn
I_1^{\leq K}(t,x,\kappa)\ar\leq \ar C\cdot  \frac{\overline\nu\big(\beta t+x\big)}{\beta t+x} \cdot \sup_{s\in[0,T]} \sup_{z\in[0,K]}  \mathbf{P}\bigg( X_{t-s}\in \bigcup_{w\in B(t,x,\kappa) }\big(x-z-w+\varDelta_\delta\big)  \bigg).
\eeqnn
A simple calculation shows that for any sufficiently large $t$, 
\beqnn
\big|X_{t-s}+\beta(t-s)\big| \geq \frac{\kappa}{2}\cdot \beta (t-s) ,
\eeqnn
uniformly in $s\in[0,T]$, $w \in B(t,x,\kappa)$, $z\in [0,K]$ and $X_{t-s} \in x-z-w+\varDelta_\delta$. 
Consequently, 
\beqnn
\sup_{x\geq 0}\frac{ I_1^{\leq K}(t,x,\kappa)}{\overline\nu(\beta t+x)/(\beta t+x)}  \ar\leq \ar C\cdot  \sup_{s\in[0,T]}  \mathbf{P}\Big(  \big|X_{t-s}+\beta(t-s)\big| \geq \frac{\kappa}{2}\cdot  \beta (t-s) \Big),
\eeqnn
which goes to $0$ as $t\to\infty$, since $(X_t+\beta t)/t\overset{\rm a.s.}\to 0$ by the law of large numbers. 
Combining this together with \eqref{eqn.4009} yields the limit \eqref{eqn.3009} directly.

\textbf{Case $i=2$.} We now prove the limit  \eqref{eqn.3010}. We first introduce two index sets:
\beqnn
\underline{\mathcal{K}}(t,x,\kappa)\ar:=\ar  \big\{j\in \mathbb{N}:   (1-\kappa)\cdot (\beta t+x) < j\kappa\delta + \eta(\beta t+x) \leq (1+\kappa)\cdot (\beta t+x) -\kappa\delta \big\}, \cr\cr
\overline{\mathcal{K}}(t,x,\kappa)\ar:=\ar  \big\{j\in \mathbb{N}:   (1-\kappa)\cdot (\beta t+x) -\kappa\delta < j\kappa\delta + \eta(\beta t+x) \leq (1+\kappa)\cdot (\beta t+x)   \big\}. 
\eeqnn
The interval $\big((1-\kappa)\cdot (\beta t+x),(1+\kappa)\cdot (\beta t+x)\big]$ contains $ \cup_{j\in\underline{\mathcal{K}}(t,x,\kappa)} \big(  \eta(\beta t+x) + j\kappa\delta +\varDelta_{\kappa\delta} \big)$ and is covered by  $\cup_{j \in \overline{\mathcal{K}}(t,x,\kappa)} \big(  \eta(\beta t+x) + j\kappa\delta +\varDelta_{\kappa\delta} \big)$. 
Hence $ I_{21}(t,x,\kappa) \leq I_2(t,x,\kappa) \leq I_{22}(t,x,\kappa)$, where
\beqnn
I_{21}(t,x,\kappa) \ar:=\ar  \int_0^T ds \int_0^\infty   \underline{n}_s(dz)  \sum_{j\in\underline{\mathcal{K}}(t,x,\kappa)}\int_{  \eta(\beta t+x) + j\kappa\delta }^{ \eta(\beta t+x) + (j+1)\kappa\delta } \mathbf P_{z+w}\big( X_{t-s}\in x+\varDelta_\delta,\tau_0^->t-s \big) \, \nu(dw),\cr
I_{22}(t,x,\kappa) \ar:=\ar  \int_0^T ds \int_0^\infty   \underline{n}_s(dz)  \sum_{j\in\overline{\mathcal{K}}(t,x,\kappa)} \int_{  \eta(\beta t+x) + j\kappa\delta }^{ \eta(\beta t+x) + (j+1)\kappa\delta } \mathbf P_{z+w}\big( X_{t-s}\in x+\varDelta_\delta,\tau_0^->t-s \big) \, \nu(dw) . 
\eeqnn
Consequently, the limit \eqref{eqn.3010} follows immediately when we establish the following lower and upper bounds hold:
\beqlb\label{eqn.3011}
\liminf_{\kappa\to0+}\liminf_{t\to\infty} \inf_{x\geq 0}
\bigg(  \frac{ I_{21}(t,x,\kappa)}{\alpha \cdot \overline\nu(\beta t+x)/(\beta t+x)} -\frac{\underline n(\zeta\wedge T)}{ \underline n(\zeta)  } \int_x^{x+\delta}V(z)\, dz  \bigg)  \geq 0
\eeqlb
and
\beqlb\label{eqn.3012}
\limsup_{\kappa\to0+}\limsup_{t\to\infty} \sup_{x\geq 0}
\bigg(  \frac{ I_{22}(t,x,\kappa)}{\alpha \cdot  \overline\nu(\beta t+x)/(\beta t+x)} -\frac{\underline n(\zeta\wedge T)}{ \underline n(\zeta)  } \int_x^{x+\delta}V(z)\, dz  \bigg)  \leq 0.
\eeqlb 
\medskip

\textbf{(i) Lower bound \eqref{eqn.3011}.}  Firstly, for any  $j\geq0$ and $w\in \eta(\beta t+x) + j\kappa\delta +\varDelta_{\kappa \delta}$, we have 
\beqnn
x-z-w+\varDelta_\delta  \supset x-z-\eta(\beta t+x) - j\kappa\delta + \varDelta_{(1-\kappa)\delta}=: A_j(t,x,z),\quad \forall \, x,z\geq 0,
\eeqnn
which along with the spatial-homogeneity of $X$ yields that 
\beqlb\label{eqn.4010}
\mathbf P_{z+w}\big( X_{t-s}\in x+\varDelta_\delta,\tau_0^->t-s\big)
\ar=\ar \mathbf P\big( X_{t-s}\in x-z-w+\varDelta_\delta,\tau_{-z-w}^->t-s\big)\cr\cr
\ar\geq\ar  \mathbf{P}\Big( X_{t-s}\in  A_j(t,x,z),\,\tau^-_{-z-\eta(\beta t+x)-j\kappa\delta}>t-s  \Big).
\eeqlb
Substituting this into the definition of $I_{21}(t,x,\kappa)$ yields 
\beqnn
\frac{ I_{21}(t,x,\kappa)}{\alpha \cdot \overline\nu(\beta t+x)/(\beta t+x)}
\geq \widehat{I}_{21}(t,x,\kappa)\cdot  \inf_{j\in\underline{\mathcal{K}}(t,x,\kappa)}\frac{ \nu\big(\eta(\beta t+x) + j\kappa\delta +\varDelta_{\kappa \delta}\big) }{\alpha \cdot \overline\nu(\beta t+x)/(\beta t+x)\cdot \kappa\delta} ,
\eeqnn
with
\beqnn 
\widehat{I}_{21}(t,x,\kappa)\ar:=\ar  \int_0^T ds \int_0^\infty   \kappa\delta \sum_{j\in\underline{\mathcal{K}}(t,x,\kappa)} \mathbf{P}\Big( X_{t-s}\in  A_j(t,x,z),\,\tau^-_{-z-\eta(\beta t+x)-j\kappa\delta}>t-s  \Big)\,   \underline{n}_s(dz).
\eeqnn 
By the definition of $\underline{\mathcal{K}}(t,x,\kappa)$ and \eqref{eqn.30}, we have as $t\to\infty$,
\beqlb\label{eqn.4011}
\inf_{x\geq 0}  \inf_{j\in\underline{\mathcal{K}}(t,x,\kappa)}\frac{ \nu\big(\eta(\beta t+x) + j\kappa\delta +\varDelta_{\kappa \delta}\big) }{\alpha \cdot \overline\nu(\beta t+x)/(\beta t+x)\cdot \kappa\delta} 
\ar\geq\ar   \inf_{x\geq 0}\inf_{|\lambda|\leq \kappa} \frac{ \nu\big( (1+\lambda)\cdot (\beta t+x) +\varDelta_{\kappa \delta}\big) }{\alpha \cdot \overline\nu(\beta t+x)/(\beta t+x)\cdot \kappa\delta}\cr
\ar=\ar  \inf_{z\geq \beta t}\inf_{|\lambda|\leq \kappa} \frac{ \nu\big( (1+\lambda)\cdot z+\varDelta_{\kappa \delta}\big) }{\alpha \cdot \overline\nu(z)/z\cdot \kappa\delta} \to (1+\kappa)^{-\alpha-1},
\eeqlb
which goes to $1$ as $\kappa\to 0+$. 
Consequently, the limit \eqref{eqn.3011} follows if we can prove 
\beqlb\label{eqn.4014}
\liminf_{\kappa\to0+}\liminf_{t\to\infty} \inf_{x\geq 0}
\bigg(    \widehat{I}_{21}(t,x,\kappa)  -\frac{\underline n(\zeta\wedge T)}{ \underline n(\zeta)  } \int_x^{x+\delta}V(z)\, dz  \bigg)  \geq 0.
\eeqlb
For any large constant $M>0$, we have 
\beqlb\label{eqn.4013}
\widehat{I}_{21} (t,x,\kappa)  \geq \int_0^T ds \int_0^M\kappa\delta  \sum_{j\in\underline{\mathcal{K}}(t,x,\kappa)} \mathbf{P}\Big( X_{t-s}\in  A_j(t,x,z),\,\tau^-_{-z-\eta(\beta t+x)-j\kappa\delta}>t-s  \Big)\,   \underline{n}_s(dz) .
% \ar\geq \ar \kappa\delta \cdot\inf_{z\in [0,M]}  \sum_{j\in\underline{\mathcal{K}}(t,x,\kappa)} \mathbf{P}\Big( X_{t-s}\in  A_j(t,x,z),\,\tau^-_{-z-\eta(\beta t+x)-j\kappa\delta}>t-s  \Big) \cdot \int_0^T   \underline{n}\big(\epsilon_s\leq M,\,\zeta>s)\, ds
\eeqlb
The definition of $\tau_y^-$ tells that 
\beqnn
\lefteqn{\kappa\delta  \inf_{z\in [0,M]}  \sum_{j\in\underline{\mathcal{K}}(t,x,\kappa)} \mathbf{P}\Big( X_{t-s}\in  A_j(t,x,z),\,\tau^-_{-z-\eta(\beta t+x)-j\kappa\delta}>t-s  \Big)}\cr
\ar=\ar  \kappa\delta  \inf_{z\in [0,M]}  \sum_{j\in\underline{\mathcal{K}}(t,x,\kappa)} \mathbf{P}\Big( X_{t-s}\in  A_j(t,x,z),\,\underline{X}_{t-s}\geq -z-\eta(\beta t+x)-j\kappa\delta  \Big) \cr
\ar\geq\ar  \mathbf{E}\bigg[ \kappa\delta  \inf_{z\in [0,M]} \sum_{j\in\underline{\mathcal{K}}(t,x,\kappa)} \mathbf{1}_{\{ X_{t-s}\in  A_j(t,x,z),\,\underline{X}_{t-s}\geq -z-\eta(\beta t+x)-j\kappa\delta\}}  \bigg].
\eeqnn
We claim the following, to be proved later: for any $u,v\in\mathbb R$ and $s\in[0,T]$,
\beqlb\label{eqn.4012}
\lefteqn{ \inf_{z\in [0,M]} \sum_{j\in\underline{\mathcal{K}}(t,x,\kappa)} \mathbf{1}_{\{ u\in  A_j(t,x,z),\,v\geq -z-\eta(\beta t+x)-j\kappa\delta\}} }\cr
\ar\geq\ar \Big( \frac{1}{\kappa\delta} \int_x^{x+\delta} \mathbf{1}_{\{ u-v\leq w \}}\,dw -2\Big) \cdot \mathbf{1}_{\{ |u+\beta (t-s)|\leq \kappa(\beta t+x)/2\}}.
\eeqlb
Combining these estimates together yields that 
\beqnn
\lefteqn{ \kappa\delta \inf_{z\in [0,M]}  \sum_{j\in\underline{\mathcal{K}}(t,x,\kappa)} \mathbf{P}\Big( X_{t-s}\in  A_j(t,x,z),\,\tau^-_{-z-\eta(\beta t+x)-j\kappa\delta}>t-s  \Big)}\cr
\ar\geq \ar  \mathbf{E}\bigg[\Big(  \int_x^{x+\delta} \mathbf{1}_{\{ X_{t-s}-\underline{X}_{t-s}\leq z \}}\,dz -2\kappa\delta  \Big) \cdot \mathbf{1}_{\{ |X_{t-s}+\beta (t-s)|\leq \kappa(\beta t+x)/2  \}}\bigg]\cr
\ar\geq \ar \int_x^{x+\delta} \mathbf{P} \Big(  X_{t-s}-\underline{X}_{t-s}\leq z ,\,   \big|X_{t-s}+\beta (t-s)\big|\leq \frac{\kappa(\beta t+x)}{2}   \Big) \,dz - 2\kappa\delta \cr
\ar\geq \ar \int_x^{x+\delta} \mathbf{P} \big(  X_{t-s}-\underline{X}_{t-s}\leq z   \big) \,dz - \delta \cdot \mathbf{P} \Big(    \big|X_{t-s}+\beta (t-s)\big|> \frac{\kappa(\beta t+x)}{2}   \Big) - 2\kappa\delta\cr
\ar\geq\ar \int_x^{x+\delta} \mathbf{P} \big(   \overline{X}_{t-s}\leq z  \big) \,dz - \delta \cdot \mathbf{P} \Big(    \big|X_{t-s}+\beta (t-s)\big|> \frac{\kappa\beta (t-s)}{2}   \Big) - 2\kappa\delta. 
\eeqnn
Here the last inequality follows from $\beta t+x\geq \beta(t-s)$ and $X_{t-s}-\underline{X}_{t-s}\overset{\rm d}= \overline{X}_{t-s}$ by the duality lemma. 
Substituting it into the right-hand side of \eqref{eqn.4013} yields
\beqnn
\lefteqn{\widehat{I}_{21} (t,x,\kappa)   -\frac{\underline n(\zeta\wedge T)}{ \underline n(\zeta)  } \int_x^{x+\delta}V(z)\, dz}\cr\cr
\ar\geq\ar \int_0^T \underline{n}\big(\epsilon_s\leq M,\,\zeta >s\big)\, ds \int_x^{x+\delta} \mathbf{P} \big(   \overline{X}_{t-s} \leq z \big) \,dz  -  2\kappa\delta \int_0^T \underline{n}\big( \epsilon_s\leq M,\,\zeta >s\big)\, ds \cr\cr
\ar\ar - \delta  \int_0^T \underline{n}\big( \epsilon_s\leq M,\,\zeta >s\big)\cdot \mathbf{P} \Big(  \big|X_{t-s}+\beta (t-s)\big|> \frac{\kappa\beta (t-s)}{2} \Big)\, ds -\frac{\underline n(\zeta\wedge T)}{ \underline n(\zeta)  } \int_x^{x+\delta}V(z)\, dz   \cr\cr
\ar\geq\ar \int_0^T \underline{n}\big(\epsilon_s\leq M,\,\zeta >s\big)\, ds \cdot \int_x^{x+\delta}   \inf_{r\in[0,T]}\mathbf{P} \big(   \overline{X}_{t-r}\leq z  \big) \,dz  -\frac{\underline n(\zeta\wedge T)}{ \underline n(\zeta)  } \int_x^{x+\delta}V(z)\, dz  \cr\cr
\ar\ar - \delta  \int_0^T \underline{n}\big(  \zeta >s\big)\, ds\cdot \sup_{r\in[0,T]} \mathbf{P} \Big(    \big|X_{t-r}+\beta (t-r)\big|>  \frac{\kappa\beta (t-r)}{2}    \Big) -  2\kappa\delta \int_0^T \underline{n}\big( \zeta >s\big)\, ds . 
\eeqnn
Firstly, the last term on the right-hand side of the last inequality goes to $0$ as $\kappa \to 0$. 
Secondly, the law of large numbers  yields that the third term  also goes to $0$ as $t\to\infty$. Consequently,
\beqnn
\lefteqn{\liminf_{\kappa\to0+}\liminf_{t\to\infty} \inf_{x\geq 0}
	\bigg(    \widehat{I}_{21} (t,x,\kappa)  -\frac{\underline n(\zeta\wedge T)}{ \underline n(\zeta)  } \int_x^{x+\delta}V(z)\, dz  \bigg)}\cr
\ar\geq\ar  \liminf_{t\to\infty} \inf_{x\geq 0} \bigg( \int_0^T \underline{n}\big(\epsilon_s\leq M,\,\zeta >s\big)\, ds \int_x^{x+\delta}   \inf_{r\in[0,T]}\mathbf{P} \big(    \overline{X}_{t-r} \leq z  \big) \,dz-\frac{\underline n(\zeta\wedge T)}{ \underline n(\zeta)  } \int_x^{x+\delta}V(z)\, dz \bigg).
\eeqnn
Using the triangle inequality, then $\underline{n}(\zeta)<\infty$ (see \eqref{eqn.501}) and  $V(z)\leq V(\infty)<\infty$ to the right-hand side yields
\beqnn \lefteqn{\liminf_{\kappa\to0+}\liminf_{t\to\infty} \inf_{x\geq 0}
	\bigg(  \widehat{I}_{21} (t,x,\kappa)  -\frac{\underline n(\zeta\wedge T)}{ \underline n(\zeta)  } \int_x^{x+\delta}V(z)\, dz  \bigg)}\cr
\ar\geq\ar \inf_{x\geq 0} \bigg(\int_0^T \underline{n}\big(\epsilon_s\leq M,\,\zeta >s\big)\, ds - \underline n(\zeta\wedge T) \bigg) \cdot   \int_x^{x+\delta} \frac{V(z)}{ \underline n(\zeta)  }\, dz \cr
\ar\ar +  \int_0^T \underline{n}\big(\epsilon_s\leq M,\,\zeta >s\big)\, ds  \cdot \liminf_{t\to\infty} \inf_{x\geq 0}\int_x^{x+\delta}   \Big(\inf_{r\in[0,T]}\mathbf{P} \big(    \overline{X}_{t-r} \leq z  \big) - \frac{V(z)}{\underline n(\zeta)}\Big)\,dz \cr
\ar\geq \ar - \sup_{x\geq 0}  \int_x^{x+\delta} \frac{V(z)}{ \underline n(\zeta)  }\, dz  \cdot  \bigg|\int_0^T \underline{n}\big(\epsilon_s\leq M,\,\zeta >s\big)\, ds - \underline n(\zeta\wedge T)\bigg| \cr
\ar\ar - \int_0^T \underline{n}\big(\epsilon_s\leq M,\,\zeta >s\big)\, ds  \cdot \limsup_{t\to\infty} \sup_{x\geq 0}\int_x^{x+\delta}   \Big|\inf_{r\in[0,T]}\mathbf{P} \big(    \overline{X}_{t-r} \leq z  \big) - \frac{V(z)}{\underline n(\zeta)}\Big|\,dz\cr
\ar\geq\ar - C \cdot  \bigg|\int_0^T \underline{n}\big(\epsilon_s\leq M,\,\zeta >s\big)\, ds - \underline n(\zeta\wedge T)\ \bigg|  - C\cdot \limsup_{t\to\infty} \sup_{r\in[0,T]} \sup_{z\geq 0}   \bigg| \mathbf{P} \big(    \overline{X}_{t-r} \leq z  \big) -    \frac{V(z)}{\underline n(\zeta)}\bigg| ,
\eeqnn 
for some constant $C>0$ independent of $M$. 
Obviously, the first term on the right-hand side of the last inequality goes to $0$ as $M\to\infty$. 
Furthermore, since $ \overline{X}_{t-r} \overset{\rm a.s.}\to \overline{X}_\infty<\infty$ as $t\to\infty$, and  the Wiener--Hopf factorization tells that  $\mathbf{P}\big(\overline X_\infty\leq z \big)=V(z)/\underline n(\zeta)$, we have as $t\to\infty$ 
\beqlb\label{eqn.4015}
\sup_{z\geq 0}\bigg| \mathbf{P} \big(    \overline{X}_{t} \leq z  \big) -    \frac{V(z)}{\underline n(\zeta)}\bigg| \to 0,
\eeqlb
and then the second term is identically zero. 
Consequently, to complete the proof of \eqref{eqn.4014} (and therefore of \eqref{eqn.3011}), it remains only to establish the inequality \eqref{eqn.4012}.
 Fix $M,T>0$. For $x\geq0$, $s\in[0,T]$, $u,v\in\mathbb R$, and $z\in[0,M]$, let $j\in\mathbb Z$  satisfy
 \beqlb\label{eqn5.44}
 \big|u+\beta(t-s)\big| \leq \frac{\kappa}{2}(\beta t+x)
 \quad\mbox{and}\quad
 \max\{x,u-v\} 	< u+z+\eta(\beta t+x)+j\kappa\delta \leq x+(1-\kappa)\delta.
 \eeqlb
 Then, for all sufficiently large $t$, uniformly in $x\geq0$ and $s\in[0,T]$,
 \beqnn
  \eta(\beta t+x)+j\kappa\delta > x-u-z	\geq \Big(1-\frac{\kappa}{2}\Big)(\beta t+x)-\beta T-M>	(1-\kappa)(\beta t+x),
 \eeqnn
 and
 \beqnn
  \eta(\beta t+x)+j\kappa\delta \leq 	x+(1-\kappa)\delta-u-z
  \leq \Big(1+\frac{\kappa}{2}\Big)(\beta t+x)+(1-\kappa)\delta
  <(1+\kappa)(\beta t+x)-\kappa\delta.
 \eeqnn
 Since $\eta<1-\kappa$, the preceding estimates, together with \eqref{eqn5.44}, give
 \beqnn
 j\in\underline{\mathcal K}(t,x,\kappa),\quad u\in A_j(t,x,z),
 \quad 
 v>-z-\eta(\beta t+x)-j\kappa\delta.
 \eeqnn
 Consequently,
 \beqnn
 \lefteqn{ \sum_{j\in\underline{\mathcal K}(t,x,\kappa)} \mathbf 1_{\{u\in A_j(t,x,z),\, v\geq-z-\eta(\beta t+x)-j\kappa\delta \}}}\cr
 \ar\geq\ar
 \#\big\{j\in\mathbb Z: \max\{x,u-v\} < u+z+\eta(\beta t+x)+j\kappa\delta \leq x+(1-\kappa)\delta \big\}\cdot \mathbf 1_{\{|u+\beta(t-s)| \leq \kappa(\beta t+x)/2\}}.
 \eeqnn
 Since consecutive points $u+z+\eta(\beta t+x)+j\kappa\delta$ are spaced by $\kappa\delta$, the cardinality factor on the right-hand side of the preceding estimate is bounded below by
 \beqnn
 \frac{1}{	\kappa\delta}\cdot	\big[x+(1-\kappa)\delta-\max\{x,u-v\}\big]^+-1
 \ar\geq\ar \frac{1}{	\kappa\delta}\cdot  \big[x+\delta-\max\{x,u-v\}\big]^+ -2\cr\cr
 \ar=\ar \frac{1}{\kappa\delta} \int_x^{x+\delta} \mathbf 1_{\{u-v\leq w\}}\,dw -2.
 \eeqnn
 Combining these  estimates and taking the infimum over $z\in[0,M]$ proves \eqref{eqn.4012}.  
 \medskip

 \textbf{(ii) Upper bound \eqref{eqn.3012}.} 
 Similarly as in \eqref{eqn.4010},   for any $j\geq 0$ and $w\in \eta(\beta t+x) + j\kappa\delta +\varDelta_{\kappa \delta}$ we have
 \beqnn
 \mathbf{P}_{z+w}\big( X_{t-s}\in x+\varDelta_\delta,\tau_0^->t-s \big)   \leq \mathbf{P}\Big( X_{t-s}\in U_j(x,z,\kappa),\, \tau_{-z-\eta(\beta t+x)-(j+1)\kappa\delta}^->t-s \Big),
 \eeqnn
 with $U_j(x,z,\kappa):= x-z - \eta(\beta t+x)-(j+1)\kappa\delta +\varDelta_{(1+\kappa)\delta} $. 
%For any $j\geq0$ and  $w\in \eta(\beta t+x)+j\kappa\delta +\varDelta_{\kappa\delta}$, we have
% \beqnn
% x-z-w+\varDelta_\delta \in x-z - \eta(\beta t+x)-(j+1)\kappa\delta +\varDelta_{(1+\kappa)\delta} =: U_j(x,z,\kappa)
% \eeqnn
% and then
% \beqnn
%  \mathbf{P}_{z+w}\big( X_{t-s}\in x+\varDelta_\delta,\tau_0^->t-s \big) 
% \ar=\ar \mathbf P\big( X_{t-s}\in  x-z-w+\varDelta_\delta,  \tau^-_{-z-w}>t-s\big)
% \cr\cr
% \ar\leq\ar  \mathbf{P}\big( X_{t-s}\in U_j(x,z,\kappa),\, \tau^-_{-z-\eta(\beta t+x)-(j+1)\kappa\delta}>t-s\big)\cr\cr
% \ar=\ar \mathbf{P}\big( X_{t-s}\in U_j(x,z,\kappa),\,  \underline{X}_{t-s}\geq -z-\eta(\beta t+x)-(j+1)\kappa\delta \big)
% \eeqnn
 This yields that 
 \beqnn
  \frac{ I_{22}(t,x,\kappa)}{\alpha \cdot \overline\nu(\beta t+x)/(\beta t+x)}
  \leq  \widehat{I}_{22}(t,x,\kappa)\cdot \sup_{j\in\overline{\mathcal{K}}(t,x,\kappa)}\frac{ \nu\big(\eta(\beta t+x) + j\kappa\delta +\varDelta_{\kappa \delta}\big) }{\alpha \cdot \overline\nu(\beta t+x)/(\beta t+x)\cdot \kappa\delta}
 \eeqnn
 where 
 \beqnn
  \widehat{I}_{22}(t,x,\kappa)\ar:=\ar \int_0^T ds \int_0^\infty    \kappa\delta\sum_{j\in\overline{\mathcal{K}}(t,x,\kappa)}  \mathbf{P}\Big( X_{t-s}\in U_j(x,z,\kappa),\, \tau_{-z-\eta(\beta t+x)-(j+1)\kappa\delta}^->t-s \Big) \,   \underline{n}_s(dz).
 \eeqnn 
 An argument analogous to \eqref{eqn.4011} gives 
 \beqnn
  \limsup_{\kappa\to0+}  \limsup_{t\to\infty} \sup_{x\geq 0} \sup_{j\in\overline{\mathcal{K}}(t,x,\kappa)}\frac{ \nu\big(\eta(\beta t+x) + j\kappa\delta +\varDelta_{\kappa \delta}\big) }{\alpha \cdot \overline\nu(\beta t+x)/(\beta t+x)\cdot \kappa\delta} \leq \limsup_{\kappa\to0+}  (1-\kappa)^{-\alpha-1} =1. 
 \eeqnn
 Consequently, we just need to prove that 
 \beqlb\label{eqn.4017}
  \limsup_{\kappa\to0+}\limsup_{t\to\infty} \sup_{x\geq 0} \bigg(    \widehat{I}_{22}(t,x,\kappa) -\frac{\underline n(\zeta\wedge T)}{ \underline n(\zeta)  } \int_x^{x+\delta}V(z)\, dz  \bigg)  \leq 0.
 \eeqlb
 To analyze $\widehat{I}_{22}(t,x,\kappa)$, we first claim, and will prove later, that for any $u,v\in\mathbb R$ and $z\geq 0$,
 \beqlb\label{eqn.4016}
 \sum_{j\in\overline{\mathcal{K}}(t,x,\kappa)} \mathbf 1_{\{u\in U_j(x,z,\kappa),\,	v\geq -z-\eta(\beta t+x)-(j+1)\kappa\delta \}} \leq  \frac{1}{\kappa\delta} \int_x^{x+\delta} \mathbf 1_{\{u-v\leq w\}}\, dw+2.
 \eeqlb
 This  induces that 
 \beqnn
 \lefteqn{  \sum_{j\in\overline{\mathcal{K}}(t,x,\kappa)}  \mathbf{P}\Big( X_{t-s}\in U_j(x,z,\kappa),\, \tau_{-z-\eta(\beta t+x)-(j+1)\kappa\delta}^->t-s \Big) }\cr
 \ar=\ar  \sum_{j\in\overline{\mathcal{K}}(t,x,\kappa)}  \mathbf{P}\Big( X_{t-s}\in U_j(x,z,\kappa),\, \underline{X}_{t-s}\geq -z-\eta(\beta t+x)-(j+1)\kappa\delta \Big)\cr
 \ar=\ar  \mathbf{E}\bigg[ \sum_{j\in\overline{\mathcal{K}}(t,x,\kappa)} \mathbf{1}_{\{ X_{t-s}\in U_j(x,z,\kappa),\, \underline{X}_{t-s}\geq -z-\eta(\beta t+x)-(j+1)\kappa\delta\}} \bigg]\cr
 \ar\leq\ar \frac{1}{\kappa\delta}\int_x^{x+\delta} \mathbf{P}\big(  X_{t-s}-\underline{X}_{t-s} \leq  w\big) dw+2 
 = \frac{1}{\kappa\delta}\int_x^{x+\delta} \mathbf{P}\big(  \overline{X}_{t-s}  \leq  w\big) dw+2 .
 \eeqnn
 Here the last equality follows from the identity $X_{t-s}-\underline{X}_{t-s}\overset{\rm d}= \overline{X}_{t-s} $ by the duality lemma. 
 Substituting this estimate into the definition of $ \widehat{I}_{22}(t,x,\kappa)$ yields that 
 \beqnn
 \lefteqn{ \sup_{x\geq 0}
	\bigg(    \widehat{I}_{22}(t,x,\kappa) -\frac{\underline n(\zeta\wedge T)}{ \underline n(\zeta)  } \int_x^{x+\delta}V(z)\, dz  \bigg)  }\cr 
 \ar\leq\ar  2 \underline{n}(\zeta) \kappa\delta    +  \underline{n}(\zeta \wedge T) \cdot \sup_{x\geq 0} \bigg(  \int_x^{x+\delta} \sup_{s\in[0,T]}\mathbf{P}\big(  \overline{X}_{t-s} \leq  z\big) \,dz -\int_x^{x+\delta} \frac{V(z) }{ \underline n(\zeta)  } \, dz\bigg)\cr
 \ar\leq\ar 2 \underline{n}(\zeta) \kappa\delta    +  \underline{n}(\zeta \wedge T) \cdot \delta \cdot \sup_{z\geq 0} \sup_{s\in[0,T]}\bigg| \mathbf{P}\big(  \overline{X}_{t-s} \leq  z\big)- \frac{V(z) }{ \underline n(\zeta)  }\bigg|  ,
 \eeqnn
 which goes to $0$ as $t\to\infty$ and then $\kappa\to 0+$ by \eqref{eqn.4015}.
 Consequently, it remains to identify the inequality  \eqref{eqn.4016} to complete the proof of \eqref{eqn.4017}  and  therefore of \eqref{eqn.3012}. 
 Fix $u,v\in\mathbb R$ and $z\geq0$.
 Let $j\in\overline{\mathcal K}(t,x,\kappa)$ be any index satisfying $u\in U_j(x,z,\kappa) $ and $v\geq-z-\eta(\beta t+x)-(j+1)\kappa\delta$.
 By the definition of $U_j(x,z,\kappa)$ and the preceding inequality,
 \beqnn
 x-\kappa\delta<u+z+\eta(\beta t+x)+j\kappa\delta\leq x+\delta
 \quad \mbox{and}\quad 
 u-v-\kappa\delta\leq u+z+\eta(\beta t+x)+j\kappa\delta.
 \eeqnn
 Consequently,
 \beqnn
 \lefteqn{ \sum_{j\in\overline{\mathcal K}(t,x,\kappa)} \mathbf 1_{\{u\in U_j(x,z,\kappa),\,v\geq -z-\eta(\beta t+x)-(j+1)\kappa\delta\}}}\cr\cr
 \ar\leq\ar \#\big\{j\in\mathbb Z:\max\{x,u-v\}-\kappa\delta< u+z+\eta(\beta t+x)+j\kappa\delta \leq x+\delta \big\}.
\eeqnn
Since consecutive points $u+z+\eta(\beta t+x)+j\kappa\delta$ are spaced by $\kappa\delta$, the cardinality on the right-hand side is bounded above by
\beqnn
 \frac{1}{ \kappa\delta }\cdot 	\big[x+(1+\kappa)\delta-\max\{x,u-v\}\big]^+ +1 \ar\leq\ar \frac{1}{	\kappa\delta}\cdot 	\big[x+\delta-\max\{x,u-v\}\big]^+ +2 \cr
 \ar=\ar \frac{1}{\kappa\delta} \int_x^{x+\delta} \mathbf 1_{\{u-v\leq w\}}\,dw +2.
 \eeqnn
 Combining the preceding estimates proves \eqref{eqn.4016}. 
 This completes the whole proof.
 \qed

 \subsection{Proofs of Lemma \ref{Lemma.601}-\ref{Lemma.604}}
 \label{Sec.402}

 \textbf{\textit{Proofs of Lemma~\ref{Lemma.601} and \ref{Lemma.602}}}  Firstly, a simple calculation shows that as $t\to\infty$,
 \beqnn
  \sup_{s\in[0,R]} \sup_{x\geq 0} \sup_{z\in[0,y]} \Big| \frac{\beta(t-s)+(x-y+z)^+}{\beta t +x} -1 \Big|
  \leq \frac{\beta R +y}{\beta t} \to 0.
 \eeqnn
 Therefore, by  $\overline{\nu}\in \mathrm{RV}^\infty_{-\alpha}$ and the dominated convergence theorem,  the two desired limits in Lemma~\ref{Lemma.601} and \ref{Lemma.602} follow  immediately if we can prove that as $t\to\infty$,  
 \beqlb\label{eqn.606}
 \sup_{s\in[0,R]}\sup_{x\geq 0}    \sup_{z\in[0,y]} \big| \varepsilon(t-s,x,y-z)\big| = \sup_{s\in[0,R]} \sup_{x\geq 0}   \sup_{u\in[0,y]} \big| \varepsilon(t-s,x,u)\big|  \to 0.
 \eeqlb
 where 
 \beqnn
 \varepsilon(r,x,u)\ar:=\ar \frac{\beta r+(x-u)^+}{\overline\nu\big(\beta r+(x-u)^+\big)}\cdot \underline{n}_r\big(  \big((x-u)^+,(x-u+\delta)^+] \big) - \alpha    \int_{(x-u)^+}^{(x-u+\delta)^+} V(z)\, dz .
 \eeqnn 
 For any $\lambda \in(0,\delta\wedge 1)$, when $(x-u+\delta)^+ -(x-u)^+>\lambda$, this uniform convergence follows directly from Theorem~\ref{MainThm01}, i.e.,
 \beqlb\label{eqn.607}
 \limsup_{t\to\infty}\sup_{s\in[0,R]}  \sup_{x\geq 0}   \sup_{u\in[0,y]} \big| \varepsilon(t-s,x,u)\big| \cdot \mathbf{1}_{\{ (x-u+\delta)^+ -(x-u)^+ >\lambda \}} =0.
 \eeqlb
 When $(x-u+\delta)^+ -(x-u)^+\leq \lambda$, we have 
 \beqnn
  \sup_{s\in[0,R]} \sup_{x\geq 0}  \sup_{u\in[0,y]} \alpha   \int_{(x-u)^+}^{(x-u+\delta)^+} V(z)\, dz  
 \leq \alpha    V(\infty)\cdot\lambda,
 \eeqnn
 Moreover, since $\delta \mapsto \underline{n}_t(x+\varDelta_\delta)$ is non-decreasing, we use Theorem~\ref{MainThm01}  and obtain that for sufficiently large $t$, 
 \beqnn
 \lefteqn{\sup_{s\in[0,R]}  \sup_{x\geq 0}   \sup_{u\in[0,y]} \frac{\beta (t-s) +(x-u)^+}{\overline\nu\big(\beta (t-s)+(x-u)^+\big)}\cdot \underline{n}_{t-s}\big(  \big((x-u)^+,(x-u+\delta)^+] \big)}\cr
 \ar\leq\ar \frac{\beta (t-s) +(x-u)^+}{\overline\nu\big(\beta (t-s)+(x-u)^+\big)}\cdot  \underline{n}_{t-s}\big(   (x-u)^+ +\varDelta_{\lambda} \big)  \leq C  \int_{(x-u)^+}^{(x-u)^++\lambda} V(z)\,dz
 \leq C V(\infty)\cdot \lambda.
 \eeqnn
 These two estimates together yield
 \beqnn
   \limsup_{\lambda \to0+}\limsup_{t\to\infty}\sup_{s\in[0,R]}  \sup_{x\geq 0}   \sup_{u\in[0,y]} \big| \varepsilon(t-s,x,u)\big| \cdot \mathbf{1}_{\{ (x-u+\delta)^+ -(x-u)^+ \leq \lambda \}} =0.
 \eeqnn
 Combining this together with \eqref{eqn.607} directly yields \eqref{eqn.606}. This completes the proof.
 \qed

 \textbf{\textit{Proof of Lemma~\ref{Lemma.603}.}} 
 By using the change of variables $r=t-s$ and then Fubini's theorem, 
 \beqnn
 I_3^{y,\delta}(t,R,x)\ar=\ar \frac{\beta t +x}{\overline\nu(\beta t+x)} \int_0^R dr\int_0^y \underline{n}_{r}\big( \big((x-y+z)^+, (x-y+z+\delta)^+\big] \big) \, \overline{n}_{t-r}(dz) \cr
 \ar=\ar  \int_0^R dr\int_0^\infty  \Big(
 \frac{\beta t +x}{\overline\nu(\beta t+x)}\cdot \overline{n}_{t-r} \big( B_y^\delta(x,z)\big)\Big)\, \underline{n}_{r} (dz)
 \eeqnn
 and 
 \beqnn
 \widetilde{I}_3^{y,\delta}(R,x) = 
 \int_0^Rdr \int_0^\infty \Big(  \alpha    \int_{B_y^\delta(x,z)} \widehat{V}(u)\,du\Big)\, \underline{n}_r (dz), 
 \eeqnn
 where $B_y^\delta(x,z):= (0,y]\cap [y-x+z-\delta,y-x+z)$.
 This follows that 
 \beqnn
  I_3^{y,\delta}(t,R,x) -\widetilde{I}_3^{y,\delta}(R,x)  
  \ar=\ar \int_0^R dr\int_0^\infty  \Big(
  \frac{\beta t +x}{\overline\nu(\beta t+x)}\cdot \overline{n}_{t-r} \big( B_y^\delta(x,z)\big) - \alpha    \int_{B_y^\delta(x,z)} \widehat{V}(u)\,du \Big)\, \underline{n}_{r} (dz).
 \eeqnn
 For any large constant $M>y$,   using $\overline\nu \in \mathrm{RV}^\infty_{-\alpha}$, then Lemma~\ref{Lem5.4} and the uniform convergence theorem for monotone functions on compact intervals yields 
 \beqlb\label{eqn.4023}
  \lefteqn{\limsup_{t\to\infty} \sup_{x\in[0,M]} \big|  I_3^{y,\delta}(t,R,x) -\widetilde{I}_3^{y,\delta}(R,x)   \big|}\cr
  \ar\leq\ar \underline{n}(\zeta) \cdot \limsup_{t\to\infty} \sup_{r\in[0,R]} \sup_{x\in[0,M]} \sup_{z\geq 0} \Big|
  \frac{\beta t +x}{\overline\nu(\beta t+x)}\cdot \overline{n}_{t-r} \big( B_y^\delta(x,z)\big) - \alpha    \int_{B_y^\delta(x,z)} \widehat{V}(u)\,du \Big|\cr
  \ar=\ar \underline{n}(\zeta) \cdot \limsup_{t\to\infty} \sup_{r\in[0,R]} \sup_{x\in[0,M]} \sup_{z\geq 0} \Big|
  \frac{\beta (t-r)}{\overline\nu\big(\beta (t-r)\big)}\cdot \overline{n}_{t-r} \big( B_y^\delta(x,z)\big) - \alpha    \int_{B_y^\delta(x,z)} \widehat{V}(u)\,du \Big|
  =0.
 \eeqlb
 We now consider the case of $x\geq M$. Note that $B_y^\delta(x,z) \subset (0,y]$ and $B_y^\delta(x,z)\neq \emptyset$ only if $z\in x-y +\varDelta_{y+\delta}$, we have as $M\to\infty$,
 \beqlb\label{eqn.4022}
  \sup_{x\geq M}\widetilde{I}_3^{y,\delta}(R,x) 
  \ar\leq\ar  \alpha  \cdot  \int_0^y \widehat{V}(u)\,du \cdot \sup_{x\geq M} \int_0^R  \underline{n}_r(x-y +\varDelta_{y+\delta}) \, dr\cr
  \ar\leq \ar \alpha  \cdot  \int_0^y \widehat{V}(u)\,du \cdot  \int_0^R  \underline{n}\big( \epsilon_r>M-y,\,\zeta>r\big) \, dr
  \to 0.
 \eeqlb
 Moreover, using Lemma~\ref{Lem5.4} again yields that uniformly for all sufficiently large $t$,
 \beqnn
  \sup_{r\in[0,R]} \sup_{z\geq 0} \overline{n}_{t-r} \big( B_y^\delta(x,z)\big)  
  \leq  \sup_{r\in[0,R]}   \overline{n}_{t-r} \big( (0,y]\big)  \leq C\cdot  \frac{\overline\nu(\beta t)}{\beta t},
 \eeqnn
 which along with the fact that $B_y^\delta(x,z)\neq \emptyset$ only if $z\in x-y +\varDelta_{y+\delta}$ induces that
 \beqnn
 I_3^{y,\delta}(t,R,x) 
 \ar\leq\ar  C\cdot  \frac{\overline\nu(\beta t)}{\beta t} \cdot \frac{\beta t +x}{\overline\nu(\beta t+x)}\cdot \int_0^R  \underline{n}_{r} \big(x-y +\varDelta_{y+\delta}\big) \,dr\cr\cr
 \ar\leq\ar  C\cdot  \frac{\overline\nu(\beta t)}{\beta t} \cdot \frac{\beta t +x}{\overline\nu(\beta t+x)}\cdot \sup_{r\in[0,R]} \underline{n}_{r} \big(x-y +\varDelta_{y+\delta}\big)  . 
 \eeqnn
 In view of \eqref{eqn.312}-\eqref{eqn.4020} and then using  \eqref{eqn.30}, we have that for all sufficiently large $x$, 
 \beqnn
  \lefteqn{\sup_{r\in[0,R]}\underline{n}\big( \epsilon_r\in x-y +\varDelta_{y+\delta},\, \mathcal{J}^{\lambda (x-y)}\leq r,\,\zeta>r \big) }\cr\cr
  \ar\leq\ar\sup_{r\in[0,R]} \underline{n}\big( \#\mathcal{B}_r^{\lambda (x-y)},\, \epsilon_r\in x-y +\varDelta_{y+\delta},\, \zeta>r \big)
  \leq C\cdot \nu\big(x-y +\varDelta_{y+\delta}\big) 
  \leq C\cdot \frac{\overline\nu(x)}{x}.
 \eeqnn
On the other hand, similarly as in the proof of Proposition~\ref{Prop.501}, we also can prove that for all sufficiently large $x$, 
 \beqnn
 \sup_{r\in[0,R]}\underline{n}\big( \epsilon_r\in x-y +\varDelta_{y+\delta},\, \mathcal{J}^{\lambda (x-y)}> r,\,\zeta>r \big)
 \leq C\cdot  \nu\big(x-y +\varDelta_{y+\delta}\big)  \cdot \underline{n}\big(\overline\epsilon > \eta (x-y)\big). 
 \eeqnn
 Note $\underline{n}\big(\overline\epsilon  > \eta(x-y)\big)\to0$ as $x\to\infty$. Combining these two upper bounds gives that  
 \beqnn
 \sup_{x\geq M} I_3^{y,\delta}(t,R,x)  
 \ar\leq\ar  C\cdot  \frac{\overline\nu(\beta t)}{\beta t} \cdot  \sup_{x\geq M} \frac{\beta t +x}{\overline\nu(\beta t+x)}\cdot  \frac{\overline\nu(x)}{x}\cr
 \ar\leq\ar  C\cdot  \frac{\overline\nu(\beta t)}{\beta t} \cdot \Big( \sup_{x\in [M,t]} \frac{\beta t +x}{\overline\nu(\beta t+x)}\cdot  \frac{\overline\nu(x)}{x} +  \sup_{x>t} \frac{\beta t +x}{\overline\nu(\beta t+x)}\cdot  \frac{\overline\nu(x)}{x} \Big)\cr
 \ar\leq\ar  C\cdot  \frac{\overline\nu(\beta t)}{\beta t} \cdot \Big(   \frac{\beta t +t}{\overline\nu(\beta t+t)}\cdot  \frac{\overline\nu(M)}{M} +  \sup_{x>t} \frac{\beta t +x}{\overline\nu(\beta t+x)}\cdot  \frac{\overline\nu(x)}{x} \Big) 
 \eeqnn
  By $\overline{\nu} \in \mathrm{RV}^\infty_{-\alpha}$ and Potter's theorem for regularly varying functions; see e.g. Theorem 1.5.6 in \cite[p.25]{BinghamGoldieTeugels1987}, the last expression can be bounded by
  \beqnn
  C\cdot  \frac{\overline\nu(\beta t)}{\beta t} \cdot    \frac{\beta t +t}{\overline\nu(\beta t+t)}\cdot  \frac{\overline\nu(M)}{M} + C\cdot  \frac{\overline\nu(\beta t)}{\beta t}, 
  \eeqnn
  which goes to $0$ as $t\to\infty$ and then $M\to \infty$. Combining these estimates together, we obtain 
  \beqnn
  \limsup_{M\to\infty} \limsup_{t\to\infty}  \sup_{x\geq M} I_3^{y,\delta}(t,R,x)   =0. 
  \eeqnn
 Combining this, \eqref{eqn.4022} together with \eqref{eqn.4023} yields   $\sup_{x\geq 0}\big|  I_3^{y,\delta}(t,R,x)- \widetilde{I}_3^{y,\delta}(R,x) \big| \to0$ as $t\to\infty$. 
 This completes the proof.
 \qed

 \textbf{\textit{Proof of Lemma~\ref{Lemma.604}}} Recall the notation $I_4^{y,\delta}(t,R,x)$ defined in the proof of Corollary~\ref{Main.Corollary02}.   For $t>2R$, we may decompose $I_4^{y,\delta}(t,R,x)$ into the following two parts:
 \beqlb
 I_{41}^{y,\delta}(t,R,x)\ar:=\ar\frac{\beta t +x}{\overline\nu(\beta t+x)}\int_R^{t/2} ds\int_0^y \underline{n}_{t-s}\big( \big((x-y+z)^+, (x-y+z+\delta)^+\big] \big) \, \overline{n}_s(dz) , \label{eqn.4024}\\
 I_{42}^{y,\delta}(t,R,x)\ar:=\ar\frac{\beta t +x}{\overline\nu(\beta t+x)}\int_{t/2}^{t-R} ds\int_0^y \underline{n}_{t-s}\big( \big((x-y+z)^+, (x-y+z+\delta)^+\big] \big) \, \overline{n}_s(dz)  .  \label{eqn.4025}
 \eeqlb
 By Theorem~\ref{MainThm01}, there exists a constant $C>0$ such that for all sufficiently large $t$ and  all $x\geq0$,
 \beqnn
  \sup_{s\in[R,t/2]} \sup_{z\in [0,y]}\underline{n}_{t-s}\big( \big((x-y+z)^+, (x-y+z+\delta)^+\big] \big) 
  \leq C \cdot\frac{\overline\nu(\beta t+x)}{\beta t+x} . 
 \eeqnn
 Substituting this into the right-hand side of \eqref{eqn.4024}, then using \eqref{eqn.447},  
 \beqlb\label{eqn.4026}
  \limsup_{R\to\infty} \limsup_{t\to\infty} \sup_{x\geq 0} I_{41}^{y,\delta}(t,R,x) 
  \leq C \cdot   \limsup_{R\to\infty} \int_R^\infty \overline{n}\big(\epsilon_s\in(0,y],\,\zeta >s\big)\,ds   =0. 
 \eeqlb
 Applying the change of variables $r=t-s$ to  $I_{42}^{y,\delta}(t,R,x)$ gives 
 \beqnn
 I_{42}^{y,\delta}(t,R,x)
 \ar=\ar \frac{\beta t +x}{\overline\nu(\beta t+x)}\int_{R}^{t/2} dr\int_0^y \underline{n}_r\big( \big((x-y+z)^+, (x-y+z+\delta)^+\big] \big) \, \overline{n}_{t-r}(dz) \cr
 \ar\leq\ar \frac{\beta t +x}{\overline\nu(\beta t+x)}\int_{R}^{t/2} \sup_{z\in[0,y]} \underline{n}_r\big( \big((x-y+z)^+, (x-y+z+\delta)^+\big] \big) \cdot \overline{n}_{t-r}\big([0,y]\big)\, dr . 
 \eeqnn
 Using Theorem~\ref{MainThm01} again yields that   for all sufficiently large $r>R$ and  all $x\geq0$,
 \beqnn
  \sup_{z\in[0,y]} \underline{n}_r\big( \big((x-y+z)^+, (x-y+z+\delta)^+\big] \big)
  \leq C\cdot \frac{\overline\nu(\beta r+x)}{\beta r+x}. 
 \eeqnn
 In addition,  by Lemma~\ref{Lem5.4} we also have for all sufficiently large $t$ and $r\in [R,t/2]$,
 \beqnn
 \overline{n}_{t-r}\big([0,y]\big) \leq C\cdot \frac{\overline\nu(\beta(t-r))}{t-r}. 
 \eeqnn
 Combining these estimates together yields that for all sufficiently large $R$, $t>2R$ and  all $x\geq0$,
 \beqnn
 \sup_{x\geq 0} I_{42}^{y,\delta}(t,R,x)
  \ar\leq\ar C\cdot  \sup_{x\geq 0} \frac{\beta t +x}{\overline\nu(\beta t+x)}\int_{R}^{t/2}   \frac{\overline\nu(\beta r+x)}{\beta r+x} \cdot  \frac{\overline\nu(\beta(t-r))}{t-r}  \, dr \cr
  \ar\leq\ar C\cdot  \sup_{x\in[0,t]} \frac{\beta t +x}{\overline\nu(\beta t+x)}\int_{R}^{t/2}   \frac{\overline\nu(\beta r+x)}{\beta r+x} \cdot  \frac{\overline\nu(\beta(t-r))}{t-r}  \, dr\cr
  \ar\ar +C\cdot  \sup_{x\geq t} \frac{\beta t +x}{\overline\nu(\beta t+x)}\int_{R}^{t/2}   \frac{\overline\nu(\beta r+x)}{\beta r+x} \cdot  \frac{\overline\nu(\beta(t-r))}{t-r}  \, dr\cr
  \ar\leq\ar C\cdot \Big( \int_R^\infty \frac{\overline{\nu}(\beta r)}{\beta r}\,dr + \int_R^{t/2} \frac{\overline\nu(\beta(t-r))}{t-r}  \, dr \Big) \cr
  \ar\leq \ar C\cdot \Big( \int_R^\infty \frac{\overline{\nu}(\beta r)}{\beta r}\,dr + \int_{t/2}^\infty \frac{\overline\nu(\beta r)}{r}  \, dr \Big),
 \eeqnn
 which goes to $0$ as $t\to\infty$ and then $R\to\infty$. 
 Combining this together with \eqref{eqn.4026} yields the desired uniform limit. 
 This completes the proof.
 \qed

 \bibliographystyle{plain}
 
 \bibliography{Reference}

 \end{document}